\documentclass[a4paper, 11pt, reqno]{amsart}

\usepackage{amssymb}
\usepackage{amsfonts}
\usepackage{amsmath}
\usepackage{bbm}
\usepackage[margin=2.2cm]{geometry}
\usepackage{hyperref}
\usepackage{enumitem,verbatim}
\usepackage{bm}
\usepackage{hyperref}
\usepackage{tikz}
\usepgflibrary{decorations.markings}
\usetikzlibrary{calc,arrows,decorations.markings}

\usetikzlibrary{calc,math}

\tikzmath{\a1=30;\e1=0;\e2=-5;\e3=-8;\e4=-1;\e5=-1;\e6=0;\e7=-2;\e8=-5;
          \a2=10;\f1=1;\f2=0;\f3=-1;\f4=-1;\f5=1;\f6=2;\f7=0;\f8=0;
          \a3=-30;\g1=0;\g2=2;\g3=2;\g4=1;\g5=-1;\g6=-1;\g7=0;\g8=1;
          \a4=-10;
          \a5=25;
          \a6=20;
          \h1=0;\h2=0;\h3=1;\h4=0;\h5=-1;\h6=0;\h7=-1;\h8=-1;
  \x1 = -1;  \x2 = 2;
  \r1 = 2; \r2 = 1;
  \R1 = 1; \R2 = 1;
  \q1 = \R2 / (\R1 + \R2);
  \q2 = \R1 / (\R1 + \R2);
  \k1=0;\k2=20;\k3=40;\k4=80;\k5=100;\k6=140;\k7=200;\k8=220;\k9=250;\k10=300;\l2=320;\l3=350;
}
          
\newtheorem{thm}{Theorem}[section]

\newtheorem{prop}[thm]{Proposition}
\newtheorem{lemma}[thm]{Lemma}

\newtheorem{claim}[thm]{Claim}
\newtheorem{conj}[thm]{Conjecture}

\newtheorem{defn}[thm]{Definition}
\newtheorem{question}[thm]{Question}

\theoremstyle{definition}
\newtheorem{definition}[thm]{Definition}

\renewcommand{\epsilon}{\varepsilon}
\numberwithin{thm}{section}
\numberwithin{equation}{section}

\theoremstyle{remark}

\newcommand{\ep}{\varepsilon}
\newcommand{\eps}{\varepsilon}

\newcommand{\al}{\alpha}
\newcommand{\de}{\delta}

\newcommand{\De}{\Delta}

\newcommand{\cS}{\mathcal{S}}

\newcommand{\Int}{\mathrm{Int}}
\newcommand{\Ext}{\mathrm{Ext}}
\newcommand{\Cen}{\mathrm{Cen}}

\renewcommand{\l}{\ensuremath{\ell}}

\newcommand{\Cor}{\mathrm{Cor}}

\newenvironment{claimproof}[1]{\par\noindent\underline{Proof:}\space#1}{\leavevmode\unskip\penalty9999 \hbox{}\nobreak\hfill\quad\hbox{$\blacksquare$}}

\newcommand{\COMMENT}[1]{}
\renewcommand{\COMMENT}{\footnote} 

\title{Proof of Koml\'os's conjecture on Hamiltonian subsets}
\author{Jaehoon Kim, Hong Liu, Maryam Sharifzadeh and Katherine Staden}

\thanks{J.K.\ was supported by ERC grant~306349; H.L.\ was supported by EPSRC grant~EP/K012045/1, ERC
	grant~306493 and the Leverhulme Trust
	Early Career Fellowship~ECF-2016-523; M.Sh.\ and K.S.\ were supported by ERC grant~306493.
}
\date{\today}

\begin{document}

\begin{abstract}
Koml\'os conjectured in 1981 that among all graphs with minimum degree at least $d$, the complete graph $K_{d+1}$ minimises the number of Hamiltonian subsets, where a subset of vertices is Hamiltonian if it contains a spanning cycle.
We prove this conjecture when $d$ is sufficiently large. 
In fact we prove a stronger result: for large $d$, any graph $G$ with average degree at least $d$ contains almost twice as many Hamiltonian subsets as $K_{d+1}$, unless $G$ is isomorphic to $K_{d+1}$ or a certain other graph which we specify.
\end{abstract}
\maketitle

\section{Introduction}
A cycle in a graph $G$ is \emph{Hamiltonian} if it spans the whole vertex set of $G$. The Hamiltonian cycle problem, i.e.~deciding whether a given graph contains a Hamiltonian cycle, is one of Karp's original NP-complete problems~\cite{karp}.
It has been studied in various directions by numerous researchers over the last 60 years. 
As a computationally hard problem, the \emph{extremal problem} of finding best-possible sufficient conditions that guarantee the existence of a Hamiltonian cycle is of great interest.
The classical theorem of Dirac~\cite{dirac}, giving a minimum degree condition, was one of the first such results.
There have been some exciting recent developments in this area, for example in finding optimal packings of Hamiltonian cycles~\cite{krivsam,kko,apps}, decompositions of graphs into edge-disjoint Hamiltonian cycles~\cite{nash-w,kelly}, 
and finding Hamiltonian cycles in hypergraphs (see the survey paper~\cite{Ham-survey} for more references). 

The \emph{enumerative problem} of counting the number of cycles in a given class of graphs is one of the oldest problems in graph theory.
In 1897, Ahrens~\cite{ahrens} proved that, for any graph $G$ on $n$ vertices, if we denote by $\nu(G)$ the number of cycles in $G$, then
$$
e(G)-n+1 \leq \nu(G) \leq 2^{e(G)-n+1}-1.
$$
Volkmann~\cite{volkmann} proved that, when $G$ has minimum degree $\delta$, then $\nu(G) \geq \delta(\delta+1)$.
Various authors have investigated the problem of maximising or minimising the number of cycles in other particular classes of graphs, e.g.~\cite{at,bce}.

Another direction of research is to count the number of Hamiltonian cycles in a graph $G$.
Suppose that $\delta(G) \geq |G|/2$ (and hence $G$ contains at least one Hamiltonian cycle, by Dirac's theorem).
Then there is a formula due to Cuckler and Kahn~\cite{ck1} which asymptotically determines the logarithm of the number of Hamiltonian cycles.

In this paper we consider a question which is both extremal and enumerative, namely minimising in a graph $G$ the number $c(G)$ of \emph{distinguishable} cycles, where two subgraphs $G_1,G_2$ of $G$ are distinguishable if $V(G_1) \neq V(G_2)$.
An equivalent formulation, and the one which we shall mainly use, is that $c(G)$ is the number of \emph{Hamiltonian subsets} of $G$.
Here, $A \subseteq V(G)$ is 
\emph{Hamiltonian} if $G[A]$ contains a Hamiltonian cycle, i.e.~$G$ contains a cycle whose vertex set is $A$.
Thus, in contrast to the Hamiltonian cycle problem where one is interested in the collection of cycles which span a given set $V(G)$, we investigate the collection of sets which are spanned by a cycle.

As Dirac's theorem relates the existence of a Hamiltonian cycle to minimum degree, a very natural question is to ask how $c(G)$ relates to minimum degree; that is, minimising $c(G)$ given $\delta(G) \geq d$ for some integer $d$.
Here, the number of vertices $n$ of $G$ is not fixed.
Now, every subset of size at least three is a candidate for a Hamiltonian subset, and there are almost $2^n$ of these.
Thus, as $n$ increases, the number of candidates increases, and one expects that $c(G)$ will also increase.
Therefore it is natural to conjecture that, given the stipulation $\delta(G) \geq d$, to minimise $c(G)$ one should minimise $n$.
Clearly the unique graph $G$ of minimum order with $\delta(G) \geq d$ is the complete graph $K_{d+1}$.
This is the substance of a 1981 conjecture of Koml\'os (see~\cite{tuza2,tuza1,tuza3}):

\begin{conj}[see~\cite{tuza2}]\label{komlos}
For all integers $d > 0$ and all graphs $G$ with $\delta(G) \geq d$, we have
$$
c(G) \geq c(K_{d+1}).
$$
\end{conj}

Since every subset of size at least three is Hamiltonian in a complete graph, we have
\begin{equation}\label{complete}
c(K_{d+1}) = 2^{d+1} - \binom{d+1}{2} - d - 2.
\end{equation}

In this paper, we prove Koml\'os's conjecture for all sufficiently large $d$.
In fact, we prove a stronger result which replaces minimum degree with average degree (where the average degree $d(G)$ of a graph $G$ satisfies $d(G)|G|=2e(G)$), and shows that the extremal graph is stable in a rather precise sense.
The current best bound on $c(G)$ is due to Tuza~\cite{tuza2}.
\begin{thm}[\cite{tuza2}]\label{Tuza}
For all $d \geq 3$ and every graph $G$ with $d(G) \geq d$, we have that $c(G) \geq 2^{d/2}$.
\end{thm}


\subsection{New results}

The purpose of this paper is to prove Koml\'os's conjecture for all large $d$.
As mentioned above, we prove a stronger result.
To state it, we need a definition.
Given any positive integer $d$, let $K_{d+1} \ast K_d$ denote the graph obtained by taking vertex-disjoint copies of $K_{d+1}$ and $K_d$, and identifying them at a single vertex (see Figure~\ref{figext}(i)).
Note that $K_{d+1} \ast K_{d}$ has average degree exactly $d$.
Moreover,
$$
c(K_{d+1}\ast K_d) = c(K_{d+1}) + c(K_d) = \frac{3}{2} \cdot 2^{d+1} - d^2 - 2d - 3.
$$

Our main result is the following.

\begin{thm}\label{main}
For all $0< \alpha \leq 1$ there exists $d_0>0$ such that for all integers $d \geq d_0$, the following holds. Let $G$ be a graph with average degree at least $d$ which is not isomorphic to $K_{d+1}$ or $K_{d+1}\ast K_d$.
Then $c(G) \geq (2-\alpha) 2^{d+1}$.
\end{thm}

This theorem is best possible in the sense that the constant $2$ cannot be improved.
Indeed, there are many graphs showing that $2$ is best possible.
Figure~\ref{figext}(iii) shows four graphs $G$ with average degree at least $d$ such that $c(G) = (2-o(1))2^{d+1}$; and any $(d+2)$-vertex graph $G$ with average degree at least $d$ also satisfies $c(G) = (2-o(1))2^{d+1}$.
Also the statement is not necessarily true when $d$ is not an integer.
The graph $G$ in Figure~\ref{figext}(ii) has average degree slightly less than $d+\frac{1}{2}$ but again satisfies 
$c(G) = (2-o(1))2^{d+1} < (2-o(1))2^{d(G)+1}$; and $K_{d+2}-e$ has average degree slightly less than $d+1$ but again satisfies $c(G) = (2-o(1))2^{d+1} < (2-o(1))2^{d(K_{d+2}-e)+1}$.

A further remark is that Theorem~\ref{main} is \emph{not true} for $d=2$.
Indeed, every graph with average degree at least two contains a cycle, while $c(C_n)=1$ for all $n$.
So there are infinitely many graphs which minimise $c(G)$.
(Note that this does not contradict Conjecture~\ref{komlos}.)
Although we are not aware of a similar occurence for any $d \geq 3$, it would be interesting to determine the minimum $d_0$ one can take in the statement of Theorem~\ref{main}.

\noindent
\begin{figure}
\scalebox{0.65}{
\begin{tikzpicture}

\draw (-7,2.5) -- (17.7,2.5);
\draw (3.3,2.5) -- (3.3,5.5);

\begin{scope}[xshift=-1.5cm,yshift=4cm]
\begin{scope}
\draw (\x1,\R1) arc (90:270:\R1) cos(\x1*\q1+\x2*\q2, 0);
\begin{scope}[yscale=-1]
\draw (\x1,\R1) arc (90:270:\R1) cos(\x1*\q1+\x2*\q2, 0);
\end{scope}
\end{scope}

\node[draw=none] at (-5.2,0) {\large$(i)$};
\node[draw=none] at (5.8,0) {\large$(ii)$};
\node[draw=none] at (-5.2,-4) {\large$(iii)$};

\node[draw=none] at (-1,0) {\large$K_{d+1}$};
\node[draw=none] at (1.5,0) {\large$K_{d}$};
\node[draw,shape=circle,fill=black,inner sep=2pt] at (0.5,0) {};

\begin{scope}[rotate=180,xshift=-0.9cm,scale=0.8]
\draw (\x1,\R1) arc (90:270:\R1) cos(\x1*\q1+\x2*\q2, 0);
\begin{scope}[yscale=-1]
\draw (\x1,\R1) arc (90:270:\R1) cos(\x1*\q1+\x2*\q2, 0);
\end{scope}
\end{scope}

\end{scope}

\begin{scope}[xshift=10cm,yshift=4cm]
\begin{scope}
\draw (\x1,\R1) arc (90:270:\R1) cos(\x1*\q1+\x2*\q2, 0);
\begin{scope}[yscale=-1]
\draw (\x1,\R1) arc (90:270:\R1) cos(\x1*\q1+\x2*\q2, 0);
\end{scope}
\end{scope}

\node[draw=none] at (-1,0) {\large$K_{d+1}$};
\node[draw=none] at (1.7,0) {\large$K_{d+1}$};
\node[draw,shape=circle,fill=black,inner sep=2pt] at (0.45,0) {};

\begin{scope}[rotate=180,xshift=-0.9cm,scale=1]
\draw (\x1,\R1) arc (90:270:\R1) cos(\x1*\q1+\x2*\q2, 0);
\begin{scope}[yscale=-1]
\draw (\x1,\R1) arc (90:270:\R1) cos(\x1*\q1+\x2*\q2, 0);
\end{scope}
\end{scope}

\end{scope}




\begin{scope}[xshift=-5cm]
\draw (0,0) circle [radius=1];
\draw (2.1,0) circle [radius=1];

\node[draw=none] at (0,0) {\large$K_{d+1}$};
\node[draw=none] at (2.1,0) {\large$K_{d+1}$};
\end{scope}

\begin{scope}[xshift=1cm]
\begin{scope}[rotate=-90]
\draw (\x1,\R1) arc (90:270:\R1) cos(\x1*\q1+\x2*\q2, 0);
\begin{scope}[yscale=-1]
\draw (\x1,\R1) arc (90:270:\R1) cos(\x1*\q1+\x2*\q2, 0);
\end{scope}
\end{scope}

\begin{scope}[xshift=0.35cm,yshift=-0.67cm]
\begin{scope}[rotate=-210,scale=0.8]
\draw (\x1,\R1) arc (90:270:\R1) cos(\x1*\q1+\x2*\q2, 0);
\begin{scope}[yscale=-1]
\draw (\x1,\R1) arc (90:270:\R1) cos(\x1*\q1+\x2*\q2, 0);
\end{scope}
\end{scope}
\end{scope}

\begin{scope}[xshift=-0.35cm,yshift=-0.67cm]
\begin{scope}[rotate=30,scale=0.8]
\draw (\x1,\R1) arc (90:270:\R1) cos(\x1*\q1+\x2*\q2, 0);
\begin{scope}[yscale=-1]
\draw (\x1,\R1) arc (90:270:\R1) cos(\x1*\q1+\x2*\q2, 0);
\end{scope}
\end{scope}
\end{scope}

\node[draw=none] at (0,1) {\large$K_{d+1}$};
\node[draw=none] at (-1,-1) {\large$K_{d}$};
\node[draw=none] at (1,-1) {\large$K_{d}$};
\node[draw,shape=circle,fill=black,inner sep=2pt] at (0,-0.45) {};
\end{scope}

\begin{scope}[xshift=6cm]

\begin{scope}[rotate=0]
\draw (\x1,\R1) arc (90:270:\R1) cos(\x1*\q1+\x2*\q2, 0);
\begin{scope}[yscale=-1]
\draw (\x1,\R1) arc (90:270:\R1) cos(\x1*\q1+\x2*\q2, 0);
\end{scope}
\end{scope}

\begin{scope}[xshift=0.5cm]
\draw (0,0) .. controls (0.5,1.1) and (1.5,1.1) .. (2,0);
\draw (0,0) .. controls (0.5,-1.1) and (1.5,-1.1) .. (2,0);
\end{scope}

\begin{scope}[xshift=2.9cm]
\begin{scope}[rotate=-180,scale=0.8]
\draw (\x1,\R1) arc (90:270:\R1) cos(\x1*\q1+\x2*\q2, 0);
\begin{scope}[yscale=-1]
\draw (\x1,\R1) arc (90:270:\R1) cos(\x1*\q1+\x2*\q2, 0);
\end{scope}
\end{scope}
\end{scope}

\node[draw=none] at (-1,0) {\large$K_{d+1}$};
\node[draw=none] at (1.5,0) {\large$K_{d}$};
\node[draw=none] at (3.5,0) {\large$K_{d}$};
\node[draw,shape=circle,fill=black,inner sep=2pt] at (0.5,0) {};
\node[draw,shape=circle,fill=black,inner sep=2pt] at (2.5,0) {};

\end{scope}

\begin{scope}[xshift=13cm]

\begin{scope}[rotate=0,scale=0.8]
\draw (\x1,\R1) arc (90:270:\R1) cos(\x1*\q1+\x2*\q2, 0);
\begin{scope}[yscale=-1]
\draw (\x1,\R1) arc (90:270:\R1) cos(\x1*\q1+\x2*\q2, 0);
\end{scope}
\end{scope}

\begin{scope}[xshift=0.5cm]
\draw (0,0) .. controls (0.5,1.3) and (1.7,1.3) .. (2.2,0);
\draw (0,0) .. controls (0.5,-1.3) and (1.7,-1.3) .. (2.2,0);
\end{scope}

\begin{scope}[xshift=3.1cm]
\begin{scope}[rotate=-180,scale=0.8]
\draw (\x1,\R1) arc (90:270:\R1) cos(\x1*\q1+\x2*\q2, 0);
\begin{scope}[yscale=-1]
\draw (\x1,\R1) arc (90:270:\R1) cos(\x1*\q1+\x2*\q2, 0);
\end{scope}
\end{scope}
\end{scope}

\node[draw=none] at (-0.8,0) {\large$K_{d}$};
\node[draw=none] at (1.5,0) {\large$K_{d+1}$};
\node[draw=none] at (3.7,0) {\large$K_{d}$};
\node[draw,shape=circle,fill=black,inner sep=2pt] at (0.5,0) {};
\node[draw,shape=circle,fill=black,inner sep=2pt] at (2.7,0) {};

\end{scope}
\end{tikzpicture}}
\caption{(i) the graph $K_{d+1} \ast K_d$; (ii) a graph $G$ with $d(G) \approx d + \frac{1}{2}$ and $c(G) = (2-o(1))2^{d+1}$; (iii) four graphs $G$ with $d(G) = d$ and $c(G) = (2-o(1))2^{d+1}$.}
\label{figext}
\end{figure}
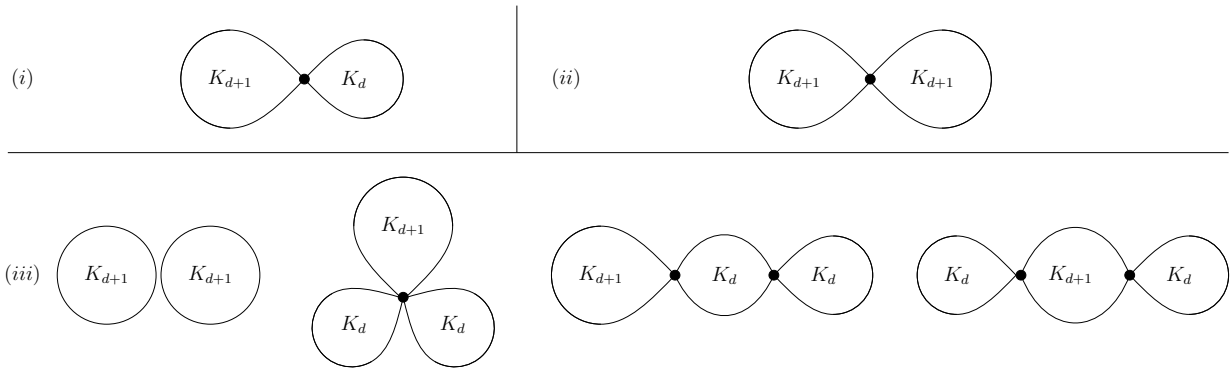

Our techniques also imply for large $d$ the following conjecture of Tuza~\cite{tuza3}, a bipartite analogue of Conjecture~\ref{komlos}.
Since the proof is very similar to the proof of Theorem~\ref{main}, we sketch the details in Section~\ref{conclusion}.

\begin{thm}\label{bip}
There exists $d_0 > 0$ such that, for all integers $d > d_0$ and all bipartite graphs $G$ with minimum degree at least $d$, we have $c(G) \geq c(K_{d,d})$.
\end{thm}

The graphs $G$ in Conjecture~\ref{komlos} could have any number $n \geq d+1$ of vertices, and consequently $G$ could have any given density.
This makes the problem difficult to attack as many available tools depend on the density of the graph.
The notions of pseudorandomness and expansion have both played a major role in recent advances in the Hamiltonian cycle problem.
In particular, the Regularity-Blow-up method of Koml\'os, S\'ark\"ozy and Szemer\'edi~\cite{KSSblowup}, which involves partitioning $G$ into pseudorandom subgraphs, has been the key tool in the solution of many important dense graph problems.
For problems involving sparser graphs, the concept of 
`sparse expansion', introduced by Koml\'os and Szemer\'edi~\cite{K-Sz-1,K-Sz-2} has proved very effective
(see also~\cite{BLSh,L-M,Richard} for some recent results in which such expanders play a role).
A novel aspect of our proof is to combine both approaches; depending, roughly speaking, on whether our graph $G$ is dense or sparse.

The main ingredient in the sparse case is the following general theorem about expander graphs, which we state here as it may be of independent interest. Roughly speaking, it states that an expander graph $G$ which is almost $d$-regular and not too sparse contains a set $Z$ of order $200d$ such that for every half-sized subset $U$ of $Z$, there is a cycle in $G$ whose intersection with $Z$ contains almost every vertex in $U$ and no vertices in $Z\setminus U$.

\begin{thm}\label{lem-build-sub}
For given $0<\epsilon_1 \leq 1$ and $L\geq1$, there exist $d_0$ and $K_0$ such that the following holds
for any $d\geq d_0, K\geq K_0$ and $n\in \mathbb{N}$ with $\log^{100}n \leq d\leq n/K$.
If $H$ is an $n$-vertex $(\ep_1, d/30)$-expander with $d/10\le \de(H)\le \De(H)\le Ld$, then $V(H)$ contains a set $Z$ of size $200d$ such that, for every subset $U\subseteq Z$ of size $100d$, there exists a cycle $C_U$ with $V(C_U)\cap Z \in \binom{U}{\geq 98d}$.
\end{thm}

We defer the definition of an expander graph to Section~\ref{sec-main-sparse} and the other notation in the statement to Section~\ref{sec-prelim}.

\subsection{Related research}

Considering a different local condition to that of minimum degree, Tuza~\cite{tuza2} proved that, if every edge of a graph $G$ lies in $t$ copies of $K_3$, then $c(G) \geq c(K_{t+2})$.
In~\cite{tuza2,tuza1,tuza3}, he also considered the problem of, for a given graph $F$ and class $\mathcal{G}$ of graphs, minimising the number of distinguishable subdivisions of $F$ in a graph $G \in \mathcal{G}$.
Note that this is a generalisation since a cycle is a subdivision of $K_3$.
The case when $\mathcal{G}$ is the set of graphs $G$ in which every edge lies in $t$ copies of $F$ is addressed in~\cite{tuzapaper}.
Considering a global condition, Knor~\cite{knor} provided estimates on the number of Hamiltonian subsets in $k$-connected graphs $G$ for various values of $k$, in terms of the number of vertices of $G$.

Given our resolution of Koml\'os's conjecture for large $d$ and the fact that the graphs which minimise $c(G)$ are small, it makes sense to ask: what is the minimum value of $c(G)$ over all \emph{$n$-vertex} graphs $G$ with minimum (or average) degree $d$?
A related question of Perkins~\cite{will} is minimising the normalised parameter $c(G)/|G|$, which penalises small graphs.

\begin{question}
What is $f(d) := \displaystyle\liminf_{n \rightarrow \infty}\left\lbrace \frac{c(G)}{n} : \delta(G)\geq d \text{ and } |G|=n \right\rbrace$?
\end{question}

For $d=2$ we have that $c(G) \geq 1$ with equality if and only if $G=C_n$ is a cycle.
Thus $f(2)=0$.
For general integers $d$, 
a simple lower bound for $f(d)$ can be obtained as follows.
Given any graph $G$ with $\delta(G) \geq d$ on $n$ vertices, let $F \subseteq G$ be a forest such that each component spans a component of $G$.
Then, for any $e \in E(G)\setminus E(F)$, the graph $F \cup \lbrace e \rbrace$ contains precisely one cycle, which necessarily contains $e$.
So $c(G)/|G| \geq (e(G)-e(F))/|G| \geq d/2-1$.
For an upper bound, when $(d+1) | n$, the union $G$ of $\frac{n}{d+1}$ vertex-disjoint $(d+1)$-cliques has $c(G)=\frac{n}{d+1}\cdot c(K_{d+1})$. Thus
$$
\frac{d}{2} - 1 \leq f(d) \leq \frac{c(K_{d+1})}{d+1}.
$$
The upper bound is not tight.
This can be seen by taking the vertex-disjoint union of $(d+1)$-cliques and a cycle $C$ of appropriate length, and adding an edge between each clique and $C$ to ensure $\delta(G) = d$.

What about the problem of \emph{maximising} the number of (distinguishable) cycles? For a fixed number of vertices $n$, clearly $K_n$ contains the most.
If instead we fix maximum degree $\Delta$, then $\nu(G)$ and $c(G)$ both increase as the number of vertices in $G$ increases.
To avoid such trivialities, Tuza asked a more restrictive question: what is the maximum number $\nu^{\mathrm {ind}}(G)$ of \emph{induced} cycles in an $n$-vertex graph $G$? 
Note that any two induced cycles in a graph are distinguishable.
This conjecture was verified in a strong sense by Morrison and Scott~\cite{tashscott}, who showed that there is a unique graph attaining the maximum for every sufficiently large $n$ (which depends on the value of $n \mod 3$).

Another way of avoiding such trivialities is to ask for the maximum number of cycles in an $n$-vertex graph $G$ in a restrictive class of graphs.
Recently, Arman, Gunderson and Tsaturian~\cite{agt} showed that $\nu(K_{\lfloor n/2\rfloor, \lceil n/2\rceil})$ is maximum among all $n$-vertex triangle-free graphs, for $n \geq 141$.

\subsection{Organisation of the paper}

The remainder of the paper is organised as follows.
In Section~\ref{sec: sketch}, we sketch the proof of Theorem~\ref{main}.
Our notation and some tools needed for the proof are listed in Section~\ref{sec-prelim}.
Sections~\ref{dense} and~\ref{sec-main-sparse} form the bulk of proof and deal with the `dense' and `sparse' cases respectively (see Section~\ref{sec: sketch} for explanations of these terms).
In Section~\ref{end}, we combine our auxiliary results to complete the proof of Theorem~\ref{main}.
Finally, in Section~\ref{conclusion}, we sketch how to prove Theorem~\ref{bip}.

\section{Sketch of the proof of Theorem~\ref{main}}\label{sec: sketch}

Suppose that $G$ is a counterexample to Theorem~\ref{main} whose order $n$ is minimal for a fixed large number $d$.
So $d(G) \geq d$ and $G \notin \lbrace K_{d+1},K_d \ast K_{d+1} \rbrace$.
One can show that the minimality of $G$ implies that it has minimum degree at least $d/2$, and any subgraph has average degree at most $d$.
From now on we distinguish two cases according to whether $n/d$ is bounded (the \emph{dense case}) or unbounded (the \emph{sparse case}).
Each case uses different techniques and methods and are essentially separate.

\subsection{The dense case: $n/d$ is bounded (Section~\ref{dense}).}
If $n < 1.19d$ then a simple probabilistic argument proves Theorem~\ref{main} (see Lemma~\ref{lem: small graph}).
Indeed, a Chernoff bound implies that almost every subset of $G$ is Hamiltonian, via the classical degree sequence theorem of P\'osa~\cite{posathm}.
Note that $1.19$ is somewhat arbitrary here.

Thus we may assume that $1.19d \leq n \leq Kd$ for some (large) constant $K$ which does not depend on $d$.
Since we are assuming that $d$ is large, $G$ is a large dense graph, and we can employ the celebrated Regularity-Blow-up method pioneered by Koml\'os, S\'ark\"ozy and Szemer\'edi~\cite{KSSblowup}.

The regularity lemma of Szemer\'edi~\cite{szem} implies that $G$ has a partition into equally-sized clusters $V_1,\ldots,V_r$ of vertices such that almost all ordered pairs of clusters induce a pseudorandom subgraph of $G$; and an exceptional set $V_0$ of size $o(n)$.
The `reduced graph' $R$ of $G$ has vertices $1,\ldots,r$, where $ij$ with $1 \leq i < j \leq r$ is an edge if $G[V_i,V_j]$ is pseudorandom and dense. 
The idea here is to show that $R$ contains some special structure $Q$ with the property that
\begin{itemize}
\item[($\dagger$)] there are at least $(2-o(1))2^{d+1}$ choices of $\lbrace V_i' : i \in V(Q) \rbrace$ where $V_i' \subseteq V_i$ for all $i$, such that $\bigcup_{i \in V(Q)}V_i'$ is a Hamiltonian subset.
\end{itemize}
Clearly this will prove Theorem~\ref{main} in this case.
Let $m := |V_1| = \ldots = |V_r|$ and let
\begin{equation}\label{sketchd'}
d' := (1-o(1))\frac{dr}{n} = (1-o(1))\frac{d}{m}.
\end{equation}
There are three possible structures $Q$ we can guarantee in $R$:
\begin{itemize}
\item[(1)] two vertex-disjoint cycles whose orders sum to at least $1.8d'$;
\item[(2)] a path of length $(1+\frac{1}{100})d'$;
\item[(3)] a large `sun' with large `corona' (the definitions of which we defer to Section~\ref{sec: suns}).
\end{itemize}
The reason we can find such a structure $Q$ is as follows.
It is well-known that the reduced graph $R$ inherits many of the properties of $G$.
In particular, the average degree and minimum degree of $R$ are closely related to those of $G$.
Thus
$R$ has average degree at least $d'$ and minimum degree at least $d'/2$.
Moreover, $r \geq 1.18d'$ by (\ref{sketchd'}).
This is enough to find $Q$ in $R$ satisfying either (1), (2) or (3) (see Lemma~\ref{lem: sun}).
Note that this is not always possible if $n$ is not bounded away from $d$, since then $r$ may not be bounded away from $d'$. This is the reason we consider the `very dense' case that $n<1.19d$ separately.

To see how to guarantee ($\dagger$), for the purposes of simplicity let us assume that $Q$ satisfies (1), and further that $Q$ has a subgraph $C$ which is a cycle of length $c \geq 1.2d'$ (the remaining cases are similar but more involved, see Lemma~\ref{lem: 1.2 to K}).
Write $C := 1\ldots c$.
To show ($\dagger$), we will use a variant of the Blow-up lemma~\cite{KKOT16,KSSblowup}.
First, we do some pre-processing.
Standard tools allow us to remove $o(|V_i|)$ vertices from each $V_i$ with $i \in [c]$ to obtain new equally-sized clusters $U_1,\ldots,U_c$, such that $G[U_i,U_{i+1}]$ is still dense and pseudorandom (where addition is modulo $c$), and also has large minimum degree.
We say that $G' := \bigcup_{i \in [c]}G[U_i,U_{i+1}]$ is `super-regular with respect to $C$'.
Roughly speaking, a special case of the Blow-up lemma states that, for the purposes of embedding a (spanning) bounded degree subgraph, any graph $H$ which is super-regular with respect to $C$ behaves as if each $H[U_i,U_{i+1}]$ were \emph{complete}.

For each $i \in [c]$, let $V_i' \subseteq U_i$ with $|V_i'| = m/2$ be arbitrary.
Then one can show that, with high probability, the graph $H := \bigcup_{i \in [c]}G[V_i',V_{i+1}']$ is super-regular with respect to $C$.
The Blow-up lemma tells us that, if the complete blow-up of the cycle $C_c$ with parts of size $m/2$ contains a spanning cycle, then so does $H$.
But this is clearly seen to be true (find the cycle by winding round parts).
Thus almost every such $V_1',\ldots,V_c'$ is such that $\bigcup_{i \in [c]}V_i'$ is a Hamiltonian subset.
Moreover, the number of such choices is
$$
(1-o(1)) \cdot \prod_{i \in [c]}\binom{|U_i|}{m/2} \geq (1-o(1))\binom{(1-o(1))m}{m/2}^{1.2d'} \geq 2^{1.1md'} \stackrel{(\ref{sketchd'})}{>} (2-o(1))2^{d+1}.
$$
So ($\dagger$) holds, as required.

\subsection{The sparse case: $n/d$ is unbounded (Section~\ref{sec-main-sparse}).}
For the sparse case, we will work with expander graphs.
This notion of expansion was first introduced by Koml\'os and Szemer\'edi~\cite{K-Sz-1}.
Roughly speaking, a graph is an \emph{expander} if every set that is not too large or small has large external neighbourhood (see Definition~\ref{defn-expander}).
The main property (Lemma~\ref{diameter}) of expander graphs which we require is that
\begin{itemize}
\item[($\star$)] if $H$ is an expander graph, then between every pair of large sets, there is a short path; and this path can be chosen to avoid an arbitrary small set.
\end{itemize}

Our aim in Section~\ref{sec-main-sparse} is to prove the following (see Lemma~\ref{lem-sparse}).
\begin{itemize}
	\item[($\ddagger$)] Given large $K > 0$, when $d$ is sufficiently large and $H$ is an almost $d$-regular expander graph on at least $Kd$ vertices, then $c(H) \geq 2^{50d}$.
\end{itemize}
Here, almost $d$-regular means that there are constants $0<\ell_1<\ell_2$ such that $\ell_1d \leq d_H(x) \leq \ell_2d$ for all $x \in V(H)$.
Koml\'os and Szemer\'edi~\cite{K-Sz-1} proved that every graph $G$ contains an expander which has almost the same average degree as $G$ (see Lemma~\ref{lem-expander}).
Then, roughly speaking, ($\ddagger$) is used in our proof to ensure that any expander subgraph of $G$ has at most $Kd$ vertices, so we are in the dense case.

The proof of ($\ddagger$) has the following general structure.
We will locate a special set $Z \subseteq V(H)$ of size $200d$ with the following property (see Theorem~\ref{lem-build-sub} and Lemma~\ref{lemma-sparse-sub}). For every subset $U\subseteq Z$ of size $|Z|/2$, we can find a cycle $C_U$ whose intersection with $Z$ is almost the whole of $U$. This then implies that a large fraction of the ${|Z|\choose |Z|/2}$ cycles $C_U$ are distinguishable. 

To construct such a set $Z$, we will use different strategies depending on the edge density of $H$. In Section~\ref{sec-dense}, we deal with the case when $H$ is relatively dense. In this case, a key structure in our construction is a `web' (See Definition~\ref{defn-web} and Figure~\ref{figweb}), which guarantees many vertex-disjoint paths within a relatively small neighbourhood of a `core' vertex. This structure is inspired by some recent work on clique subdivisions (see~\cite{L-M} and~\cite{Richard}). We will iteratively construct many webs that are almost pairwise disjoint (in fact, they have disjoint `interiors' (Lemma~\ref{lem-web})), and let $Z$ be the set of core vertices of these webs. Then for each $|Z|/2$-set $U$ in $Z$, to construct $C_U$, we will connect the webs corresponding to $U$ via paths through their `exteriors' in a cyclic manner. 
We hope to find vertex-disjoint (short) paths between the (large) exteriors of webs, avoiding previously-built paths, which together with the paths inside the webs leading to their core vertices form the desired cycle $C_U$.
Property ($\star$) is vital in obtaining these paths.

However, such `short' paths can still block all webs, making it impossible to integrate their core vertices into the cycle $C_U$. To overcome this, we will choose our paths in a more careful way, such that we avoid using too many vertices inside any particular web. Then the fact that the webs chosen are almost disjoint enables us to incorporate most of the vertices in $U$ into $C_U$.

In Section~\ref{sec-sparse}, we deal with the case when $H$ is very sparse. Here, an obvious difficulty in using the previous approach is that a single `short' path could use the vertices in all webs due to the fact that the graph has very few edges. Instead, we will choose a set of vertices which are pairwise far apart in $H$ to serve as $Z$. Such a set of vertices exists since $H$ has small maximum degree compared to its order. To find the cycle $C_U$ (see Lemma~\ref{lemma-sparse-sub}), we grow the neighbourhood around the vertices in $U$ to a reasonably large size so that we can connect them via a path which avoids used vertices.

\subsection{Finishing the proof (Section~\ref{end})}

There are several difficulties with the above approach.
Statement ($\ddagger$) implies that any \emph{almost $d$-regular} expander subgraph $H$ of $G$ must be dense.
To find such an expander we must first remove large degree vertices from $G$ and then apply Lemma~\ref{lem-expander} to find $H$.
But then the average degree $d(H)$ of $H$ could be less than $d$, in which case we may have $c(H) < c(K_{d+1})$ (for example if $H = K_d$).
In the case when $G$ is $2$-connected, we circumvent this problem by finding two vertex-disjoint almost $d$-regular expander subgraphs $H_1,H_2$ of $G$ and, for every fixed $x_i,y_i \in V(H_i)$, finding many $x_i,y_i$-paths $P_i$ in $H_i$.
The pairs $(P_1,P_2)$ give rise to distinguishable cycles.
(The general case is somewhat more technical and involves considering the so-called `block-structure' of $G$ and finding cycles and paths inside maximal 2-connected subgraphs of $G$.)

\section{Notation and preliminaries}\label{sec-prelim}

\subsection{Notation}
For an integer $N$, let $[N]:=\{1,\dots, N\}$. Given a set $X$ and $k \in \mathbb{N}$, let $\binom{X}{k}$ be the collection of $k$-element subsets of $X$ and let $\binom{X}{\leq k}$ be the collection of subsets of $X$ with at most $k$ elements.
Define $\binom{X}{\geq k}$ analogously.

Given a graph $G$, we write $V(G)$ and $E(G)$ for its vertex and edge set respectively, and let $e(G) := |E(G)|$. Sometimes, we identify $G$ with $V(G)$ by writing $x\in G$ instead of $x\in V(G)$ and $|G|$ instead of $|V(G)|$. For any $X \subseteq V(G)$, let $G-X$ be the graph obtained from $G$ by removing the vertices of $X$ and any edges incident to them. If $x \in V(G)$ we abbreviate $G-\lbrace x \rbrace$ to $G-x$.
Given graphs $G_1,G_2$, we write $G_1 \cup G_2$ to denote a graph with $V(G_1\cup G_2)= V(G_1)\cup V(G_2)$ and $E(G_1 \cup G_2) = E(G_1)\cup E(G_2)$.
Define the \emph{neighbourhood} $N_G(X)$ and \emph{external neighbourhood} $\Gamma_G(X)$ of $X$ by setting
\begin{align*}
N_G(X) := \lbrace u \in V(G) : uv \in E(G) \mbox{ for some } v \in X \rbrace \quad\text{and}\quad\Gamma_G(X):=N_G(X) \setminus X.
\end{align*}
Given $X \subseteq V(G)$ and $x \in V(G)$, write $d_G(x,X) := |N_G(x) \cap X|$. The \emph{degree} of $x \in V(G)$ is $d_G(x) := d_G(x,V(G))$.
Let $\delta(G)$ be the minimum degree of $G$ and let $\delta_i(G) := d_i$ where $d_1 \leq d_2 \leq \dots \leq d_n$ is the degree sequence of the graph $G$. Write $\Delta(G)$ for the maximum degree and $d(G) := 2e(G)/|G|$ for the \emph{average degree} of $G$. 

We say that $W=(x_1,\dots, x_m)$ is a \emph{walk} in a graph $G$ if $x_ix_{i+1}\in E(G)$ for all $i\in [m-1]$; and a \emph{circuit} if additionally $x_mx_1 \in E(G)$. For a given walk $W=(x_1,\dots, x_m)$ and vertex $x$, we let ${\rm deg}(x,W)= |\{i \in [m]: x=x_i\}|$.
We say that a walk $P=(x_1,\dots, x_m)$ is a \emph{path} if $x_i\neq x_j$ for all $i\neq j \in [m]$ and we define $\Int(P):= \{x_2,\dots, x_{m-1}\}$ to be the \emph{interior} of $P$. We say that a path $P=(x_1,\dots, x_m)$ is a \emph{$u,v$-path} if $\{u,v\}=\{x_1,x_m\}$. 
The \emph{length} of a path is the number of vertices it contains.
We say a graph $G$ \emph{admits a vertex partition $(R,V_1,\dots, V_r)$} if $\{V_i: i\in [r]\}$ forms a partition of $V(G)$ into independent sets; $R$ is a graph on vertex set $[r]$; and $G[V_i,V_j]$ is an empty graph for all $ij\notin E(R)$.

Unless otherwise specified, we write $\log(\cdot)$ for the natural logarithm $\log_e(\cdot)$.
For $a,b,c\in \mathbb{R}$ we write $a = b\pm c$ if $b-c \leq a \leq b+c$. In order to simplify the presentation, we omit floors and ceilings and treat large numbers as integers whenever this does not affect the argument. The constants in the hierarchies used to state our results have to be chosen from right to left. More precisely, if we claim that a result holds whenever $0 <  a \ll b \leq 1$, then this means that there is a non-decreasing
function $f : (0, 1] \rightarrow (0, 1]$ such that the result holds for all $0 < a, b\leq1 $ with $a \leq f(b)$. Hierarchies with more constants are defined in a similar way.

Additional notation will be defined as and when it is required.

\subsection{Tools}
In this section we state some tools that we will use to prove our main result. 
For $n\in \mathbb{N}$ and $0\leq p\leq 1$ we write $\mathrm{Bin}(n,p)$ to denote the binomial distribution with parameters $n$ and $p$. 
The hypergeometric random variable $X$ with parameters $(n,m,k)$ is
defined as follows. We let $N$ be a set of size $n$, fix $S \subseteq N$ of size
$|S|=m$, pick uniformly at random a set $T \subseteq N$ of size $|T|=k$,
then define $X=|T \cap S|$. Note that $\mathbb{E}X = km/n$.
We will use the following standard Chernoff-type bound (see e.g.~Theorem 2.10 in~\cite{JLR}).

\begin{prop}\label{Chernoff Bounds}
	Suppose $X$ has binomial or hypergeometric distribution and $0<a<3\mathbb{E}[X]/2$. Then
	$\mathbb{P}(|X - \mathbb{E}[X]| \ge a) \le 2 e^{-\frac{a^2}{3\mathbb{E}[X]}}$.
\end{prop}

The next result is well-known and guarantees that every graph has a subgraph with large average and minimum degrees.

\begin{prop}\label{lem: min deg subgraph}
	Every graph $G$ with average degree $d$ contains a subgraph $H$ with $d(H)\geq d$ and $\delta(H)\geq d/2$.
\end{prop}

We will also use the classical theorem of Erd\H{o}s and Gallai from 1959 which gives a best-possible condition on the minimum length of a cycle in a graph with a given number of vertices and edges (we state a slightly weaker version here).

\begin{thm}[\cite{erdgal}]\label{lem: cycle av deg}
	For all $d \geq 2$, every graph $G$ with $d(G) \geq d$ contains a cycle of length at least $d$. 
\end{thm}

We will need the following theorem of P\'osa, which gives a sufficient condition on the degree sequence of a graph for the presence of a Hamiltonian cycle. (Note that there are several strengthenings of this result, but the version below suffices for our purposes.)

\begin{thm}[\cite{posathm}]\label{posa}
	Let $G$ be a graph on $n \geq 3$ vertices with degree sequence $d_1 \leq d_2 \leq \dots \leq d_n$.
	If $d_i \geq i+1$ for all $i < (n-1)/2$ and if additionally $d_{\lceil n/2 \rceil} \geq \lceil n/2 \rceil$ when $n$ is odd, then $G$ contains a Hamiltonian cycle.
\end{thm}

We will need the following easy bounds on certain binomial coefficients.

\begin{prop}\label{binom bounds}
Suppose that $0 < 1/n \ll 1$ where $n \in \mathbb{N}$. Then 
	$$
	\binom{n}{\lfloor n/4-1  \rfloor} \geq 2^{4n/5}\quad \text{and} \quad \binom{n}{2\lfloor n/4-1 \rfloor} \geq  2^{n - \log_2{n}}.
	$$
\end{prop}
\begin{proof}
	Stirling's formula implies that, for all $m \geq 1$ and $k \geq 2$, we have
	$$
	\sqrt{m}\binom{km}{m} \geq \frac{k^{k(m-1)+1}}{(k-1)^{(k-1)(m-1)}}.
	$$
	So, setting $m := \lfloor  n/4-1 \rfloor$, since $1/n\ll 1$, we have
	$$
	\binom{n}{\lfloor  n/4-1  \rfloor} \geq \binom{4m}{m} \geq \frac{4^{4m-3}}{\sqrt{m} \cdot 3^{3m-3}} = 2^{(8-3\log_2 3)m - 6+3\log_2 3 - (\log_2 m)/2} \geq 2^{3.24 m} \geq 2^{4n/5}.
	$$
	Now set $m' := 2\lfloor  n/4-1  \rfloor$. Then
	$$
	\binom{n}{2\lfloor  n/4-1  \rfloor} \geq \binom{2m'}{m'} \geq 2^{2m'-1-(\log_2 m')/2} \geq 2^{n-\log_2 n},
	$$
	as required.
\end{proof}



\section{The dense case}\label{dense}

Our main aim in this section is, very roughly speaking, to prove Theorem~\ref{main} in the case when $G$ is an $n$-vertex graph with average degree at least $d$ and large minimum degree which is \emph{dense} in the sense that $d$ is linear in $n$.
The two main results of this section are the following lemmas.
Given distinct vertices $x,y\in V(G)$, denote by $p_{xy}(G)$ the number of vertex subsets $U$ containing $\lbrace x,y\rbrace$ such that $G[U]$ contains a spanning $x,y$-path, which is precisely the number of distinguishable $x,y$-paths.

\begin{lemma}\label{lem: small graph}
Suppose $0< 1/d \ll \alpha \leq 1$ and that $n \in \mathbb{N}$ satisfies $d+1\leq n\leq 1.2 d$. 
If $G$ is an $n$-vertex graph with $d(G)\geq d$ and $\delta(G)\geq d/2$, 
then
\begin{itemize}
\item[(i)] $c(G)\geq (1-\alpha/2) \cdot 2^{n}$;
\item[(ii)] if $x,y \in V(G)$ are distinct, then $p_{xy}(G)\geq (1-\alpha/2) \cdot 2^{n-2}$.
\end{itemize} 
\end{lemma}

\begin{lemma}\label{lem: 1.2 to K}
Suppose $0< 1/d \ll 1/K \leq 1$ and that $n\in \mathbb{N}$ satisfies $1.19d \leq n \leq K d$.
Let $G$ be a $2$-connected $n$-vertex graph with $d(G) \geq d$ and $\delta_2(G)\geq d/2$. 
If $x,y$ are two distinct vertices of $G$, then
\begin{eqnarray*}
c(G)> 2^{(1+ \frac{1}{200})d}\ \text{ and }\ p_{xy}(G) > 2^{0.89d}.
\end{eqnarray*}
\end{lemma}

Lemma~\ref{lem: small graph} addresses the case when $G$ is almost complete, i.e.~the number of vertices is at most $1.2d$.
Its proof uses Theorem~\ref{posa} and a probabilistic argument and appears in Section~\ref{proof: small graph}.

Lemma~\ref{lem: 1.2 to K} addresses the case when our graph $G$ on $n$ vertices has average degree linear in $n$ but also bounded away from $n$. In this case, we apply the Regularity lemma and use Lemma~\ref{lem: sun} to find in the reduced graph either (i) two vertex-disjoint cycles $C_1,C_2$; (ii) a sun $S$ (see Section~\ref{sec: suns} for the definition); or (iii) a path $P$; the clusters of which, in the original graph $G$, span at least $(1+ 1/100)d$ vertices.  Then, using a probabilistic argument, we show that at least $2^{(1+ 1/200)d}$ subsets of the vertices lying in clusters of $C_1 \cup C_2$, $S$ or $P$ span a cycle.
We use similar arguments to show that $p_{xy}(G)$ is large for any distinct $x,y \in V(G)$.
The proof requires some more sophisticated tools and auxiliary results (Sections~\ref{sec: suns} to~\ref{sec: reg}), so we defer its proof to Section~\ref{proof: 1.2 to K}.

\subsection{The proof of Lemma~\ref{lem: small graph}}\label{proof: small graph}

We do not need any additional tools to prove our first main lemma, and so we proceed immediately with the proof.

\medskip
\noindent
\emph{Proof of Lemma~\ref{lem: small graph}.}
Note that $\sum_{v\in V(G)} d_{G}(v) \geq dn$ and $d_{G}(v)\geq d/2$ for all $v\in V(G)$. Let 
$$U := \{ v \in V(G): d_{G}(v)\geq 2d/3\}.$$
Then we get 
$$
dn \leq \sum_{v\in V(G)} d_{G}(v) \leq n |U| +   2d(n-|U|)/3,
$$
and so, using the fact that $n \leq 1.2 d$,
\begin{eqnarray*}
|U| \geq \frac{dn/3}{(n-2d/3)} \geq \frac{5n}{8}.
\end{eqnarray*}

To prove (i), choose a set $V'\subseteq V(G)$ uniformly at random by including each $v \in V'$ with probability $1/2$, independently of all other choices. Then for every $v\in V(G)$ we have
\begin{eqnarray*}
\mathbb{E}[|V'|] = n/2;\quad\mathbb{E}[|V'\cap U|] = |U|/2;\quad\text{and}\quad\mathbb{E}[ d_G(v,V')] = d_{G}(v)/2.
\end{eqnarray*}
Consider the following properties:
\begin{itemize}
\item[(A1)] $n':=|V'| = n/2 \pm d^{2/3}$;
\item[(A2)] $|U\cap V'| \geq \frac{5n}{16} - d^{2/3}$;
\item[(A3)] for all $v\in V(G)$ we have $d_{G}(v,V') \geq d_{G}(v)/2-d^{2/3}$.
\end{itemize}
Then Proposition~\ref{Chernoff Bounds} implies that the probability that, when $1/d\ll \alpha$, all of (A1)--(A3) hold is at least
\begin{eqnarray*}
1- \left( 2e^{-\frac{2d^{4/3}}{3n}}+2e^{-\frac{16d^{4/3}}{15n}} +\sum_{v\in V(G)}2e^{-\frac{2d^{4/3}}{3d_G(v)}}\right) &\geq&  1- e^{-d^{1/4}} \geq  1-\frac{\alpha}{2},
\end{eqnarray*}
where the penultimate inequality follows from $d_G(v)<n\le 1.2d$ for all $v$.

Now, for a given set $V'$ of size $n'$ which satisfies (A1)--(A3), let $d_1\leq \dots \leq d_{n'}$ be the degree sequence of the induced graph $G[V']$. We claim that this degree sequence satisfies the conditions of Theorem~\ref{posa}, and hence $G[V']$ is Hamiltonian.
To see this, it is sufficient to find an integer $k$ with $2 \leq k < n'/2$ such that $d_1 \geq k$ and $d_k > n'/2$.
We claim that $k:= 19n'/48$ suffices.

Indeed, (A1), (A3) and the fact that $\de(G)\ge d/2$, $n\le 1.2d$ and $1/d\ll 1$ imply that 
\begin{eqnarray*}
d_1 \geq \frac{d}{4} - d^{2/3} \geq \frac{5n}{24}-d^{2/3} \geq \frac{5n'}{12}-2d^{2/3} > k.
\end{eqnarray*}
Also, (A1) and (A2) imply that $|U\cap V'| \geq 5n'/8 - 2d^{2/3}$, and (A3) together with the definition of $U$ implies that, for all $v \in U\cap V'$, we have 
$$d_{G}(v,V') \geq \frac{d}{3} - d^{2/3} \geq \frac{5n'}{9}-2d^{2/3} > \frac{n'}{2}.$$ 
Thus at least 
$$|U\cap V'|\ge \frac{5n'}{8}-2d^{2/3} > \frac{29n'}{48}+1=n' - k+1$$ 
vertices in $V'$ have degree at least $n'/2$, and so $d_{k} > n'/2$, as required.

We have shown that, for at least $(1-\alpha/2) \cdot 2^{n}$ distinct subsets $V'\subseteq V(G)$, we have that $G[V']$ is Hamiltonian.
Thus $c(G)\geq (1-\alpha/2) \cdot 2^{n}$, proving (i).

For part (ii), the argument is very similar.
Choose a set $V'\subseteq V(G)$ uniformly at random among all sets $V'$ containing both $x$ and $y$ (i.e.~include each of $x,y$ with probability one and every other vertex $v$ with probability $1/2$, independently of all other choices).
A near identical argument shows that (A1)--(A3) hold with probability at least $1-\alpha/2$.
Now obtain $G'$ from $G[V']$ by adding a new vertex $z$ and the edges $xz,yz$.
We now have that 
$$
d_1 = d_{G'}(z) = 2; \quad d_2 \geq k; \quad \text{and}\quad d_k > n'/2.
$$
Theorem~\ref{posa} implies that $G'$ contains a Hamiltonian cycle $C$.
Since $d_{G'}(z) = 2$, $C$ contains the edges $xz,yz$.
Thus $C$ without these edges is a path with endpoints $x,y$ which spans $V'$.
There are $2^{n-2}$ subsets of $V(G)$ containing both $x$ and $y$; so 
for at least $(1-\alpha/2) \cdot 2^{n-2}$ choices of $V'$, we have such a path.
Thus $p_{xy}(G)\geq (1-\alpha/2) \cdot 2^{n-2}$, proving (ii).
\hfill$\square$

\subsection{Suns, paths and cycles}\label{sec: suns}

Let $a,b \in \mathbb{N}$ and consider a sequence $I = (i_1,\ldots,i_b)$ with $1\leq i_1 < i_2 <\dots < i_b \leq a$ such that $i_{j} - i_{j-1} \geq 2$ and $i_1 +a-i_b \geq 2$. 
Define a graph $S$ as follows. Let $V(S) := \{x_1,\dots, x_a\} \cup \{ y_i : i \in I\}$, $x_{a+1}:=x_1$ and $x_{0} := x_a$.
Let
$$
E(S):= \{ x_i x_{i+1} : i\in [a] \} \cup \{ x_{i-1}y_{i}, y_{i} x_{i+1} : i\in I\}.
$$
We call such a graph $S$ an {\em ($a$,$b$)-sun}.
Note that by definition $b\le a/2$. The set of vertices remaining after deleting $\lbrace x_i : i \in I \rbrace$ from $S$ induce a cycle in $S$; and similarly after deleting $\lbrace y_i: i \in I \rbrace$.
We call $\Cor(S) := \lbrace x_i,y_i : i \in I \rbrace$ the \emph{corona} of $S$.
See Figure~\ref{figsun} for an illustration of a sun.

\begin{figure}
\tikzstyle{every node}=[circle, draw, fill=black,
                        inner sep=0pt, minimum width=4pt]
\begin{tikzpicture}[scale=0.6]
    \draw \foreach \x in {0,10,...,350}
    {
        (\x:4) node {}  -- (\x+10:4)
    };

    \draw (\k1:4) -- (\k1+10:5) node[label=\k1:$y_{i_1}$] {} -- (\k1+20:4);
    \draw (\k2:4) -- (\k2+10:5) node[label=\k2:$y_{i_2}$] {} -- (\k2+20:4);
    \draw (\k3:4) -- (\k3+10:5) node[label=\k3:$y_{i_3}$] {} -- (\k3+20:4);
    \draw (\k4:4) -- (\k4+10:5) node[label=\k4:$y_{i_4}$] {} -- (\k4+20:4);
    \draw (\k5:4) -- (\k5+10:5) node[label=\k5:$y_{i_5}$] {} -- (\k5+20:4);
    \draw (\k6:4) -- (\k6+10:5) node[label=\k6:$y_{i_6}$] {} -- (\k6+20:4);
    \draw (\k7:4) -- (\k7+10:5) node[label=\k7:$y_{i_7}$] {} -- (\k7+20:4);
    \draw (\k8:4) -- (\k8+10:5) node[label=\k8:$y_{i_8}$] {} -- (\k8+20:4);
    \draw (\k9:4) -- (\k9+10:5) node[label=\k9:$y_{i_9}$] {} -- (\k9+20:4);
    
    \node[label=(\k1+180):$x_{i_1}$] at (\k1+10:4) {};
    \node[label=(\k2+180):$x_{i_2}$] at (\k2+10:4) {};
    \node[label=(\k3+180):$x_{i_3}$] at (\k3+10:4) {};
    \node[label=(\k4+190):$x_{i_4}$] at (\k4+10:4) {};
    \node[label=(\k5+180):$x_{i_5}$] at (\k5+10:4) {};
    \node[label=(\k6+180):$x_{i_6}$] at (\k6+10:4) {};
    \node[label=(\k7+180):$x_{i_7}$] at (\k7+10:4) {};
    \node[label=(\k8+180):$x_{i_8}$] at (\k8+10:4) {};
    \node[label=(\k9+180):$x_{i_9}$] at (\k9+10:4) {};
\end{tikzpicture}
\caption{A $(36,9)$-sun with labelled corona.}
\label{figsun}
\end{figure}

In the next lemma, we show that any graph $G$ with average degree at least $d$ and minimum degree at least $d/2$, and whose vertex set has size at least slightly larger than $d$, contains either (i) two vertex-disjoint cycles, the sum of whose lengths is large; (ii) a long path; or (iii) a large sun with large corona.
Later we will use Lemma~\ref{lem: sun} to find these structures in a reduced graph obtained after applying the Regularity lemma.

\begin{lemma}\label{lem: sun}
	Let $n\geq 1.18 d$ and let $G$ be an $n$-vertex graph with $d(G)\geq d$ and $\delta(G) \geq d/2$.
	Then $G$ contains at least one of the following:
	\begin{enumerate}
		\item[(i)] \label{item-sub1} two vertex-disjoint cycles $C_1,C_2$ with $|C_1|+|C_2| \geq 1.8d$;
		\item[(ii)] \label{item-sub2} a path $P$ with $|P| \geq (1+ 1/100)d$;
		\item[(iii)] \label{item-sub3} an $(a,b)$ sun with $a\geq d$ and $b\geq d/20$.
	\end{enumerate}
\end{lemma}

\begin{proof}
	First, we show that either there are two disjoint cycles in $G$ satisfying (i), or there is a large dense component $H$ in $G$. In the latter case, we will show that $H$ contains the configuration in (ii) or (iii). 
	
	By averaging, we can find in $G$ at least one component $H$ with $d(H)\ge d(G)\ge d$. Let $n_1=|V(H)|$, $d_1=d(H)$, and $d_2=d(G\setminus H)$. Observe that we may assume $d_2<d$, since otherwise Theorem~\ref{lem: cycle av deg} guarantees two vertex-disjoint cycles, one in $H$ and one in $G\setminus H$, with total length at least $d_1+d_2\ge 2d$, implying~(i). We claim that we may additionally assume
	\begin{eqnarray}\label{eq-ratio}
	d\le d(G)\le d_1\le 1.01 d\quad\mbox{ and }\quad\frac{n}{n_1}<\frac{22}{21}.
	\end{eqnarray} 
	Indeed, if $d_1\ge 1.01d$, then by Theorem~\ref{lem: cycle av deg} there is a cycle of length at least $1.01d$, implying~(ii). To see why we may assume the second inequality in~\eqref{eq-ratio}, note that as $d-d_2>0$, we have
	\begin{eqnarray}\label{eq-ratioformula}
	n_1d_1+(n-n_1)d_2\ge nd\quad\Leftrightarrow\quad n_1(d_1-d_2)\ge n(d-d_2)\quad\Leftrightarrow\quad \frac{n}{n_1}\le\frac{d_1-d_2}{d-d_2}.
	\end{eqnarray}
	Now, if $n/n_1\ge 22/21$ then by~\eqref{eq-ratioformula}, we have
	\begin{eqnarray*}
		\frac{d_1-d_2}{d-d_2}\ge\frac{22}{21}\quad\Rightarrow\quad d_1+d_2\ge 22d-20d_1\overset{\eqref{eq-ratio}}{\ge}22d-20.2d=1.8d.
	\end{eqnarray*}
	By Theorem~\ref{lem: cycle av deg}, we can find in $H$ a cycle $C_1$ of length at least $d_1$ and in $G\setminus H$ a cycle $C_2$ of length at least $d_2$. The above inequality implies $|C_1|+|C_2|\ge d_1+d_2\ge 1.8d$, yielding~(i). 
	
	From now on, we will work in the connected $n_1$-vertex graph $H$ which has average degree $d_1$. Then~\eqref{eq-ratio} together with $n\ge 1.18d$ implies that
	\begin{eqnarray}\label{eq-H}
	\de(H)\ge\frac{d}{2}\ge\frac{d_1/1.01}{2}>0.495 d_1\quad\mbox{ and }\quad n_1>\frac{21}{22}n\ge \frac{21}{22}\cdot 1.18\cdot \frac{d_1}{1.01}\ge 1.115 d_1.
	\end{eqnarray} 
	
	Let $C = x_1\ldots x_a$ be a longest cycle in $H$ and set $x_{a+1} := x_1$. Theorem~\ref{lem: cycle av deg} implies that $a \geq d_1$. Again, $a \leq 1.01d_1$, since otherwise (ii) holds by the fact that $d_1\ge d$. Now let $U := V(H) \setminus V(C)$. Then, using~\eqref{eq-H},
	\begin{eqnarray}\label{eq-U}
    	|U|=n_1-a \ge 1.115d_1-1.01d_1\geq 0.105d_1.
	\end{eqnarray}

	Suppose that $d(H[U]) \geq d_1/100$.
	By Theorem~\ref{lem: cycle av deg}, $H[U]$ contains a cycle $C'$ with $|C'| \geq d_1/100$.
	Since $H$ is connected, there is a path between the vertex-disjoint cycles $C$ and $C'$. This path together with $C$ and $C'$ forms a path of length at least $|C|+|C'| \geq 1.01d_1$, and so (ii) holds.
	
	Thus we may assume that $d(H[U]) < d_1/100$. If there are at most $d_1/20$ vertices $u \in U$ such that $d_H(u,U) \leq d_1/40$, then, by~\eqref{eq-U}
	\begin{eqnarray*}
		d(H[U])|U| = \sum_{u \in U}d_H(u,U) \geq \left(|U|-\frac{d_1}{20}\right) \frac{d_1}{40} > \frac{d_1}{100}\cdot |U|,
	\end{eqnarray*}
	a contradiction. Thus, writing $b:= d_1/20$, we can find distinct vertices $z_1,\ldots,z_{b}$ in $U$ such that $d_H(z_j,U) \leq d_1/40$ for all $j \in [b]$.
	Since $\delta(H)> 0.495 d_1$, for every $j \in [b]$,
	\begin{align}\label{eq: zj C nbrs}
	d_H(z_j,C) \ge \delta(H)-d_H(z_j,U)>0.495d_1- \frac{d_1}{40}= 0.47d_1.
	\end{align} 
	We will now iteratively construct a sequence of distinct indices $i_1,\ldots,i_b$ in $[a]$ such that the following hold for all $j \in [b]$:
	\begin{itemize}
		\item[(S1)$_j$] $x_{i_j-1},x_{i_j+1} \in N_H(z_j)$;
		\item[(S2)$_j$] $|i_j-i_{\ell}| \not\equiv 0,1 \mod a$ for all $\ell \in [j-1]$.
	\end{itemize}
	Suppose that, for some $j \le b$, we have chosen distinct $i_1,\ldots,i_{j-1}$ in $[a]$ such that (S1)$_{j'}$ and (S2)$_{j'}$ hold for all $j' \in [j-1]$.
	We will show that we can choose an appropriate $i_j$.
	To this end, we use \eqref{eq: zj C nbrs} to choose $1 \leq k_1 < \ldots < k_{0.47d_1} \leq a$ so that $x_{k_i}$ is a neighbour of $z_j$ for all $i \in [0.47d_1]$.
	For each $t \in \mathbb{N}$, let $p_t$ be the number of indices $\ell \in [0.47d_1]$ such that $k_{\ell+1}-k_\ell = t$ if $\ell < 0.47d_1$; or $k_1-k_{0.47d_1} = t-a$. Notice first that $p_1 = 0$, otherwise $z_j$ has a pair of consecutive neighbours in $C$, contradicting to the maximality of $C$. By definition, $p_2$ is at most the number of indices $i \in [a]$ for which $x_{i-1},x_{i+1} \in N_H(z_j)$, i.e.~which satisfy (S1)$_j$. Then, since $\sum_{t}p_t=0.47d_1$,
	\begin{eqnarray*}
		1.01d_1 \geq a = \sum_{t \in \mathbb{N}}t\cdot p_t \ge 2p_2+3\sum_{t\ge 3}p_t\geq 2p_2 + 3(0.47d_1-p_2)= 1.41d_1-p_2,
	\end{eqnarray*}
	and so $p_2 \geq 0.4d_1$. Thus there are at least $0.4d_1$ indices $i$ in $[a]$ which satisfy (S1)$_j$.
	Now, the number of indices $i \in [a]$ which do not satisfy (S2)$_j$ is at most $3(j-1)$ (the indices $i_1,\ldots,i_{j-1}$ already obtained, as well as their neighbours on $C$).
	But 
	$$3(j-1)< 3b = \frac{3d_1}{20}< 0.4d_1 \leq p_2.$$
	Therefore, we can choose $i_j \in [a]$ such that (S1)$_{j'}$ and (S2)$_{j'}$ hold for all $j' \in [j]$ and obtain $i_1,\ldots,i_b$ as required. We have found an $(a,b)$-sun with vertex set $C \cup \lbrace z_1,\ldots,z_b \rbrace$ and corona $\lbrace x_{i_j} : j \in [b]  \rbrace \cup \lbrace z_1,\ldots,z_b \rbrace$.
	So (iii) holds, completing the proof.
\end{proof}


The next proposition shows that, for any cycle, sun or path $H$, there is a walk $W$ in $H$ such that the set of values ${\rm deg}(x,W)$ for $x\in V(H)$ is well-controlled.
Its proof is straightforward, we defer it to the appendix.

\begin{prop}\label{homomorphism}
Let $n,a,b\in \mathbb{N}$ where $b\le  a/2$.
Let $H$ be a graph and let $u,v \in V(H)$.\footnote{When not specified, $u$ and $v$ need not be distinct.} Suppose that $H$ is either
(i) a cycle $C$ on $a$ vertices;
(ii) a path $Q$ on $a$ vertices;
or (iii)  an $(a,b)$-sun $S$.
Then in each case, there exists a walk $W$ in $H$ from $u$ to $v$ such that
\begin{itemize}
\item[(i)] $N_2 = V(C)$.
\item[(ii)] $N_1$ is the set of two endpoints of $Q$ and $N_2$ is the set of internal vertices of $Q$.
\item[(iii)] $N_1 = \Cor(S)$ and $N_2 := V(S)\setminus \Cor(S)$;
\end{itemize}
where
$$
N_k := \lbrace x \in V(H) : {\rm deg}(x,W) \in \lbrace kn,kn+1,kn+2 \rbrace \rbrace.
$$
\end{prop}

\subsection{The Regularity and Blow-up lemmas}\label{sec: reg}

In our proof, we apply Szemer\'edi's Regularity lemma~\cite{szem}.
For a comprehensive survey of this lemma and its many applications, see~\cite{ksssurvey,kssurvey}.
To state the lemma we need some more definitions. We write $d_G(A,B)$ for 
the \emph{density} $\frac{e(G[A,B])}{\vert A \vert \vert B \vert}$ of a bipartite graph $G$ with vertex classes $A$ and $B$. 
Given $\epsilon > 0$, we say that $G$ is 
$\epsilon$\emph{-regular} if every $X \subseteq A$ and $Y \subseteq B$ with $\vert X \vert \geq \epsilon \vert A \vert$ 
and $\vert Y \vert \geq \epsilon \vert B \vert$ satisfy $\vert d(A,B) - d(X,Y) \vert \leq \epsilon$.  
Given $\epsilon, \gamma \in (0,1)$ we say that $G$ is \emph{$(\epsilon, \gamma)$-regular} if $G$ is 
$\epsilon$-regular and $d_G(A,B) \geq \gamma$. 
We say that $G$ is $(\epsilon , \gamma )$\emph{-super-regular} if both of the following hold:
\begin{itemize} 
\item $G$ is $(\epsilon,\gamma)$-regular;
\item $d_G(a) \geq \gamma|B|$ and $d_G(b) \geq \gamma|A|$ for all $a \in A$, $b \in B$.
\end{itemize}

Let $R$ be a graph on vertex set $[r]$. We say that a graph $H$ is \emph{$(\eps,\gamma)$-(super-)regular with respect to vertex partition $(R,V_1,\dots, V_r)$} if $\lbrace V_i : i \in V(R) \rbrace$ is a partition of $V(H)$; and $H[V_i,V_j]$ is $(\eps,\gamma)$-(super-)regular for all $ij \in E(R)$.

We will use the following degree form of the Regularity lemma, which can be easily derived from the usual version (see~\cite{kssurvey}).

\begin{lemma} \label{reg}
\emph{(Degree form of the Regularity lemma)} 
Suppose $0< 1/n \ll 1/M \ll \epsilon , 1/M' <1$ with $n\in \mathbb{N}$ and $\gamma \in \mathbb{R}$.
Then there is a partition of the vertex set of $G$ into $V_0, V_1, \ldots , V_{r}$ and a spanning subgraph $G^\prime$ of $G$ such that the following holds:
\begin{itemize}
\item[(i)] $M^\prime \leq r \leq M$;
\item[(ii)] $\vert V_0 \vert \leq \epsilon n$;
\item[(iii)] $\vert V_1 \vert = \ldots = \vert V_{r}\vert$;
\item[(iv)] $d_{G^\prime}(x) > d_{G}(x) - (\gamma + \epsilon) n$ for all $x \in V(G)$;
\item[(v)] for all $i \in [r]$ the graph $G^\prime[V_i]$ is empty;
\item[(vi)] for all $1 \leq i < j \leq r$, $G'[V_i , V_j]$ is $\epsilon$-regular and has density either 0 or at least $\gamma$.
\end{itemize}
\end{lemma}

We call $V_1 , \ldots , V_r$ \emph{clusters} and refer to $G^\prime$ as the \emph{pure graph}. The last condition of the lemma says that all pairs of clusters are $\epsilon$-regular (but possibly with different densities). The \emph{reduced graph} $R$ of $G$ with parameters $\epsilon$, $\gamma$ and $M^\prime$ is the graph whose vertex set is $[r]$ and in which $ij$ is an edge precisely when $G'[V_i , V_j]$ is $\epsilon$-regular and has density at least $\gamma$.

The following is a well-known tool which is similar to Proposition~8 in~\cite{maxplanar}, which states that a regular partition of a graph can be made into a super-regular partition by a slight modification.

\begin{lemma}[\cite{maxplanar}]\label{superslice}
Suppose $0 < 1/m \ll \epsilon \ll \gamma, 1/\Delta_R \leq 1$ and $1/m \ll 1/r$ with $r,m \in \mathbb{N}$.
Let $R$ be a graph with $V(R) = [r]$ and $\Delta(R)\leq \Delta_R$.
Let $G$ be a $(\epsilon,\gamma)$-regular graph with respect to vertex partition $(R,V_1,\dots, V_r)$ such that $|V_i|=(1 \pm \epsilon)m$ for all $1 \leq i \leq r$.
Then there exists $G'\subseteq G$ which is $(4\sqrt{\epsilon},\gamma/2)$-super-regular with respect to vertex partition $(R,V'_1,\dots, V'_r)$ such that $V'_i\subseteq V_i$ and $|V'_i| = (1- \sqrt{\epsilon})m$.
\end{lemma}

The following `slicing' lemma states that every induced subgraph of a regular pair which is not too small is still regular (with a weaker regularity parameter).

\begin{lemma}\label{slicing}
Let $\eps,\gamma \in \mathbb{R}$ be such that $\gamma > 2\eps > 0$.
Let $(A,B)$ be an $(\eps,\gamma)$-regular pair.
Suppose that $A' \subseteq A$ and $B' \subseteq B$ are such that $|A'| \geq \sqrt{\eps}|A|$ and $|B'| \geq \sqrt{\eps}|B|$.
Then $(A',B')$ is an $(\sqrt{\eps},\gamma/2)$-regular pair.
\end{lemma}

For technical reasons, we will use a modification of the celebrated Blow-up lemma of Koml\'os, S\'ark\"ozy and Szemer\'edi (see Remark 8 in \cite{KSSblowup}).
Indeed, this result does not allow the situation in which $1/r$ is much smaller than $\epsilon$, as we require.
Although Csaba proved a version which does allow this situation (see Lemma 5 in \cite{Csaba}), his result is not formulated to allow a non-equitable partition. Thus we use the following lemma from \cite{KKOT16} (a simplified version of Theorem 6 in \cite{KKOT16}).

The Blow-up lemma states that, for the purpose of embedding a (spanning) bounded degree graph $H$, a graph $G$ which is super-regular with respect to some graph $J$ behaves like a complete `blow-up' of $J$.
Moreover, if there are a small number of special vertices in $H$ which each have a large `target set' in $G$, then there is an embedding of $H$ that maps every special vertex into its associated target set.

\begin{thm}\label{blowup}
Suppose $0<1/n \ll \eps, \alpha \ll \gamma, 1/\Delta, 1/\Delta_R,c \leq 1$ and $1/n\ll 1/r$ where $r,n \in \mathbb{N}$. 
Suppose further that
\begin{itemize}
\item[(i)] $n \leq n_i \leq n + 2$, for all $i\in[r]$;
\item[(ii)] $R$ is a graph on vertex set $[r]$ with $\Delta(R)\leq \Delta_R$;
\item[(iii)] $G$ is $(\eps,\gamma)$-super-regular with respect to vertex partition $(R,V_1,\dots, V_r)$ such that $|V_i| = n_i$ for all $i \in [r]$;
\item[(iv)] $H$ is a graph which admits a vertex partition $(R,X_1,\dots, X_r)$ such that $\Delta(H) \leq \Delta$ and $|X_i|=n_i$.
\end{itemize}
Moreover, suppose that in each class $X_i$ there is a set of $\alpha n_i$ special vertices $y \in V(H)$, each of them equipped with a set $S_y \subseteq V_i$ with $|S_y| \geq cn_i$.
Then there is an embedding of $H$ into $G$ such that every special vertex $y$ is mapped to a vertex in $S_y$. 
\end{thm}

We will use Theorem~\ref{blowup} to prove the following lemma which is essentially the special case when $H$ is a path, except that we now prescribe the endpoints of the embedding of $H$ in $G$ (rather than prescribing a set of allowed endpoints) and some sets $V_i$ are twice the size of others.
Again, the proof is deferred to the appendix.

\begin{lemma}\label{blowuppath}
Suppose $0< 1/n \ll \epsilon \ll \gamma, 1/\Delta_R \leq 1$ and $1/n\ll 1/r$ where $r,n\in \mathbb{N}$. 
Suppose further that
\begin{itemize}
\item[(i)] $n_i \in \{n,n+1,n+2,2n,2n+1,2n+2\}$ for all $i\in [r]$;
\item[(ii)] $R$ is a graph on vertex set $[r]$ with $\Delta(R)\leq \Delta_R$;
\item[(iii)] $G$ is a graph and $x,y \in V(G)$ are two distinct vertices. Suppose further that $G$ is $(\eps,\gamma)$-super-regular with respect to vertex partition $(R,V_1,\dots, V_r)$, where $x \in V_j$ and $y \in V_k$; and $|V_i| = n_i$ for all $i \in [r]$.
\item[(iv)] $W = (w_1,\ldots, w_m)$ is a walk in $R$ from $j$ to $k$ such that ${\rm deg}(i,W)= n_i$ for all $i\in [r]$.
\end{itemize}
Then $G$ contains a spanning path $P$ with endpoints $x,y$.
\end{lemma}

\subsection{The proof of Lemma~\ref{lem: 1.2 to K}}\label{proof: 1.2 to K}

Let $K\geq 1.19$ be given.
Choose integers $M,M', d_0$ and constants $\eps,\gamma$ such that 
$$0 < 1/d_0\ll 1/M \ll 1/M' \ll \epsilon \ll \gamma \ll 1/K.$$ Now fix $d \geq d_0$.
Let $G$ be a $2$-connected graph on $n$ vertices such that $1.19d \leq n \leq Kd$, $d(G)\geq d$ and $\delta_2(G) \geq d/2$.

Apply the Regularity lemma (Lemma~\ref{reg}) to $G$ with parameters $\epsilon,\gamma,M'$ to obtain clusters $V_1, \ldots, V_{r}$ of equal size, an exceptional set $V_0$ of size at most $\epsilon n$, and a pure graph $G'$ which is $(\epsilon,\gamma)$-regular with respect to a vertex partition $(R,V_1,\dots, V_r)$. The vertex set of $R$ is $[r]$, where $M' \leq r \leq M$. Let $m:=|V_1|=\ldots = |V_r|$ and $\xi:= d/n$. Then
\begin{align}\label{eq: xi m}
1/K \leq \xi \leq 100/119 \enspace \text{ and } \enspace (1-\eps)n \leq rm \leq n.
\end{align}

First we show that $R$ essentially inherits the minimum and average degree of $G$.
\begin{claim}\label{inherit}
Let $d' := (1-\sqrt{\gamma})\xi r$. Then
$$
r \geq 1.18d';\quad\delta(R)\geq d'/2;\quad\text{and}\quad d(R) \geq d'.
$$
\end{claim}

\begin{claimproof}
The first inequality is immediate from the definition of $d'$, that $\gamma \ll 1/K$ and that $\xi\le 100/119$.
For a contradiction to the second, suppose there is some $i\in [r]$ with $d_{R}(i) < d'/2$. 
Count the number of edges $e_i$ in $G'$ incident to $V_i$ as follows.
Since $\delta_2(G)\geq \xi n/2$, we have, using Lemma~\ref{reg}(iv) and~(v), that
\begin{align}\label{eq: ei lower}
e_i = \sum_{v\in V_i} d_{G'}(v) \geq (\xi/2-\epsilon-\gamma) n (m-1)  \geq (\xi/2-2\gamma)n m .
\end{align}
By assumption, less than $d'/2$ clusters $V_j$ are such that $(V_i,V_j)$ is an $(\epsilon,\gamma)$-regular pair in $G'$, and $G'[V_i,V_j]$ is empty for all other $V_j$.
So
$$
e_i < d' m ^2/2 \leq (1-\sqrt{\gamma})\xi n m /2,
$$
since $r m  \leq n$.
Together with (\ref{eq: ei lower}), this gives a contradiction because and $2\gamma < \xi\sqrt{\gamma}/2$. This proves that $\delta(R) \geq d'/2$.

For the final inequality, since $rm\le n\le Kd$ and $\epsilon\ll\gamma \ll 1/K$,
\begin{align*}
d(R) &= \frac{2 e(R)}{r} \geq \frac{ 2 e(G') }{rm^2} \geq \frac{ (d - (\gamma+\epsilon)n)n  }{r  m ^2}\ge \frac{d-2\gamma\cdot Kd}{m} \stackrel{\eqref{eq: xi m}}{\geq} \frac{ (1-2\gamma K)dr }{  n }\\ &\geq (1-\sqrt{\gamma})\frac{dr}{n} =d',
\end{align*}
as required.
\end{claimproof}

\medskip
\noindent
Claim~\ref{inherit} together with Lemma~\ref{lem: sun} implies that $R$ contains one of the following:
\begin{itemize}
\item[(1)] vertex-disjoint cycles $H,J$, where $|J|\leq |H|\leq 101d'/100$ and $|H|+|J| \geq 1.8d'$;
\item[(2)] a path $H$ with $|H| = (1+ 1/100)d'$;
\item[(3)] an $(a,b)$ sun $H$ with $d' \leq a < 101d'/100$ and $ d'/20\le b\le a/2$.
\end{itemize}

(If the upper bounds in (1) and (3) do not hold, then we have an instance of (2).)
Observe that, in all cases, $\Delta(H) \leq 4$ and $|H|\le 2d'$. Throughout the rest of this proof, denote
\begin{equation}\label{m'}
m' := (1-\sqrt{\eps})m \quad \text{and} \quad m_0 := \lfloor m'/4-1 \rfloor.
\end{equation}
Lemma~\ref{superslice} with $H$ and $G'$ playing the roles of $R$ and $G$ respectively implies that, for each $i \in V(H)$, $V_i$ contains a subset $V_i'$ of size $m'$ such that for every edge $ij$ of $H$, the graph $G'[V_i',V_j']$ is $(4\sqrt{\eps},\gamma/2)$-super-regular.
Let $V' := \bigcup_{i \in V(H)}V_i'$.
Let $x',y' \in V'$ be distinct vertices and $i_1,i_2 \in V(H)$ be such that $x'\in V_{i_1}'$ and $y'\in V_{i_2}'$. Apply Proposition~\ref{homomorphism} with $m_0,i_1$ and $i_2$ playing the roles of $n,u$ and $v$ respectively to obtain a partition $N_1 \cup N_2$ of $V(H)$ such that there is a walk $W= (x_1,\dots, x_\ell)$ in $H$ such that $x_1=i_1$ and $x_\ell=i_2$, where for each $k\in [2]$,
\begin{equation}\label{Nk}
N_k := \left\lbrace i \in [r] : {\rm deg}(i,W) \in \lbrace km_0,km_0+1,km_0+2 \rbrace \right\rbrace.
\end{equation}
Let $n_i := {\rm deg}(i,W)$ for all $i \in V(H)$. Note that $n_i\in \{m_0,m_0+1, m_0+2, 2m_0, 2m_0+1, 2m_0+2\}$.
So a crude but useful estimate is that, for all $i \in V(H)$,
\begin{equation}\label{ni}
n_i \geq m'/5.
\end{equation}

\begin{claim}\label{Ncount}
If $4|N_1|/5+|N_2| \geq \gamma d'$, then
$$
p_{x'y'}(G'[V']) \geq 2^{(1-2\sqrt{\gamma})(4|N_1|/5+|N_2|)d/d'}.
$$
\end{claim}

\begin{claimproof}
Suppose that $X \subseteq V(G)$ is such that 
\begin{itemize}
\item[(i)] $X = \bigcup_{i \in V(H)}X_i$ where $X_i \subseteq V_i'$ and $|X_i|=n_i$;
\item[(ii)] $x' \in X_{i_1}$ and $y' \in X_{i_2}$;
\item[(iii)] $G'[X_i,X_j]$ is $(\eps^{1/3},\gamma/3)$-super-regular for all $ij \in E(H)$.
\end{itemize}
Then Lemma~\ref{blowuppath} applied with the following graphs and parameters implies that $G'[X]$ contains a spanning path with endpoints $x'$ and $y'$: \newline

{
\begin{tabular}{c|c|c|c|c|c|c|c|c|c|c|c}
object/parameter &  $H$ & $G'[X]$ & $X_i$ & $W$ & $\eps^{1/3}$ & $\gamma/3$ & $x'$ & $y'$ & $i_1$ & $i_2$ & $m_0$ \\ \hline
playing the role of & $R$ & $G$ & $V_i$ & $W$ & $\eps$ & $\gamma$ & $x$ & $y$ & $j$ & $k$ & $n$. 
\\ 
\end{tabular}
}\newline\vspace{0.2cm}

Thus $p_{x'y'}(G'[V'])$ is at least the number of subsets $X$ of $V(G)$ satisfying (i)--(iii). 
We claim that at least half of the sets $X$ satisfying the first two properties also satisfy the third.
Indeed, for each $i \in V(H)$, independently choose $X_i \subseteq V_i'$ uniformly at random among all subsets $X_i$ of $V_i'$ such that $x \in X_{i_1}$ and $y \in X_{i_2}$ and $|X_i| = n_i$.
(So $X := \bigcup_{i \in V(H)}X_i$ satisfies (i) and (ii).)
Now fix $z \in V_i'$ and $j \in N_H(i)$.
Then, since $G'[V_i',V_j']$ is $(4\sqrt{\eps},\gamma/2)$-super-regular,
$$
n_j=|X_j|\ge\mathbb{E}[d_{G'}(z,X_j)] \geq d_{G'}(z,V_j')\cdot \frac{n_j-2}{m'}\ge \frac{\gamma m'}{2}\cdot \frac{n_j}{m'}-2 \geq \frac{5\gamma n_j}{12}.
$$
So Lemma~\ref{Chernoff Bounds} implies that
$$
\mathbb{P}\left[ d_{G'}(z,X_j) < \frac{\gamma n_j}{3}\right] \leq 2e^{-\frac{(\gamma n_j/12)^2}{3n_j}}\leq 2e^{-\gamma^3n_j}\stackrel{(\ref{ni})}{\leq} e^{-\sqrt{m'}}.
$$
So the probability that $d_{G'}(z,X_j) \geq \gamma n_j/3$ for every $z \in V'_i$ with $i\in V(H)$ and every $j\in N_{H}(i)$ is at least
$$
1-|H|\cdot m'\cdot \Delta(H)\cdot e^{-\sqrt{m'}} \geq 1-2n^2 \cdot 4\cdot e^{-(n/r)^{1/3}} \geq \frac{1}{2}
$$ 
since $1/n \ll 1/r$.
We have shown that at least half of the $X \subseteq V(G)$ satisfying both~(i) and~(ii) are such that $d_{G'}(x,X_j) \geq \gamma n_j/3$ for all $i \in V(H)$, $x \in X_i$ and $j \in N_H(i)$.
Call such a set $X$ \emph{good}.

Lemma~\ref{slicing} together with~\eqref{m'},~\eqref{ni} and that $G'[V_i,V_j]$ is $(\eps,\gamma)$-regular implies that for all good $X$ we have that $G'[X_i,X_j]$ is $(\sqrt{\eps},\gamma/2)$-regular for all $ij \in E(H)$, and hence $(\eps^{1/3},\gamma/3)$-regular.
Therefore, since $X$ is good, each such $G'[X_i,X_j]$ is $(\eps^{1/3},\gamma/3)$-super-regular.
Thus every good $X$ automatically satisfies (iii) and the number of $X$ satisfying (i)--(iii) is at least the number of good $X$.

When $i_1 = i_2 =: i^*$, this shows that
\begin{eqnarray*}
p_{x'y'}(G'[V']) &\geq& \frac{1}{2}\cdot \binom{m'-2}{n_{i^*}-2}\prod_{i \in V(H)\setminus \lbrace i^* \rbrace}\binom{m'}{n_i} = \frac{1}{2}\cdot  \frac{n_{i^*}(n_{i^*}-1)}{m'(m'-1)}\prod_{i \in V(H)}\binom{m'}{n_i}\\
&\stackrel{(\ref{ni})}{>}& \frac{1}{100}\prod_{i \in V(H)}\binom{m'}{n_i}.
\end{eqnarray*}
The same bound holds when $i_1 \neq i_2$. 
Now, by~(\ref{m'}) and~(\ref{Nk}), $n_i \leq m'/2$ for all $i \in N_1 \cup N_2$.
So $\binom{m'}{n_i} \geq \binom{m'}{2m_0}$ for all $i \in N_2$ and $\binom{m'}{n_i} \geq \binom{m'}{m_0}$ for all $i \in N_1$.
By Proposition~\ref{binom bounds}, we have
\begin{align}\label{pcount}
p_{x'y'}(G[V']) &\geq \frac{1}{100}\prod_{i \in N_1}\binom{m'}{\lfloor m'/4-1 \rfloor}\prod_{j \in N_2}\binom{m'}{2\lfloor m'/4-1 \rfloor} \nonumber \\
 &\geq \frac{1}{100}\cdot 2^{(4|N_1|/5+|N_2|)m'-|N_2|\log_2 m'} \geq \frac{1}{100} \cdot 2^{(1-\gamma)(4|N_1|/5+|N_2|)m'},
\end{align} 
where, for the final inequality, we used the fact that $1/m' \ll \gamma$.
Suppose that $4|N_1|/5+|N_2| \geq \gamma d'$.
Then using the fact that $\frac{n}{r}=(1-\sqrt{\gamma})\frac{d}{d'}$, we have
\begin{eqnarray*}
(1-\gamma)\left(\frac{4|N_1|}{5}+|N_2|\right)m'&\stackrel{\eqref{eq: xi m},\eqref{m'}}{\ge}& (1-\gamma)\left(\frac{4|N_1|}{5}+|N_2|\right) \frac{(1-\sqrt{\eps})(1-\epsilon)n}{r} \\
&\ge & (1-2\gamma)\left(\frac{4|N_1|}{5}+|N_2|\right)(1-\sqrt{\gamma})\frac{d}{d'} \\
&\geq & (1-2\sqrt{\gamma})\left(\frac{4|N_1|}{5} + |N_2|\right) \frac{d}{d'} + 10.
\end{eqnarray*}
We get the last inequality by the facts $1/d \ll \gamma$ and $4|N_1|/5 + |N_2| \geq \gamma d'$.
Together with~(\ref{pcount}), this proves Claim~\ref{Ncount}.
\end{claimproof}

\medskip
\noindent
Now let $x,y$ be arbitrary distinct vertices of $V(G)$.
Since $G$ is $2$-connected, there exist two vertex-disjoint paths from $\{x,y\}$ to $V'$. Let $P_x,P_y$ be two such minimal paths, where $P_x$ is from $x$ to some $x'' \in V'$; and $P_y$ is from $y$ to some $y'' \in V'$. As the choices of $x',y'\in V'$ were arbitrary, by letting $x''=x'$ and $y''=y'$, we see that every distinguishable $x',y'$-path in $G[V']$ together with $P_x$ and $P_y$ forms a distinguishable $x,y$-path in $G$, namely we have
\begin{equation}\label{pbound}
p_{xy}(G) \geq p_{x'y'}(G'[V']).
\end{equation}
We now return to cases (1)--(3) to prove the assertion concerning $p_{xy}(G)$.

\medskip
\noindent
\textbf{Case (1):}
\emph{$H$ is a cycle with $0.9d' \leq |H| \leq 101d'/100$.}

\medskip
\noindent
Note that the lower bound on $|H|$ is implied by the existence of $J$ in this case and that $|H|\ge |J|$, $|H|+|J|\ge 1.8d'$.
Now Proposition~\ref{homomorphism} implies that $|N_2|=|H| \geq 0.9d'$, so
$
4|N_1|/5+|N_2| \geq 0.9d'
$.
Thus~(\ref{pbound}) and Claim~\ref{Ncount} with the fact that $\gamma < 10^{-5}$ imply that
\begin{equation}\label{case1}
p_{xy}(G) \geq p_{x'y'}(G'[V']) \geq 2^{(1-2 \sqrt{\gamma})|H|\cdot d/d'} \geq 2^{0.89d},
\end{equation}
as required.

\medskip
\noindent
\textbf{Case (2):}
\emph{$H$ is a path on $a \geq 101d'/100$ vertices.}

\medskip
\noindent
Now $|N_2| = |H|-2$, so $4|N_1|/5+|N_2| \geq 101d'/100-2$. Thus, Claim~\ref{Ncount} implies that 
$$p_{xy}(G) \geq 2^{(1+\frac{1}{150})d}.$$

\medskip
\noindent
\textbf{Case (3):}
\emph{$H$ is an $(a,b)$-sun with $a \geq d'$ and $b \geq d'/20$.}

\medskip
\noindent
Now $|N_1| = |\Cor(H)| = 2b$ and $|N_2|=a-b$. So 
$$
\frac{4|N_1|}{5} + |N_2| = \frac{3b}{5} + a \geq \left(1+\frac{3}{100}\right)d'.
$$
Similarly by Claim~\ref{Ncount}, we have that $p_{xy}(G) \geq 2^{(1+\frac{1}{150})d}$.

\medskip
\noindent
This completes the proof that $p_{xy}(G) \geq 2^{0.89d}$.
Proving the assertion about $c(G)$ is now easy in Cases (2) and (3).
Here, for arbitrary distinct vertices $x,y$, we have that $p_{xy}(G) \geq 2^{(1+1/150)d}$. Now choose $x,y$ such that $xy \in E(G)$.
At most one $x,y$-path uses this edge and so we have that 
$$c(G) \geq p_{xy}(G)-1 \geq 2^{(1+1/200)d}.$$

Therefore we may assume that we are in Case (1).
Using exactly the same arguments as above for $J$ instead of $H$, we can obtain a set $U' \subseteq V(G)$ (the analogue of $V'$), which is a subset of $\bigcup_{i \in V(J)}V_i$, such that for arbitrary distinct vertices $u',v' \in U'$, we have $p_{u'v'}(G'[U']) \geq 2^{(1-2\sqrt{\gamma})|J|\cdot d/d'}$ (in analogy with the middle inequality in~(\ref{case1})).

Observe that, because $H,J$ are vertex-disjoint subgraphs of $R$, the sets $U',V'$ are disjoint.
Since $G$ is $2$-connected, there exist two vertex-disjoint paths between $U'$ and $V'$.
Choose minimal such paths $P_1$ with endpoints $x',u'$ and $P_2$ with endpoints $y',v'$, where $x',y' \in V'$ and $u',v' \in U'$.
Then distinguishable $x',y'$-paths in $V'$ and $u',v'$-paths in $U'$ together with $P_1,P_2$ yield distinct Hamiltonian subsets in $G$. Recall that $|H|+|J|\ge 1.8d'$. We then have 
$$c(G) \geq p_{x'y'}(G'[V'])\cdot p_{u'v'}(G'[U']) \geq 2^{(1-2\sqrt{\gamma})(|H|+|J|)\cdot d/d'} \geq 2^{1.7d},
$$
as required.
This completes the proof of Lemma~\ref{lem: 1.2 to K}.
\hfill$\square$

\section{The sparse case}\label{sec-main-sparse}
In this section, we will prove the second main ingredient, which states that a large almost-regular expander graph contains many Hamiltonian subsets, see Lemma~\ref{lem-sparse}. To state it formally, we need the following notion of graph expansion, which was introduced by Koml\'os and Szemer\'edi~\cite{K-Sz-1}.

\subsection{Graph expansion}\label{sec-expander}

For $\epsilon_1>0$ and $t>0$, define
\begin{equation}\label{epsilon}
\epsilon(x)=\epsilon(x,\epsilon_1,t):=
\begin{cases} 0 &\mbox{if } x < t/5 \\
\frac{\eps_1}{\log^2(15x/t)} & \mbox{if } x \geq t/5, \end{cases} 
\end{equation}
\noindent where, when it is clear from context we will not write the dependency on $\ep_1$ and $t$ of $\ep(x)$. Note that, for $x\ge t/2$, $\ep(x)$ is decreasing, while $\ep(x)\cdot x$ is increasing. 

\begin{defn}[$(\ep_1,t)$-expander]\label{defn-expander}
A graph $G$ is an \emph{$(\ep_1,t)$-expander} if all subsets $X\subseteq V(G)$ of size $t/2\le |X|\le |G|/2$ satisfy
	$$|\Gamma_G(X)|\ge \ep(|X|)\cdot |X|.$$
\end{defn}

We will use the following lemma, essentially proved by Koml\'{o}s and Szemer\'{e}di~\cite{K-Sz-1,K-Sz-2}, which states that every graph $G$ contains an $(\ep_1,t)$-expander subgraph $H$ whose average degree is almost as large as that of $G$. 
\begin{lemma}[\cite{K-Sz-1,K-Sz-2}]\label{lem-expander}
Let $C>12, \epsilon_1 \leq 1/(10C), c'<1/2, d>0$ and $\ep(x)=\ep(x,\ep_1,c'd)$ as in~\eqref{epsilon}. Then every graph $G$ with $d(G)=d$ has a subgraph $H$ such that
	\begin{itemize}
		\item[(i)] $H$ is an $\left(\ep_1,c'd\right)$-expander;
		\item[(ii)] $d(H)\geq (1-\ep_0)d$, where $\ep_0:=\frac{C\ep_1}{\log 3}<1$;
		\item[(iii)] $\delta(H)\geq d(H)/2$;
		\item[(iv)] $H$ is $\nu d$-connected, where $\nu:=\frac{\ep_1}{6\log^2(5/c')}$.
	\end{itemize}
\end{lemma}
\begin{proof}
	Parts~(i)--(iii) were shown in~\cite{K-Sz-2}. We only need to show $H$ has high connectivity. Suppose $H$ has a vertex cut $S$ of size less than $\nu d$, where $\nu=\ep_1/(6\log^2(5/c'))$. Let $X$ be the smallest component in $H-S$. Then $x:=|X|<|H|/2$. On the other hand, for any vertex $v\in X$, we have $\Gamma_H(\lbrace v \rbrace)\subseteq X\cup S$. Since $\de(H)\ge (1-\ep_0)d/2$, we have that 
	$$|X|\ge \de(H)-|S|> \frac{(1-\ep_0)d}{2}-\nu d\ge \frac{d}{3} \geq \frac{c'd}{2}.$$
The expansion property (i) of $H$ implies that $|\Gamma_H(X)|\ge \ep(x) x$.		Thus, since $\Gamma_H(X)\subseteq S$ and $\ep(x)x$ is increasing for $x\ge c'd/2$, we have
	$$|S|\ge |\Gamma_H(X)|\ge\ep(x) x \ge \ep\left(\frac{d}{3}\right)  \frac{d}{3}=2\nu d>2|S|,$$
	a contradiction.
\end{proof}

We remark that the expander subgraph $H$ found in Lemma~\ref{lem-expander} could be much smaller than~$G$, e.g.~when~$G$ is a disjoint union of small cliques. The following property of expanders (Corollary 2.3 in~\cite{K-Sz-2}) is the only one which we require in our proof. It states that any two sets, provided that they are sufficiently large, are connected by a relatively short path, even after deleting an arbitrary small set of vertices.
\begin{lemma}[\cite{K-Sz-2}]\label{diameter}
	Let $\eps_1,t>0$ and let $H$ be an $n$-vertex $(\ep_1,t)$-expander and $X,X',W\subseteq V(H)$. If $|X|,|X'|\ge x\ge t/2$ and $|W|\le \eps(x)x/4$, then there is a path in $H-W$ from $X$ to $X'$ of length at most 
	$$\frac{2}{\ep_1}\log^3\left(\frac{15n}{t}\right).$$
\end{lemma}

Throughout the rest of this section, we will set
\begin{eqnarray*}
t :=c'd :=\frac{d}{30} \quad \text{ and write }\quad \ep(x) :=\ep\left(x,\ep_1,\frac{d}{30}\right).
\end{eqnarray*}
(So $\eps(x) = \eps_1(\log^2(450x/d))^{-1}$ if $x \geq d/150$ and $\eps(x)=0$ otherwise.)


\subsection{Hamiltonian subsets in expanders.}\label{Hmsbst exp}

The aim of this section and the next is to prove the following lemma, which states that expanders in which every vertex degree is not too far from $d$ contain many Hamiltonian subsets.

\begin{lemma}\label{lem-sparse}
Suppose that $0< 1/d,1/K \ll \epsilon_1, 1/L \leq 1$ and $n\in \mathbb{N}$ is such that $n\geq Kd$. Let $H$ be an $n$-vertex $(\ep_1, d/30)$-expander with $d/10\le \de(H)\le \De(H)\le Ld$. Then $c(H) \geq 2^{50d}$.
\end{lemma}

We split the proof of Lemma~\ref{lem-sparse} into two cases depending on the edge-density of the expander graph; that is, how large $d$ is compared to $n$. 
In the case when $d$ is fairly large compared to $n$, we prove Lemma~\ref{lem-sparse} in Section~\ref{sec-dense}. The proof for the remaining case may be found in Section~\ref{sec-sparse}.

For the remainder of this section, our aim is to prove Lemma~\ref{lem-web}, the main ingredient for the first case.
Roughly speaking, it states that almost regular expanders contain a large collection of `webs' with certain properties. A web (see Definition~\ref{defn-web}) is a special tree which we will use in Section~\ref{sec-dense} to construct many Hamiltonian subsets in the case when $d$ is not too small compared to $n$. A web contains many special subtrees which we call \emph{units}, which themselves contain many large stars. By a $k$-star, we mean a star with $k$ leaves.

\begin{defn}[$(h_1,h_2,h_3)$-unit]\label{defn-unit}
Given integers $h_1,h_2,h_3>0$, we define $F=F_u = \bigcup_{i \in [h_1]}P_i \cup \bigcup_{i \in [h_1]}S(x_i)$ to be an \emph{$(h_1,h_2,h_3)$-unit}~if it satisfies the following.
\begin{itemize}
\item $F$ contains distinct vertices $u$ (the \emph{core vertex} of $F$) and $x_1,\ldots,x_{h_1}$.
\item $\lbrace P_i: i\in[h_1]\rbrace$ is a collection of pairwise internally vertex-disjoint paths, each of length at most $h_3$, such that $P_i$ is a $u,x_i$-path.
\item $\lbrace S(x_i):i\in [h_1]\rbrace$ is a collection of vertex-disjoint $h_2$-stars such that $S(x_i)$ has centre $x_i$ and $\cup_{i \in [h_1]} (S(x_i)-\lbrace x_i\rbrace)$ is vertex-disjoint from $\cup_{i \in [h_1]} P_i$.
\end{itemize}
Define the \emph{exterior} $\Ext(F):=\cup_{i \in [h_1]}(V(S(x_i))- \lbrace x_i\rbrace )$ and \emph{interior} $\Int(F):=V(F)\setminus \Ext(F)$.
For every vertex $w\in \Ext(F)$, let $P(F,w)$ be the unique $u,w$-path in $F$.
\end{defn}

A web is then defined via units as follows.

\begin{defn}[$(h_0,h_1,h_2,h_3)$-web]\label{defn-web}
Given integers $h_0,h_1,h_2,h_3>0$, we define $W = \bigcup_{i \in [h_0]}Q_i \cup \bigcup_{i \in [h_0]}F_{u_i}$ to be an \emph{$(h_0,h_1,h_2,h_3)$-web}~if it satisfies the following.
\begin{itemize}
\item $W$ contains distinct vertices $v$ (the \emph{core} vertex of $W$) and $u_1,\ldots,u_{h_0}$.
\item $\lbrace Q_i:i\in[h_0]\rbrace$ is a collection of pairwise internally vertex-disjoint paths such that $Q_i$ is a $v,u_i$-path of length at most $h_3$.
\item $\lbrace F_{u_i}:i\in [h_0]\rbrace$ is a collection of vertex-disjoint $(h_1,h_2,h_3)$-units such that $F_{u_i}$ has core vertex $u_i$ and $\cup_{i \in [h_0]} (F_{u_i}-\lbrace u_i\rbrace)$ is vertex-disjoint from $\cup_{i \in [h_0]} Q_i$.
\end{itemize}
Define the \emph{exterior} $\Ext(W):=\cup_{i \in [h_0]}\Ext(F_{u_i})$, \emph{interior} $\Int(W):=V(W)\setminus \Ext(W)$ and \emph{centre} $\Cen(W):= \cup_{i \in [h_0]}V(Q_i)$.
For every vertex $w \in \Ext(W)$, let $P(W,w)$ be the unique $v,w$-path in $W$.
\end{defn}

These structures are illustrated in Figure~\ref{figweb}. Note that a web $W$ with core vertex $v$ is a tree with root $v$, and $\Cen(W) \subset \Int(W)$. 

\begin{figure}
\tikzstyle{every node}=[circle, draw, fill=black,
                        inner sep=0pt, minimum width=3pt]
\begin{tikzpicture}[]


\clip (-0.9,-5.5) rectangle (12.5,5.1);


\node[draw=none,fill=none,label=left:$v$] at (0,0) {};
\node[draw=none,fill=none,label=above:$u_1$] at (\a1+\e8:8) {};
\node[draw=none,fill=none,label=above:$u_2$] at (\a2+\f8:8) {};
\node[draw=none,fill=none,label=above:$u_{h_0}$] at (\a3+\g8:8) {};

\draw[<->] (-0.7,3.4) -- (-0.7,-3.7) node[midway,draw=none,fill=white] {$h_0$};
\draw[<->] (0,0.5) -- (7.2,4.5) node[midway,draw=none,fill=white] {$\leq h_3$};
\draw[dotted] (10,0.7) -- (10,-1.1);
\draw[dashed] (11.7,-5) rectangle (6.6,-1.7);
\node[draw=none,fill=none,label=left:{a single $(h_1,h_2,h_3)$-unit} $F(u_{h_0})$] at (12.2,-5.3) {};
\draw[<->] (12.3,-4.8) -- (12.3,-1.9) node[midway,draw=none,fill=white] {$h_1$};
\draw[<->] (6.9,-1.2) -- (10.7,-1.2) node[midway,draw=none,fill=white] {$\leq h_3$};
\draw[<->] (11.4,-1.8) -- (11.4,-2.5) node[midway,draw=none,fill=white] {$h_2$};

\begin{scope}
\draw[] (0,0) node {} to [out=\a1,in={\a1+\e1+180}] ({\a1+\e1}:1) node {}
        to [out=\a1+\e1,in={\a1+\e2+180}] (\a1+\e2:2) node {}
        to [out=\a1+\e2,in={\a1+\e3+180}] (\a1+\e3:3) node {}
        to [out=\a1+\e3,in={\a1+\e4+180}] (\a1+\e4:4) node {}
        to [out=\a1+\e4,in={\a1+\e5+180}] (\a1+\e5:5) node {}
        to [out=\a1+\e5,in={\a1+\e6+180}] (\a1+\e6:6) node {}
        to [out=\a1+\e6,in={\a1+\e7+180}] (\a1+\e7:7) node {}
        to [out=\a1+\e7,in={\a1+\e8+180}] (\a1+\e8:8) node {};

\draw[] (0,0) node {} to [out=\a2,in={\a2+\f1+180}] ({\a2+\f1}:1) node {}
        to [out=\a2+\f1,in={\a2+\f2+180}] (\a2+\f2:2) node {}
        to [out=\a2+\f2,in={\a2+\f3+180}] (\a2+\f3:3) node {}
        to [out=\a2+\f3,in={\a2+\f4+180}] (\a2+\f4:4) node {}
        to [out=\a2+\f4,in={\a2+\f5+180}] (\a2+\f5:5) node {}
        to [out=\a2+\f5,in={\a2+\f6+180}] (\a2+\f6:6) node {}
        to [out=\a2+\f6,in={\a2+\f7+180}] (\a2+\f7:7) node {}
        to [out=\a2+\f7,in={\a2+\f8+180}] (\a2+\f8:8) node {};
        
\draw[] (0,0) node {} to [out=\a3,in={\a3+\g1+180}] ({\a3+\g1}:1) node {}
        to [out=\a3+\g1,in={\a3+\g2+180}] (\a3+\g2:2) node {}
        to [out=\a3+\g2,in={\a3+\g3+180}] (\a3+\g3:3) node {}
        to [out=\a3+\g3,in={\a3+\g4+180}] (\a3+\g4:4) node {}
        to [out=\a3+\g4,in={\a3+\g5+180}] (\a3+\g5:5) node {}
        to [out=\a3+\g5,in={\a3+\g6+180}] (\a3+\g6:6) node {}
        to [out=\a3+\g6,in={\a3+\g7+180}] (\a3+\g7:7) node {}
        to [out=\a3+\g7,in={\a3+\g8+180}] (\a3+\g8:8) node {};
\end{scope}        
        

       \begin{scope}[shift={(\a1+\e8:8)},scale=0.5]
\draw[] (0,0) node {} to [out=\a5,in={\a5+\e1+180}] ({\a5+\e1}:1) node {}
        to [out=\a5+\e1,in={\a5+\e2+180}] (\a5+\e2:2) node {}
        to [out=\a5+\e2,in={\a5+\e3+180}] (\a5+\e3:3) node {}
        to [out=\a5+\e3,in={\a5+\e4+180}] (\a5+\e4:4) node {}
        to [out=\a5+\e4,in={\a5+\e5+180}] (\a5+\e5:5) node {}
        to [out=\a5+\e5,in={\a5+\e6+180}] (\a5+\e6:6) node {}
        to [out=\a5+\e6,in={\a5+\e7+180}] (\a5+\e7:7) node {}
        to [out=\a5+\e7,in={\a5+\e8+180}] (\a5+\e8:8) node {};

\draw[] (0,0) node {} to [out=\a2,in={\a2+\f1+180}] ({\a2+\f1}:1) node {}
        to [out=\a2+\f1,in={\a2+\f2+180}] (\a2+\f2:2) node {}
        to [out=\a2+\f2,in={\a2+\f3+180}] (\a2+\f3:3) node {}
        to [out=\a2+\f3,in={\a2+\f4+180}] (\a2+\f4:4) node {}
        to [out=\a2+\f4,in={\a2+\f5+180}] (\a2+\f5:5) node {}
        to [out=\a2+\f5,in={\a2+\f6+180}] (\a2+\f6:6) node {}
        to [out=\a2+\f6,in={\a2+\f7+180}] (\a2+\f7:7) node {}
        to [out=\a2+\f7,in={\a2+\f8+180}] (\a2+\f8:8) node {};
        
\draw[] (0,0) node {} to [out=\a4,in={\a4+\g1+180}] ({\a4+\g1}:1) node {}
        to [out=\a4+\g1,in={\a4+\g2+180}] (\a4+\g2:2) node {}
        to [out=\a4+\g2,in={\a4+\g3+180}] (\a4+\g3:3) node {}
        to [out=\a4+\g3,in={\a4+\g4+180}] (\a4+\g4:4) node {}
        to [out=\a4+\g4,in={\a4+\g5+180}] (\a4+\g5:5) node {}
        to [out=\a4+\g5,in={\a4+\g6+180}] (\a4+\g6:6) node {}
        to [out=\a4+\g6,in={\a4+\g7+180}] (\a4+\g7:7) node {}
        to [out=\a4+\g7,in={\a4+\g8+180}] (\a4+\g8:8) node {};

\begin{scope}[shift={(\a5+\e8:8)}]
\draw \foreach \x in {30,15,...,-30}        
      {
        (0,0)--(\x:1) node[minimum width=1pt] {}
      } ;
      \end{scope}
      \begin{scope}[shift={(\a2+\f8:8)}]
\draw \foreach \x in {30,15,...,-30}        
      {
        (0,0)--(\x:1) node[minimum width=1pt] {}
      }  ;
            \end{scope}
      \begin{scope}[shift={(\a4+\g8:8)}]
\draw \foreach \x in {70,55,...,10}        
      {
        (0,0)--(\x:1) node[minimum width=1pt] {}
      }  ;
\end{scope}

\end{scope}   
        
        
               \begin{scope}[shift={(\a2+\f8:8)},scale=0.5]
\draw[] (0,0) node {} to [out=\a6,in={\a6+\h1+180}] ({\a6+\h1}:1) node {}
        to [out=\a6+\h1,in={\a6+\h2+180}] (\a6+\h2:2) node {}
        to [out=\a6+\h2,in={\a6+\h3+180}] (\a6+\h3:3) node {}
        to [out=\a6+\h3,in={\a6+\h4+180}] (\a6+\h4:4) node {}
        to [out=\a6+\h4,in={\a6+\h5+180}] (\a6+\h5:5) node {}
        to [out=\a6+\h5,in={\a6+\h6+180}] (\a6+\h6:6) node {}
        to [out=\a6+\h6,in={\a6+\h7+180}] (\a6+\h7:7) node {}
        to [out=\a6+\h7,in={\a6+\h8+180}] (\a6+\h8:8) node {};

\draw[] (0,0) node {} to [out=\a2,in={\a2+\e1+180}] ({\a2+\e1}:1) node {}
        to [out=\a2+\e1,in={\a2+\e2+180}] (\a2+\e2:2) node {}
        to [out=\a2+\e2,in={\a2+\e3+180}] (\a2+\e3:3) node {}
        to [out=\a2+\e3,in={\a2+\e4+180}] (\a2+\e4:4) node {}
        to [out=\a2+\e4,in={\a2+\e5+180}] (\a2+\e5:5) node {}
        to [out=\a2+\e5,in={\a2+\e6+180}] (\a2+\e6:6) node {}
        to [out=\a2+\e6,in={\a2+\e7+180}] (\a2+\e7:7) node {}
        to [out=\a2+\e7,in={\a2+\e8+180}] (\a2+\e8:8) node {};
        
\draw[] (0,0) node {} to [out=\a4,in={\a4+\f1+180}] ({\a4+\f1}:1) node {}
        to [out=\a4+\f1,in={\a4+\f2+180}] (\a4+\f2:2) node {}
        to [out=\a4+\f2,in={\a4+\f3+180}] (\a4+\f3:3) node {}
        to [out=\a4+\f3,in={\a4+\f4+180}] (\a4+\f4:4) node {}
        to [out=\a4+\f4,in={\a4+\f5+180}] (\a4+\f5:5) node {}
        to [out=\a4+\f5,in={\a4+\f6+180}] (\a4+\f6:6) node {}
        to [out=\a4+\f6,in={\a4+\f7+180}] (\a4+\f7:7) node {}
        to [out=\a4+\f7,in={\a4+\f8+180}] (\a4+\f8:8) node {};

\begin{scope}[shift={(\a6+\h8:8)}]
\draw \foreach \x in {30,15,...,-30}        
      {
        (0,0)--(\x:1) node[minimum width=1pt] {}
      } ;
      \end{scope}
      \begin{scope}[shift={(\a2+\e8:8)}]
\draw \foreach \x in {30,15,...,-30}        
      {
        (0,0)--(\x:1) node[minimum width=1pt] {}
      }  ;
            \end{scope}
      \begin{scope}[shift={(\a4+\f8:8)}]
\draw \foreach \x in {30,15,...,-30}        
      {
        (0,0)--(\x:1) node[minimum width=1pt] {}
      }  ;
\end{scope}

\end{scope} 


       \begin{scope}[shift={(\a3+\g8:8)},scale=0.5]
\draw[] (0,0) node {} to [out=\a5,in={\a5+\f1+180}] ({\a5+\f1}:1) node {}
        to [out=\a5+\f1,in={\a5+\f2+180}] (\a5+\f2:2) node {}
        to [out=\a5+\f2,in={\a5+\f3+180}] (\a5+\f3:3) node {}
        to [out=\a5+\f3,in={\a5+\f4+180}] (\a5+\f4:4) node {}
        to [out=\a5+\f4,in={\a5+\f5+180}] (\a5+\f5:5) node {}
        to [out=\a5+\f5,in={\a5+\f6+180}] (\a5+\f6:6) node {}
        to [out=\a5+\f6,in={\a5+\f7+180}] (\a5+\f7:7) node {}
        to [out=\a5+\f7,in={\a5+\f8+180}] (\a5+\f8:8) node {};

\draw[] (0,0) node {} to [out=\a2,in={\a2+\e1+180}] ({\a2+\e1}:1) node {}
        to [out=\a2+\e1,in={\a2+\e2+180}] (\a2+\e2:2) node {}
        to [out=\a2+\e2,in={\a2+\e3+180}] (\a2+\e3:3) node {}
        to [out=\a2+\e3,in={\a2+\e4+180}] (\a2+\e4:4) node {}
        to [out=\a2+\e4,in={\a2+\e5+180}] (\a2+\e5:5) node {}
        to [out=\a2+\e5,in={\a2+\e6+180}] (\a2+\e6:6) node {}
        to [out=\a2+\e6,in={\a2+\e7+180}] (\a2+\e7:7) node {}
        to [out=\a2+\e7,in={\a2+\e8+180}] (\a2+\e8:8) node {};
        
\draw[] (0,0) node {} to [out=\a4,in={\a4+\g1+180}] ({\a4+\g1}:1) node {}
        to [out=\a4+\g1,in={\a4+\g2+180}] (\a4+\g2:2) node {}
        to [out=\a4+\g2,in={\a4+\g3+180}] (\a4+\g3:3) node {}
        to [out=\a4+\g3,in={\a4+\g4+180}] (\a4+\g4:4) node {}
        to [out=\a4+\g4,in={\a4+\g5+180}] (\a4+\g5:5) node {}
        to [out=\a4+\g5,in={\a4+\g6+180}] (\a4+\g6:6) node {}
        to [out=\a4+\g6,in={\a4+\g7+180}] (\a4+\g7:7) node {}
        to [out=\a4+\g7,in={\a4+\g8+180}] (\a4+\g8:8) node {};

\node[draw=none,fill=none,label=above:$x_1$] at (\a5+\f8:8) {};
\node[draw=none,fill=none,label=above:$x_2$] at (\a2+\e8:8) {};
\node[draw=none,fill=none,label=above:$x_{h_1}$] at (\a4+\g8:8) {};
        
        \begin{scope}[shift={(\a5+\f8:8)}]
\draw \foreach \x in {30,15,...,-30}        
      {
        (0,0)--(\x:1) node[minimum width=1pt] {}
      } ;
      \end{scope}
      \begin{scope}[shift={(\a2+\e8:8)}]
\draw \foreach \x in {30,15,...,-30}        
      {
        (0,0)--(\x:1) node[minimum width=1pt] {}
      }  ;
            \end{scope}
      \begin{scope}[shift={(\a4+\g8:8)}]
\draw \foreach \x in {30,15,...,-30}        
      {
        (0,0)--(\x:1) node[minimum width=1pt] {}
      }  ;
\end{scope}

\end{scope} 
\end{tikzpicture}
\caption{An $(h_0,h_1,h_2,h_3)$-web $W$.}
\label{figweb}
\end{figure}
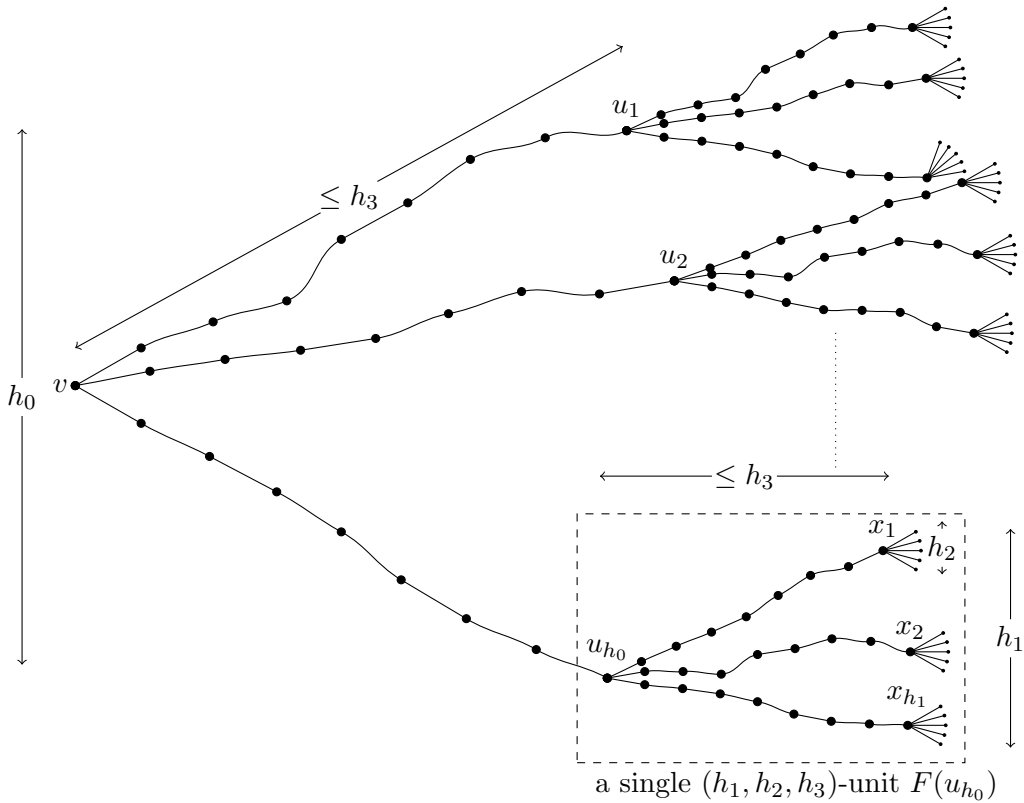

Throughout Sections~\ref{Hmsbst exp}~and~\ref{sec-dense}, we will define a parameter $m$ as follows and assume that:
$$
m:=\frac{2}{\ep_1}\log^3\left(\frac{450n}{d}\right);\quad  0 < \frac{1}{d}, \frac{1}{K} \ll \eps_1, \frac{1}{L} \leq 1; \quad\text{and}\quad \log^{100}n \leq d \leq \frac{n}{K}.
$$
These assumptions imply the following useful inequalities (whose derivations we omit).
If $d/30 \leq x \leq n$, then
\begin{eqnarray}\label{eq-n-large}
\ep(x)\ge \ep(n) > \frac{1}{m};\quad\text{ and also }\quad  n\ge Ldm^{100}  \quad \mbox{ and }\quad  d\ge m^{30}.
\end{eqnarray}

The following lemma guarantees a large collection of webs with disjoint interiors and is a key ingredient in Section~\ref{sec-dense}.

\begin{lemma}\label{lem-web}
Suppose that $0<1/d,1/K \ll \epsilon_1, 1/L \leq 1$ and $n,d\in \mathbb{N}$ with $\log^{100}n \leq d \leq n/K$. Let $H$ be an $n$-vertex $(\ep_1,d/30)$-expander with $d/10\le \delta (H)\le\Delta(H)\leq Ld$. Then $H$ contains~$200d$ webs $W_1,\ldots, W_{200d}$ such that the following hold.
	\begin{itemize}
		\item[(i)] $W_i$ is an $(m^3, m^3, d/100,4m)$-web for all $i \in [200d]$;
		\item[(ii)] $\Int(W_1),\ldots,\Int(W_{200d})$ are pairwise disjoint sets of vertices.
	\end{itemize}
\end{lemma}

One cannot hope that the webs themselves are disjoint. Indeed, the sum of the size of the exteriors of the desired collection of webs is at least $200d\cdot m^{3}\cdot m^3\cdot d/100=2d^2m^6$, which is much larger than $n$ when $d>\sqrt{n}$.

We will prove Lemma~\ref{lem-web} in the following two subsections. First we use many disjoint stars to construct a large collection of disjoint units (Lemma~\ref{lem-unit}). We then use these units to build the desired webs.

\subsubsection{From stars to units}\label{sec-unit}

\begin{lemma}\label{lem-unit}
Suppose that $0< 1/d,1/K \ll \epsilon_1, 1/L \leq 1$ and $n,d\in \mathbb{N}$ with $\log^{100}n \leq d \leq n/K$. Let $h_1 \in [m^{10}]$ and $h_2\in [d/100]$. Suppose that $H$ is an $n$-vertex $(\ep_1,d/30)$-expander with $d/10\le \delta (H)\le \Delta(H)\leq Ld$. Let $X\subseteq V(H)$ have size at most $dm^{10}$.
Then $H-X$ contains a collection of $dm^{15}/(h_1h_2)$ pairwise vertex-disjoint $(h_1, h_2, m+2)$-units.
\end{lemma}

We briefly sketch the proof of Lemma~\ref{lem-unit}.
Suppose we have already found some vertex-disjoint units and wish to find another, $F$, to add to the collection.
Remove those vertices in $X'$ used in the units that we have already found.
Our graph $H-X'$ contains a collection $\mathcal{S}$ of many large disjoint stars since it is still almost regular.
We partition the set of centres of stars in $\mathcal{S}$ into two groups $U$ and $V$. 
 Then take a maximal collection $\mathcal{P}$ of short paths in $H-X'$ which go between $U$ and $V$, whose interiors are disjoint and do not intersect $U$ or $V$.
Note that these paths could intersect the leaves of stars in $\mathcal{S}$.
We argue that this collection of paths can be extended unless there is some $v \in V$ which is the endpoint of every path in some large subset $\mathcal{P}'\subseteq \mathcal{P}$.
Let $U' \subseteq U$ be the non-$v$ endpoints of paths in $\mathcal{P}'$.
Let $v$ be the core vertex of $F$ and the vertices $I$ in paths of $\mathcal{P}'$ form the interior of $F$.
If a leaf of a $U'$-centred star lies in $I$, remove it.
Few such leaves are removed since the paths in $\mathcal{P}'$ are short and the stars of $\mathcal{S}$ are large.
The remaining leaves form the exterior of $F$.
Figure~\ref{figunit} illustrates the proof of Lemma~\ref{lem-unit}.

\begin{figure}
\tikzstyle{every node}=[circle, draw, fill=black,
inner sep=0pt, minimum width=2pt]
\begin{tikzpicture}[]

\begin{scope}[yshift=-10cm]
\draw[red,thick] (30:1) node {}
to [out=20,in={20+180}] (20:2) node[black] {}
to [out=20,in={20+1+180}] (20+1:3) node[black] {}
to [out=20+1,in={20+180}] (20:4) node[black] {}
to [out=20,in={20+-1+180}] (20+-1:5) node[black] {}
to [out=20+-1,in={20+180}] (20:6) node[black] {}
to [out=20,in={20+-1+180}] (20+-1:7) node[black] {}
to [out=20+-1,in={20+180}] (8,2.5) node[black] {};

\draw[red,thick] (15:1) node {}
to [out=10,in={10+-5+180}] (10+-5:2) node[black] {}
to [out=10+-5,in={10+-8+180}] (10+-8:3) node[black] {}
to [out=10+-8,in={10+-1+180}] (10+-1:4) node[black] {}
to [out=10+-1,in={10+-1+180}] (10+-1:5) node[black] {}
to [out=10+-1,in={10+180}] (10:6) node[black] {}
to [out=10,in={10+-2+180}] (10+-2:7) node[black] {}
to [out=10+-2,in={10+180}] (8.1,1.5) node[black] {};

\draw[red,thick] (-15:1) node {}
to [out=-10+1,in={-10+180}] (-10:2) node[black] {}
to [out=-10,in={-10+-1+180}] (-10+-1:3) node[black] {}
to [out=-10+-1,in={-10+-1+180}] (-10+-1:4) node[black] {}
to [out=-10+-1,in={-10+1+180}] (-10+1:5) node[black] {}
to [out=-10+1,in={-10+2+180}] (-10+2:6) node[black] {}
to [out=-10+2,in={-10+180}] (-10:7) node[black] {}
to [out=-10,in={-10+180}] (8.12,-1.5) node[black] {};

\begin{scope}
\draw[red,thick] \foreach \x in {30,15,-15}        
{
	(0,0)--(\x:1) node[black,minimum width=2pt] {}
} ;
\draw[] \foreach \x in {0,-30}        
{
	(0,0)--(\x:1) node[black,minimum width=2pt] {}
} ;
\end{scope}

\begin{scope}[shift={(9,4)}]
\draw \foreach \x in {30,15,...,-30}        
{
	(0,0)--(180+\x:1) node[minimum width=2pt] {}
} ;
\end{scope}
\begin{scope}[shift={(9,2.5)}]
\draw[red,thick] \foreach \x in {30,15,...,-30}        
{
	(0,0)--(180+\x:1) node[black,minimum width=2pt] {}
} ;
\end{scope}
\begin{scope}[shift={(9,1.0)}]
\draw[red,thick] \foreach \x in {30,15,...,-30}        
{
	(0,0)--(180+\x:1) node[black,minimum width=2pt] {}
}  ;
\end{scope}
\begin{scope}[shift={(9,-1)}]
\draw[red,thick] \foreach \x in {30,15,...,-30}        
{
	(0,0)--(180+\x:1) node[black,minimum width=2pt] {}
}  ;
\end{scope}
\begin{scope}[shift={(9,-2.2)}]
\draw \foreach \x in {30,15,...,-30}        
{
	(0,0)--(180+\x:1) node[black,minimum width=2pt] {}
}  ;
\end{scope}

\node[label=left:$v_1$,minimum width=3pt] at (0,2) {};
\node[label=left:$v_i$,minimum width=3pt] at (0,0) {};
\node[label=left:$v_{m^{20}}$,minimum width=3pt] at (0,-1.2) {};
\node[label=right:$u_1$,minimum width=3pt] at (9,4) {};
\node[label=right:$u_j$,minimum width=3pt] at (9,2.5) {};
\node[label=right:$u_k$,minimum width=3pt] at (9,1.0) {};
\node[label=right:$u_\ell$,minimum width=3pt] at (9,-1) {};
\node[label=right:$u_{m^{40}}$,minimum width=3pt] at (9,-2.2) {};
\node[draw=none,fill=none,label=right:$U$] at (9.8,0.9) {};
\node[draw=none,fill=none,label=right:$V$] at (-1.2,0.4) {};
\node[draw=none,fill=none,label=below:$X'$] at (4.5,-2.2) {};
\node[draw=none,fill=none,label=right:{$P_{v_i,u_j}$}] at (3.8,1.9) {};
\draw[dotted] (0.5,1.5)--(0.5,0.5);
\draw[dotted] (8.5,3.5)--(8.5,3);
\draw[dotted] (8.5,2)--(8.5,1.5);
\draw[dotted] (8.5,0.5)--(8.5,-0.5);

\draw[rounded corners] (8.75,4.25) rectangle (9.8,-2.45);
\draw[rounded corners] (-0.8,2.25) rectangle (0.25,-1.45);
\draw[rounded corners] (2,-2) rectangle (7,-2.95);
\end{scope}

\begin{scope}[yshift=-8cm]
\begin{scope}
\draw \foreach \x in {30,15,...,-30}        
{
	(0,0)--(\x:1) node[minimum width=2pt] {}
} ;
\end{scope}
\begin{scope}[yshift=-3.2cm]
\draw \foreach \x in {30,15,...,-30}        
{
	(0,0)--(\x:1) node[minimum width=2pt] {}
} ;
\end{scope}

\draw[] (15:1) node {}
to [out=20,in={20+180}] (20:2) node {}
to [out=20,in={20+1+180}] (20+1:3) node {}
to [out=20+1,in={20+180}] (20:4) node {}
to [out=20,in={20+-1+180}] (20+-1:5) node {}
to [out=20+-1,in={20+180}] (20:6) node {}
to [out=20,in={20+-1+180}] (20+-1:7) node {}
to [out=20+-1,in={20+180}] (8.15,2.5) node {};

\draw[] (0:1) node {}
to [out=-10+1,in={-10+180}] (-10:2) node {}
to [out=-10,in={-10+-1+180}] (-10+-1:3) node {}
to [out=-10+-1,in={-10+-1+180}] (3.75,-0.63) node {}
to [out=-10+-1,in={-10+1+180}] (-10+1:5) node {}
to [out=-10+1,in={-10+2+180}] (5.9,-0.97) node {}
to [out=-10+2,in={-10+180}] (-10:7) node {}
to [out=-10,in={-10+180}] (8.1,-1.5) node {};

\end{scope} 
	\end{tikzpicture}
	\caption{The proof of Lemma~\ref{lem-unit}: a unit with core vertex $v_i \in V$ which avoids $X'$.}
	\label{figunit}
\end{figure}
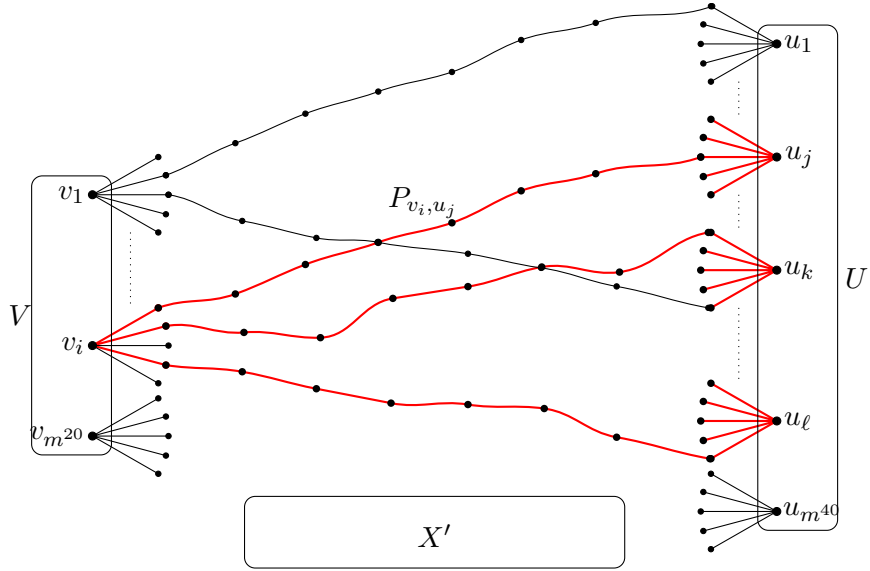

\medskip
\noindent
\emph{Proof of Lemma~\ref{lem-unit}.} Note that $dm^{15}/(h_1h_2)$ vertex-disjoint $(h_1, h_2, m+2)$-units together contain at most $(h_1h_2 + (m+2)h_1)dm^{15}/(h_1h_2) \leq 1.1d m^{16}$ vertices. Thus it suffices to show that we can find an $(h_1, h_2, m+2)$-unit in $H-X'$ for an arbitrary $X'\subseteq V(H)$ of size at most $2dm^{16}$ (then one can take $X'$ to be the union of $X$ and at most $1.1dm^{16}$ vertices in the units we have already built). Fix such a set $X'$. We first find many vertex-disjoint stars in $H-X'$.
	\begin{claim}\label{claim-star}
		$H-X'$ contains a collection of $m^{50}$ pairwise vertex-disjoint $d/50$-stars.
	\end{claim}
	\begin{claimproof}
	Consider a maximal collection $\mathcal{S}$ of vertex-disjoint $d/50$-stars in $H-X'$ and let $X''$ be the set of vertices spanned by the stars of $\mathcal{S}$. If $\mathcal{S}$ contains less than $m^{50}$ stars, then $|X''|\le dm^{50}$.
	
Then the fact that $d/10\le \delta (H)\le \Delta(H)\leq Ld$ implies that
	\begin{eqnarray*}
	d(H-X'-X'')\ge \frac{n\cdot\delta(H)-|X'\cup X''|\cdot 2\De(H)}{n}\ge \frac{d}{10}-\frac{2dm^{50}\cdot 2Ld}{n}\stackrel{\eqref{eq-n-large}}{\ge} \frac{d}{50}.
	\end{eqnarray*}	
Thus $H-X'-X''$ contains a $d/50$-star, contradicting to the maximality of $\cS$.
	\end{claimproof}
	
	\medskip
	\noindent
	By Claim~\ref{claim-star}, we can find $m^{20}+m^{40}$ pairwise vertex-disjoint $d/50$-stars $S(v_1),\dots, S(v_{m^{20}})$ and $S(u_1),\dots, S(u_{m^{40}})$
	in $H-X'$ such that $S(x)$ has centre $x$ for all $x\in \{v_i, u_j: i\in [m^{20}], j\in [m^{40}]\}$. Set $V:=\{v_1,\ldots,v_{m^{20}}\}$ and $U:=\{u_1,\ldots,u_{m^{40}}\}$. 
Let $\mathcal{P}$ be a maximal subset of $U \times V$ such that a collection $\mathcal{Q} := \{P_{v_i,u_j} : (v_i,u_j)\in \mathcal{P}\}$ of paths with the following properties exists.
	\begin{itemize}
		\item[(C1)] $P_{v_i,u_j}$ is a $v_i,u_j$-path in $H$ of length at most $m+2$ which contains an edge of $S(u_j)$;
		\item[(C2)] for each $(v_i,u_j)\in\mathcal{P}$, $\Int(P_{v_i,u_j})$ is disjoint from $X'\cup U \cup V$;
		\item[(C3)] all paths in $\mathcal{Q}$ are pairwise internally vertex-disjoint. 
	\end{itemize}
Let $\Gamma_{\mathcal{P}}(v_i) := \{ u_j : (v_i,u_j)\in \mathcal{P}\}$.	We claim that there is a $v_i \in V$ such that 
	\begin{align}\label{eq: index connects to many}
	|\Gamma_{\mathcal{P}}(v_i)| \geq m^{10}. 
	\end{align} 
	Suppose to the contrary that 	$|\Gamma_{\mathcal{P}}(v_i)| < m^{10}$ for all $i\in [m^{20}]$. Let
	\begin{align*}
	P'&:=\bigcup_{(v_i,u_j)\in\mathcal{P}} \Int(P_{v_i,u_j});\quad J:= \{j \in [m^{40}]: V(S(u_j))\cap P'=\emptyset\};\\
	A&:= \bigcup_{i\in [m^{20}]}(V(S(v_i))-\lbrace v_i \rbrace); \quad B:=\bigcup_{j\in J} (V(S(u_j))-\lbrace u_j\rbrace)\enspace \quad \text{ and }\quad W:= X'\cup U \cup V\cup P'.
	\end{align*}
We will construct a path between $A$ and $B$ which avoids $W$ by using Lemma~\ref{diameter} to contradict the maximality of $\mathcal{P}$. In order to do this, we estimate the sizes of $A, B$ and $W$.
We have
	\begin{align}\label{eq: P' size}
	|P'|\leq \sum_{(v_i,u_j)\in\mathcal{P}} |\Int(P_{v_i,u_j})|\overset{\textrm{(C1)}}{\leq}   \sum_{(v_i,u_j)\in\mathcal{P}} m\leq  m \sum_{i\in [m^{20}]} |\Gamma_{\mathcal{P}}(v_i)|  \leq m\cdot m^{10} \cdot m^{20}=m^{31}.
	\end{align}
This implies that 
\begin{align}\label{eq: J size}
 |J|\geq m^{40} - |P'| \geq m^{40} - m^{31} \geq \frac{m^{40}}{2}.
 \end{align} 	
Since the $S(v_{i})$ are vertex-disjoint $d/50$-stars, we have 
	\begin{eqnarray}\label{eq: A size}
		|A|= \left|\bigcup_{i\in[m^{20}]}(V(S(v_i))-\lbrace v_i\rbrace)\right| \geq m^{20}\cdot \frac{d}{50} \ge dm^{19},
	\end{eqnarray}
Also \eqref{eq: J size} implies that
\begin{align}\label{eq: B size}
|B|= \left|\bigcup_{j\in J} (V(S(u_j))-\lbrace u_j \rbrace)\right| \geq \frac{m^{40}}{2} \cdot  \frac{d}{50} \geq dm^{19}.
\end{align}
Recall that, by~\eqref{eq-n-large}, $d\ge m^{30}$. Thus
	\begin{eqnarray*}
	|W|= |X'|+ |U| + |V|+|P'| \stackrel{\eqref{eq: P' size}}{\leq} 2dm^{16}+m^{40} + m^{20}+ m^{31}\stackrel{\eqref{eq-n-large}}{\leq} d m^{17} < \frac{1}{4} \epsilon(dm^{19}) dm^{19}.
	\end{eqnarray*}
This together with \eqref{eq: A size} and \eqref{eq: B size} allows us to apply Lemma~\ref{diameter} with $A,B,W$ and $d m^{19}$ playing the roles of $X,X',W$ and $x$ respectively. Then we can find a minimal path $Q$ of length at most $\frac{2}{\epsilon_1}\log^3(450n/d)=m$ in $H-W$ from $A$ to $B$. Then there exists $i'\in [m^{20}]$ and $j'\in [m^{40}]$ such that each endpoint of $Q$ belongs to $S(v_{i'})$ and $S(u_{j'})$.  Then $H[Q\cup \{v_{i'},u_{j'}\}]$ contains a path $P_{v_{i'},u_{j'}}$ from $v_{i'}$ to $u_{j'}$. Note that $P_{v_{i'},u_{j'}}$ has length at most $m+2$.

Note that $(v_{i'},u_{j'}) \notin \mathcal{P}$ because $u_{j'} \in J$ and $\mathcal{Q}$ satisfies (C1).
Let $\mathcal{P}' := \mathcal{P} \cup \lbrace (v_{i'},u_{j'}) \rbrace$ and $\mathcal{Q}' := \mathcal{Q} \cup \lbrace P_{v_{i'},u_{j'}} \rbrace$.
We claim that $(\mathcal{P}',\mathcal{Q}')$ satisfies (C1)--(C3).
Indeed, $(\mathcal{P}',\mathcal{Q}')$ satisfies (C1) since every path in $\mathcal{Q}$ as well as $P_{v_{i'},u_{j'}}$ has length at most $m+2$ and contains an edge of $S(u_{j'})$. Also (C2) holds because $\Int(P_{v_{i'},u_{j'}}) \subseteq Q$ is disjoint from $W=X'\cup U \cup V\cup P'$ and $\mathcal{Q}$ satisfies (C2). Finally (C3) is satisfied since $\Int(P_{v_i,u_j}) \subseteq Q$ is disjoint from $P'=\bigcup_{(v_i,u_j)\in\mathcal{P}} \Int(P_{v_i,u_j})$ and $\mathcal{Q}$ satisfies (C3).
This contradicts the maximality of $\mathcal{P}$, and thus there exists a vertex $v_i \in V$ satisfying \eqref{eq: index connects to many}. 

Now we use $v_i$ to construct an $(h_1,h_2,m+2)$-unit as desired (see Figure~\ref{figunit}).
Let $U'\subseteq N_{\mathcal{P}}(v_i)$ be a subset of $U$ such that $|U'|=h_1 \leq m^{10}$.
Let $u_j\in N_{\mathcal{P}}(v_i)$ be arbitrary.
Then by (C2) $u_j \notin \Int(P_{v_i,u_{j'}})$ for any $u_{j'} \in U'$, and also 
$$\left|V(S(u_j))\setminus \left(\bigcup_{u_{j'} \in U'} \Int(P_{v_i,u_{j'}})\right)\right|\geq \frac{d}{50} - m^{10} \cdot m \geq \frac{d}{100} \geq h_2.$$
Thus $V(S(u_j))\setminus(\bigcup_{u_{j'} \in U'} \Int(P_{v_i,u_{j'}}))$ contains an $h_2$-star $S'(u_j)$ and $\bigcup_{u_j \in U'} (P_{v_i,u_j}\cup S'(u_j))$ forms an $(h_1,h_2,m+2)$-unit disjoint from $X'$. This finishes the proof. 
\hfill$\square$

\subsubsection{From units to webs}\label{sec-web}

We can now use Lemma~\ref{lem-unit} to prove Lemma~\ref{lem-web} (and the proof itself is similar).

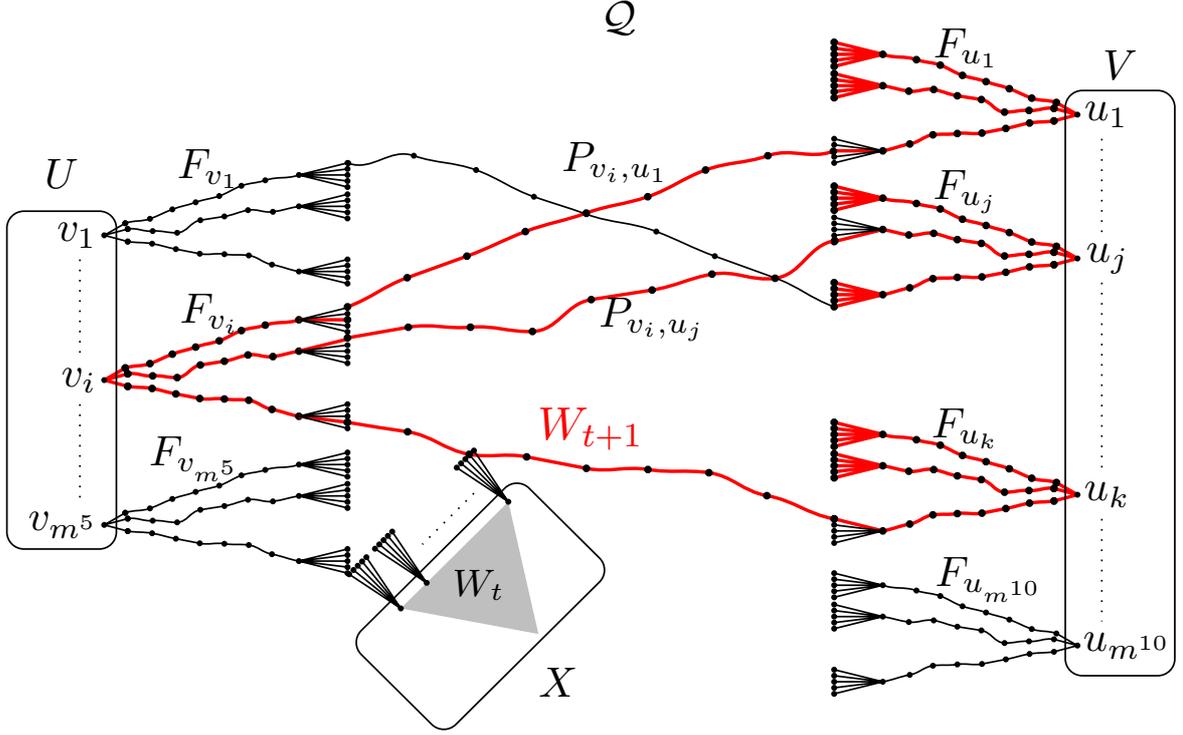
\begin{figure}
	\scalebox{1.6}{
\tikzstyle{every node}=[circle, draw, fill=black,
inner sep=0pt, minimum width=1pt]
\begin{tikzpicture}[]

\begin{scope}

\begin{scope}[scale=0.2]
\draw[red,thick] (30:1) node[black] {}
to [out=20,in={20+180}] (20:2) node[black] {}
to [out=20,in={20+1+180}] (20+1:3) node[black] {}
to [out=20+1,in={20+180}] (20:4) node[black] {}
to [out=20,in={20+-1+180}] (20+-1:5) node[black] {}
to [out=20+-1,in={20+180}] (20:6) node[black] {}
to [out=20,in={20+-1+180}] (20+-1:7) node[black] {}
to [out=20+-1,in={20+180}] (8,2.5) node[black] {};

\draw[red,thick] (15:1) node[black] {}
to [out=10,in={10+-5+180}] (10+-5:2) node[black] {}
to [out=10+-5,in={10+-8+180}] (10+-8:3) node[black] {}
to [out=10+-8,in={10+-1+180}] (10+-1:4) node[black] {}
to [out=10+-1,in={10+-1+180}] (10+-1:5) node[black] {}
to [out=10+-1,in={10+180}] (10:6) node[black] {}
to [out=10,in={10+-2+180}] (10+-2:7) node[black] {}
to [out=10+-2,in={10+180}] (8,1.2) node[black] {};

\draw[red,thick] (-15:1) node[black] {}
to [out=-10+1,in={-10+180}] (-10:2) node[black] {}
to [out=-10,in={-10+-1+180}] (-10+-1:3) node[black] {}
to [out=-10+-1,in={-10+-1+180}] (-10+-1:4) node[black] {}
to [out=-10+-1,in={-10+1+180}] (-10+1:5) node[black] {}
to [out=-10+1,in={-10+2+180}] (-10+2:6) node[black] {}
to [out=-10+2,in={-10+180}] (-10:7) node[black] {}
to [out=-10,in={-10+180}] (8,-1.5) node[black] {};

\begin{scope}
\draw[red,thick] \foreach \x in {30,15,-15}        
{
	(0,0)--(\x:1) node[black] {}
} ;
\end{scope}

\begin{scope}[shift={(8,2.5)}]
\draw[red,thick] (0,0) -- (2,0) node[black] {};
\draw[black] \foreach \x in {0.5,0.25,-0.25,-0.5}        
{
	(0,0)--(2,\x) node[black] {}
} ;
\end{scope}
\begin{scope}[shift={(8,1.2)}]
\draw[red,thick] (0,0) -- (2,0.5);
\draw[black] \foreach \x in {0.25,0,...,-0.5}        
{
	(0,0)--(2,\x) node[black] {}
}  ;
\end{scope}
\begin{scope}[shift={(8,-1.5)}]
\draw[red,thick] (0,0) -- (2,-0.25) node[black] {};
\draw[black] \foreach \x in {0.5,0.25,0,-0.5}         
{
	(0,0)--(2,\x) node[black] {}
}  ;
\end{scope}
\node[label=left:{\small{${v_i}$}}] at (0,0) {};
\node[draw=none,fill=none,label=left:{\small{$F_{v_i}$}}] at (6,2.75) {};
\end{scope}

\begin{scope}[yshift=1.2cm,scale=0.2]
\draw[black] (30:1) node {}
to [out=20,in={20+180}] (20:2) node[black] {}
to [out=20,in={20+1+180}] (20+1:3) node[black] {}
to [out=20+1,in={20+180}] (20:4) node[black] {}
to [out=20,in={20+-1+180}] (20+-1:5) node[black] {}
to [out=20+-1,in={20+180}] (20:6) node[black] {}
to [out=20,in={20+-1+180}] (20+-1:7) node[black] {}
to [out=20+-1,in={20+180}] (8,2.5) node[black] {};

\draw[black] (15:1) node {}
to [out=10,in={10+-5+180}] (10+-5:2) node[black] {}
to [out=10+-5,in={10+-8+180}] (10+-8:3) node[black] {}
to [out=10+-8,in={10+-1+180}] (10+-1:4) node[black] {}
to [out=10+-1,in={10+-1+180}] (10+-1:5) node[black] {}
to [out=10+-1,in={10+180}] (10:6) node[black] {}
to [out=10,in={10+-2+180}] (10+-2:7) node[black] {}
to [out=10+-2,in={10+180}] (8,1.2) node[black] {};

\draw[black] (-15:1) node {}
to [out=-10+1,in={-10+180}] (-10:2) node[black] {}
to [out=-10,in={-10+-1+180}] (-10+-1:3) node[black] {}
to [out=-10+-1,in={-10+-1+180}] (-10+-1:4) node[black] {}
to [out=-10+-1,in={-10+1+180}] (-10+1:5) node[black] {}
to [out=-10+1,in={-10+2+180}] (-10+2:6) node[black] {}
to [out=-10+2,in={-10+180}] (-10:7) node[black] {}
to [out=-10,in={-10+180}] (8,-1.5) node[black] {};

\begin{scope}
\draw[black] \foreach \x in {30,15,-15}        
{
	(0,0)--(\x:1) node[black] {}
} ;
\end{scope}

\begin{scope}[shift={(8,2.5)}]
\draw[black] \foreach \x in {0.5,0.25,...,-0.5}        
{
	(0,0)--(2,\x) node[black] {}
} ;
\end{scope}
\begin{scope}[shift={(8,1.2)}]
\draw[black] \foreach \x in {0.5,0.25,...,-0.5}        
{
	(0,0)--(2,\x) node[black] {}
}  ;
\end{scope}
\begin{scope}[shift={(8,-1.5)}]
\draw[black] \foreach \x in {0.5,0.25,...,-0.5}         
{
	(0,0)--(2,\x) node[black] {}
}  ;
\end{scope}
\node[label=left:{\small{${v_1}$}}] at (0,0) {};
\node[draw=none,fill=none,label=left:{\small{$F_{v_1}$}}] at (6,2.75) {};
\end{scope} 

\begin{scope}[yshift=-1.2cm,scale=0.2]
\draw[black] (30:1) node {}
to [out=20,in={20+180}] (20:2) node[black] {}
to [out=20,in={20+1+180}] (20+1:3) node[black] {}
to [out=20+1,in={20+180}] (20:4) node[black] {}
to [out=20,in={20+-1+180}] (20+-1:5) node[black] {}
to [out=20+-1,in={20+180}] (20:6) node[black] {}
to [out=20,in={20+-1+180}] (20+-1:7) node[black] {}
to [out=20+-1,in={20+180}] (8,2.5) node[black] {};

\draw[black] (15:1) node {}
to [out=10,in={10+-5+180}] (10+-5:2) node[black] {}
to [out=10+-5,in={10+-8+180}] (10+-8:3) node[black] {}
to [out=10+-8,in={10+-1+180}] (10+-1:4) node[black] {}
to [out=10+-1,in={10+-1+180}] (10+-1:5) node[black] {}
to [out=10+-1,in={10+180}] (10:6) node[black] {}
to [out=10,in={10+-2+180}] (10+-2:7) node[black] {}
to [out=10+-2,in={10+180}] (8,1.2) node[black] {};

\draw[black] (-15:1) node {}
to [out=-10+1,in={-10+180}] (-10:2) node[black] {}
to [out=-10,in={-10+-1+180}] (-10+-1:3) node[black] {}
to [out=-10+-1,in={-10+-1+180}] (-10+-1:4) node[black] {}
to [out=-10+-1,in={-10+1+180}] (-10+1:5) node[black] {}
to [out=-10+1,in={-10+2+180}] (-10+2:6) node[black] {}
to [out=-10+2,in={-10+180}] (-10:7) node[black] {}
to [out=-10,in={-10+180}] (8,-1.5) node[black] {};

\begin{scope}
\draw[black] \foreach \x in {30,15,-15}        
{
	(0,0)--(\x:1) node[black] {}
} ;
\end{scope}

\begin{scope}[shift={(8,2.5)}]
\draw[black] \foreach \x in {0.5,0.25,...,-0.5}        
{
	(0,0)--(2,\x) node[black] {}
} ;
\end{scope}
\begin{scope}[shift={(8,1.2)}]
\draw[black] \foreach \x in {0.5,0.25,...,-0.5}        
{
	(0,0)--(2,\x) node[black] {}
}  ;
\end{scope}
\begin{scope}[shift={(8,-1.5)}]
\draw[black] \foreach \x in {0.5,0.25,...,-0.5}         
{
	(0,0)--(2,\x) node[black] {}
}  ;
\end{scope}
\node[label=left:{\small{${v_{m^5}}$}}] at (0,0) {};
\node[draw=none,fill=none,label=left:{\small{$F_{v_{m^5}}$}}] at (6,2.75) {};
\end{scope}

\begin{scope}[yshift=2.2cm,xshift=8cm,scale=0.2,yscale=1,xscale=-1]

\begin{scope}[shift={(8,2.5)}]
\draw[red,thick] \foreach \x in {0.5,0.25,...,-0.5}        
{
	(0,0)--(2,\x) node[black] {}
} ;
\end{scope}
\begin{scope}[shift={(8,1.2)}]
\draw[red,thick] \foreach \x in {0.5,0.25,...,-0.5}        
{
	(0,0)--(2,\x) node[black] {}
}  ;
\end{scope}
\begin{scope}[shift={(8,-1.5)}]
\draw[red,thick] (0,0)--(2,0);
\draw[black] \foreach \x in {0.5,0.25,-0.25,-0.5}         
{
	(0,0)--(2,\x) node[black] {}
}  ;
\end{scope}
\draw[red,thick] (30:1) node {}
to [out=20,in={20+180}] (20:2) node[black] {}
to [out=20,in={20+1+180}] (20+1:3) node[black] {}
to [out=20+1,in={20+180}] (20:4) node[black] {}
to [out=20,in={20+-1+180}] (20+-1:5) node[black] {}
to [out=20+-1,in={20+180}] (20:6) node[black] {}
to [out=20,in={20+-1+180}] (20+-1:7) node[black] {}
to [out=20+-1,in={20+180}] (8,2.5) node[black] {};

\draw[red,thick] (15:1) node {}
to [out=10,in={10+-5+180}] (10+-5:2) node[black] {}
to [out=10+-5,in={10+-8+180}] (10+-8:3) node[black] {}
to [out=10+-8,in={10+-1+180}] (10+-1:4) node[black] {}
to [out=10+-1,in={10+-1+180}] (10+-1:5) node[black] {}
to [out=10+-1,in={10+180}] (10:6) node[black] {}
to [out=10,in={10+-2+180}] (10+-2:7) node[black] {}
to [out=10+-2,in={10+180}] (8,1.2) node[black] {};

\draw[red,thick] (-15:1) node {}
to [out=-10+1,in={-10+180}] (-10:2) node[black] {}
to [out=-10,in={-10+-1+180}] (-10+-1:3) node[black] {}
to [out=-10+-1,in={-10+-1+180}] (-10+-1:4) node[black] {}
to [out=-10+-1,in={-10+1+180}] (-10+1:5) node[black] {}
to [out=-10+1,in={-10+2+180}] (-10+2:6) node[black] {}
to [out=-10+2,in={-10+180}] (-10:7) node[black] {}
to [out=-10,in={-10+180}] (8,-1.5) node[black] {};
\begin{scope}
\draw[red,thick] \foreach \x in {30,15,-15}        
{
	(0,0)--(\x:1) node[black] {}
} ;
\end{scope}               

\node[label=right:{\small{${u_1}$}}] at (0,0) {};
\node[draw=none,fill=none,label=left:{\small{$F_{u_1}$}}] at (2.8,2.75) {};
\end{scope}

\begin{scope}[yshift=1.008cm,xshift=8cm,scale=0.2,yscale=1,xscale=-1]

\begin{scope}[shift={(8,2.5)}]
\draw[red,thick] \foreach \x in {0.5,0.25,...,-0.5}        
{
	(0,0)--(2,\x) node[black] {}
} ;
\end{scope}
\begin{scope}[shift={(8,1.2)}]
\draw[red,thick] (0,0)--(2,-0.5) node[black] {};
\draw[black] \foreach \x in {0.5,0.25,...,-0.25}        
{
	(0,0)--(2,\x) node[black] {}
}  ;
\end{scope}
\begin{scope}[shift={(8,-1.5)}]
\draw[red,thick] \foreach \x in {0.5,0.25,...,-0.5}         
{
	(0,0)--(2,\x) node[black] {}
}  ;
\end{scope}
\draw[red,thick] (30:1) node {}
to [out=20,in={20+180}] (20:2) node[black] {}
to [out=20,in={20+1+180}] (20+1:3) node[black] {}
to [out=20+1,in={20+180}] (20:4) node[black] {}
to [out=20,in={20+-1+180}] (20+-1:5) node[black] {}
to [out=20+-1,in={20+180}] (20:6) node[black] {}
to [out=20,in={20+-1+180}] (20+-1:7) node[black] {}
to [out=20+-1,in={20+180}] (8,2.5) node[black] {};

\draw[red,thick] (15:1) node {}
to [out=10,in={10+-5+180}] (10+-5:2) node[black] {}
to [out=10+-5,in={10+-8+180}] (10+-8:3) node[black] {}
to [out=10+-8,in={10+-1+180}] (10+-1:4) node[black] {}
to [out=10+-1,in={10+-1+180}] (10+-1:5) node[black] {}
to [out=10+-1,in={10+180}] (10:6) node[black] {}
to [out=10,in={10+-2+180}] (10+-2:7) node[black] {}
to [out=10+-2,in={10+180}] (8,1.2) node[black] {};

\draw[red,thick] (-15:1) node {}
to [out=-10+1,in={-10+180}] (-10:2) node[black] {}
to [out=-10,in={-10+-1+180}] (-10+-1:3) node[black] {}
to [out=-10+-1,in={-10+-1+180}] (-10+-1:4) node[black] {}
to [out=-10+-1,in={-10+1+180}] (-10+1:5) node[black] {}
to [out=-10+1,in={-10+2+180}] (-10+2:6) node[black] {}
to [out=-10+2,in={-10+180}] (-10:7) node[black] {}
to [out=-10,in={-10+180}] (8,-1.5) node[black] {};
\begin{scope}
\draw[red,thick] \foreach \x in {30,15,-15}        
{
	(0,0)--(\x:1) node[black] {}
} ;
\end{scope}               

\node[label=right:{\small{${u_j}$}}] at (0,0) {};
\node[draw=none,fill=none,label=left:{\small{$F_{u_j}$}}] at (2.8,2.75) {};
\end{scope} 

\begin{scope}[yshift=-0.95cm,xshift=8cm,scale=0.2,yscale=1,xscale=-1]
\begin{scope}[shift={(8,2.5)}]
\draw[red,thick] \foreach \x in {0.5,0.25,...,-0.5}        
{
	(0,0)--(2,\x) node[black] {}
} ;
\end{scope}
\begin{scope}[shift={(8,1.2)}]
\draw[red,thick] \foreach \x in {0.5,0.25,...,-0.5}        
{
	(0,0)--(2,\x) node[black] {}
}  ;
\end{scope}
\begin{scope}[shift={(8,-1.5)}]
\draw[red,thick] (0,0)--(2,0.5) node[black] {};
\draw[black] \foreach \x in {0.25,0,...,-0.5}         
{
	(0,0)--(2,\x) node[black] {}
}  ;
\end{scope}
\draw[red,thick] (30:1) node {}
to [out=20,in={20+180}] (20:2) node[black] {}
to [out=20,in={20+1+180}] (20+1:3) node[black] {}
to [out=20+1,in={20+180}] (20:4) node[black] {}
to [out=20,in={20+-1+180}] (20+-1:5) node[black] {}
to [out=20+-1,in={20+180}] (20:6) node[black] {}
to [out=20,in={20+-1+180}] (20+-1:7) node[black] {}
to [out=20+-1,in={20+180}] (8,2.5) node[black] {};

\draw[red,thick] (15:1) node {}
to [out=10,in={10+-5+180}] (10+-5:2) node[black] {}
to [out=10+-5,in={10+-8+180}] (10+-8:3) node[black] {}
to [out=10+-8,in={10+-1+180}] (10+-1:4) node[black] {}
to [out=10+-1,in={10+-1+180}] (10+-1:5) node[black] {}
to [out=10+-1,in={10+180}] (10:6) node[black] {}
to [out=10,in={10+-2+180}] (10+-2:7) node[black] {}
to [out=10+-2,in={10+180}] (8,1.2) node[black] {};

\draw[red,thick] (-15:1) node {}
to [out=-10+1,in={-10+180}] (-10:2) node[black] {}
to [out=-10,in={-10+-1+180}] (-10+-1:3) node[black] {}
to [out=-10+-1,in={-10+-1+180}] (-10+-1:4) node[black] {}
to [out=-10+-1,in={-10+1+180}] (-10+1:5) node[black] {}
to [out=-10+1,in={-10+2+180}] (-10+2:6) node[black] {}
to [out=-10+2,in={-10+180}] (-10:7) node[black] {}
to [out=-10,in={-10+180}] (8,-1.5) node[black] {};

\begin{scope}
\draw[red,thick] \foreach \x in {30,15,-15}        
{
	(0,0)--(\x:1) node[black] {}
} ;
\end{scope}

\node[label=right:{\small{${u_k}$}}] at (0,0) {};
\node[draw=none,fill=none,label=left:{\small{$F_{u_k}$}}] at (2.8,2.75) {};
\end{scope}

\begin{scope}[yshift=-2.2cm,xshift=8cm,scale=0.2,yscale=1,xscale=-1]
\draw[black] (30:1) node {}
to [out=20,in={20+180}] (20:2) node[black] {}
to [out=20,in={20+1+180}] (20+1:3) node[black] {}
to [out=20+1,in={20+180}] (20:4) node[black] {}
to [out=20,in={20+-1+180}] (20+-1:5) node[black] {}
to [out=20+-1,in={20+180}] (20:6) node[black] {}
to [out=20,in={20+-1+180}] (20+-1:7) node[black] {}
to [out=20+-1,in={20+180}] (8,2.5) node[black] {};

\draw[black] (15:1) node {}
to [out=10,in={10+-5+180}] (10+-5:2) node[black] {}
to [out=10+-5,in={10+-8+180}] (10+-8:3) node[black] {}
to [out=10+-8,in={10+-1+180}] (10+-1:4) node[black] {}
to [out=10+-1,in={10+-1+180}] (10+-1:5) node[black] {}
to [out=10+-1,in={10+180}] (10:6) node[black] {}
to [out=10,in={10+-2+180}] (10+-2:7) node[black] {}
to [out=10+-2,in={10+180}] (8,1.2) node[black] {};

\draw[black] (-15:1) node {}
to [out=-10+1,in={-10+180}] (-10:2) node[black] {}
to [out=-10,in={-10+-1+180}] (-10+-1:3) node[black] {}
to [out=-10+-1,in={-10+-1+180}] (-10+-1:4) node[black] {}
to [out=-10+-1,in={-10+1+180}] (-10+1:5) node[black] {}
to [out=-10+1,in={-10+2+180}] (-10+2:6) node[black] {}
to [out=-10+2,in={-10+180}] (-10:7) node[black] {}
to [out=-10,in={-10+180}] (8,-1.5) node[black] {};

\begin{scope}
\draw[black] \foreach \x in {30,15,-15}        
{
	(0,0)--(\x:1) 
} ;
\end{scope}

\begin{scope}[shift={(8,2.5)}]
\draw[black] \foreach \x in {0.5,0.25,...,-0.5}        
{
	(0,0)--(2,\x) node[black] {}
} ;
\end{scope}
\begin{scope}[shift={(8,1.2)}]
\draw[black] \foreach \x in {0.5,0.25,...,-0.5}        
{
	(0,0)--(2,\x) node[black] {}
}  ;
\end{scope}
\begin{scope}[shift={(8,-1.5)}]
\draw[black] \foreach \x in {0.5,0.25,...,-0.5}         
{
	(0,0)--(2,\x) node[black] {}
}  ;
\end{scope}
\node[label=right:{\small{${u_{m^{10}}}$}}] at (0,0) {};
\node[draw=none,fill=none,label=left:{\small{$F_{u_{m^{10}}}$}}] at (1.0,2.75) {};
\end{scope} 

\begin{scope}[shift={(2,0.67)},scale=0.522]
\draw[red,thick] (0,-0.12) node[black] {} to [out=20,in={20+180}] ({20}:1) node[black] {}
to [out=20,in={20+180}] (20:2) node[black] {}
to [out=20,in={20+1+180}] (20+1:3) node[black] {}
to [out=20+1,in={20+180}] (20:4) node[black] {}
to [out=20,in={20+-1+180}] (20+-1:5) node[black] {}
to [out=20+-1,in={20+180}] (20:6) node[black] {}
to [out=20,in={20+-1+180}] (20+-1:7) node[black] {}
to [out=20+-1,in={20+180}] (7.65,2.35) node[black] {};
\end{scope}

\begin{scope}[shift={(2,0.35)},scale=0.507]               
\draw[red,thick] (0,0) node[black] {} to [out=10,in={10+180}] ({10}:1) node[black] {}
to [out=10,in={10+-5+180}] (10+-5:2) node[black] {}
to [out=10+-5,in={10+-8+180}] (10+-8:3) node[black] {}
to [out=10+-8,in={10+-1+180}] (10+-1:4) node[black] {}
to [out=10+-1,in={10+-1+180}] (10+-1:5) node[black] {}
to [out=10+-1,in={10+180}] (10:6) node[black] {}
to [out=10,in={10+-2+180}] (10+-2:7) node[black] {}
to [out=10+-2,in={10+180}] (7.9,1.59) node[black] {};
\end{scope}              

\begin{scope}[shift={(2,-0.35)},scale=0.5] 
\draw[red,thick] (0,0) node[black] {} to [out=-10,in={-10+1+180}] ({-10+1}:1) node[black] {}
to [out=-10+1,in={-10+180}] (2,-0.53) node[black] {}
to [out=-10,in={-10+-1+180}] (-10+-1:3) node[black] {}
to [out=-10+-1,in={-10+-1+180}] (-10+-1:4) node[black] {}
to [out=-10+-1,in={-10+1+180}] (-10+1:5) node[black] {}
to [out=-10+1,in={-10+2+180}] (-10+2:6) node[black] {}
to [out=-10+2,in={-10+180}] (-10:7) node[black] {}
to [out=-10,in={-10+180}] (8.0,-1.6) node[black] {};
\end{scope}    

\begin{scope}[shift={(6,0.67)},scale=0.522,xscale=-1]
\draw[black] (0,-0.12) node[black] {} to [out=20,in={20+180}] ({20}:1) node[black] {}
to [out=20,in={20+180}] (20:2) node[black] {}
to [out=20,in={20+1+180}] (20+1:3) node[black] {}
to [out=20+1,in={20+180}] (3.9,1.35) node[black] {}
to [out=20,in={20+-1+180}] (20+-1:5) node[black] {}
to [out=20+-1,in={20+180}] (20:6) node[black] {}
to [out=20,in={20+-1+180}] (20+-1:7) node[black] {}
to [out=20+-1,in={20+180}] (7.65,2.16) node[black] {};
\end{scope}

\node[draw=none,fill=none,label=right:{\small{$P_{v_i,u_1}$}}] at (3.7,1.8) {};   
\node[draw=none,fill=none,label=right:{\small{$P_{v_i,u_j}$}}] at (4.0,0.5) {};   
\node[draw=none,fill=none,label=right:{\textcolor{red}{$W_{t+1}$}}] at (3.5,-0.4) {};   
\node[draw=none,fill=none,label=right:{\small{$\mathcal{Q}$}}] at (4,3) {};   
\node[draw=none,fill=none,label=right:{\small{$X$}}] at (3.5,-2.5) {}; 
\node[draw=none,fill=none,label={\small{$U$}}] at (-0.35,1.5) {}; 
\node[draw=none,fill=none,label={\small{$V$}}] at (8.35,2.4) {};                \draw[dotted] (-0.2,1.0)--(-0.2,0.2);
\draw[dotted] (-0.2,-0.2)--(-0.2,-1);
\draw[dotted] (8.2,2.0)--(8.2,1.208);
\draw[dotted] (8.2,0.808)--(8.2,-0.75);
\draw[dotted] (8.2,-1.15)--(8.2,-2.0);

\draw[rounded corners] (7.9,2.4) rectangle (8.8,-2.45);
\draw[rounded corners] (-0.8,1.4) rectangle (0.1,-1.4);

\begin{scope}[shift={(2.65,-1.68)},rotate=135]
\filldraw[draw=none,fill=lightgray] (0,0.3)--(0,-0.95)--(-0.95,-0.35) -- cycle;
\draw \foreach \y in {-0.95,0,0.3}
{
	node at (0,\y) {}
};
\draw[dotted] (0.25,-0.8)--(0.25,-0.2);
\draw[black] \foreach \y in {-0.95,0,0.3} \foreach \x in {0.1,0.05,...,-0.1}         
{
	(0,\y)--(0.5,\y+\x) node[black] {}
}  ;  
\draw[rounded corners] (0.1,0.8) -- (0.1,-1.15) -- (-1,-1.15) -- (-1,0.8) -- cycle;

\end{scope}

\node[draw=none,fill=none,label=right:{\footnotesize{$W_{t}$}}] at (2.8,-1.7) {};

\end{scope}

\end{tikzpicture}}
	\caption{The proof of Lemma~\ref{lem-web}: a web $W_{t+1}$ with core vertex $v_i$ whose interior avoids $X$ (the union of the interiors of $W_1,\ldots,W_{t+1}$). Observe that $W_{t+1}$ can and typically does intersect with the exteriors $\bigcup_{i\in[t]} \Ext(W_{i})$ of previously-built webs, but it does not intersect with the interiors $\bigcup_{i\in[t]} \Int(W_i)$ of previously-built webs.}
	\label{figwebbuild}
\end{figure}

\begin{proof}[Proof of Lemma~\ref{lem-web}]
We iteratively construct $(m^3,m^3,d/100,4m)$-webs $W_1,\dots, W_{200d}$ such that $\Int(W_i)\cap \Int(W_j)=\emptyset$ for all $ij \in \binom{[200d]}{2}$. Suppose that we have constructed $W_1, \dots , W_t$ for some $t<200d$. Let $X:= \bigcup_{i\in[t]} \Int(W_i)$. Then
	\begin{align}\label{eq: size X}
	|X|\le 200d\cdot m^3\cdot m^3\cdot 8m\le dm^{8}.
	\end{align}

Now we construct an $(m^3, m^3, d/100, 4m)$-web $W_{t+1}$ in $H-X$.
Apply Lemma~\ref{lem-unit} with $H,X,m^5,d/100$ playing the roles of $H,X,h_1,h_2$.
Since $dm^{15}(m^5 \cdot d/100)^{-1} \geq m^{10}+m^5$, this implies that $H-X$ contains
vertex-disjoint $(m^5,d/100, m+2)$-units $F_{v_1},\dots, F_{v_{m^{5}}}$ and $F_{u_1}, \dots, F_{u_{m^{10}}}$
 such that $F_{w}$ has core vertex $w$ for all $w\in\{v_i,u_j: i\in[m^5],j\in[m^{10}]\}$.
	
	Set $V:=\{v_1,\ldots,v_{m^{5}}\}$ and $U:=\{u_1,\ldots,u_{m^{10}}\}$. Let $\mathcal{P}\subseteq V\times U$ be a maximal subset such that a collection $\mathcal{Q} := \{P_{v_i,u_j} : (v_i,u_j)\in \mathcal{P}\}$ of paths with the following properties exists. Recall the definition of $P(F,w)$ given in Definition~\ref{defn-unit}.
	\begin{itemize}
		\item[(D1)] $P_{v_i,u_j}$ is a $v_i,u_j$-path of length at most $4m$ such that $P_{v_i,u_j}=P(F_{v_i},v_{ij})\cup P^*_{ij}\cup P(F_{u_j},u_{ij})$, where $v_{ij}\in \Ext(F_{v_i})$, $u_{ij}\in \Ext(F_{u_j})$ and $P^*_{ij}$ is a $v_{ij},u_{ij}$-path;
		\item[(D2)] for each $(v_i,u_j)\in\mathcal{P}$, $P^*_{ij}$ is of length at most $m$ and disjoint from $X\cup U \cup \bigcup_{k\in [m^{5}]}\Int(F_{v_k})$, and $P_{v_i,u_j}$ is disjoint from $X$;
		\item[(D3)] all paths in $\mathcal{Q}$ are pairwise internally vertex-disjoint. 
	\end{itemize}
Let $N_{\mathcal{P}}(v_i) := \{ u_j : (v_i,u_j)\in \mathcal{P}\}$.	We claim that there is a $v_i \in V$ such that 
	\begin{align}\label{eq: index connects to many2}
	|N_{\mathcal{P}}(v_i)| \geq m^{3}. 
	\end{align} 
	Suppose to the contrary that 	$|N_{\mathcal{P}}(v_i)| < m^{3}$ for all $i\in [m^{5}]$. Let
	\begin{align*}
	P'&:=\bigcup_{(v_{i},u_j)\in\mathcal{P}}\Int(P_{v_i,u_j});\quad J:= \{j \in [m^{10}]: V(F_{u_j})\cap P'=\emptyset\};\\
A&:= \bigcup_{i\in [m^5]} \{ w \in \Ext(F_{v_i}) : V(P(F_{v_i},w))\cap V(P')=\emptyset\}; \quad B:=\bigcup_{j\in J} \Ext(F_{u_j})\quad \text{ and }\\
C&:= P' \cup X\cup U \cup \bigcup_{k\in [m^{5}]}\Int(F_{v_k}).
\end{align*}
We will construct a path between $A$ and $B$ which avoids $C$ by using Lemma~\ref{diameter} to contradict the maximality of $\mathcal{P}$. In order to do this, we estimate the sizes of $A, B$ and $C$.
	
Since each $P_{v_i,u_j} \in \mathcal{Q}$ has length at most $4m$ and no $v_i \in V$ satisfies~\eqref{eq: index connects to many2}, we have
\begin{align}\label{eq: P' size 2}
	|P'|\leq \sum_{i\in [m^5]} \sum_{u_j\in N_{\mathcal{P}}(v_i)} |\Int(P_{v_i,u_j})|< m^{5}\cdot m^{3} \cdot 4m = 4m^{9}.
\end{align}
Thus we have
\begin{align}\label{eq: J size 2}
|J|\geq m^{10} - |P'| \geq m^{10} - 4m^{9} \geq \frac{m^{10}}{2}.
\end{align} 
For any $i \in [m^5]$,
recall that $\Int(F_{v_i})$ consists of $m^5$ pairwise internally vertex-disjoint paths.
By (D1),~(D2) and the fact that $|N_{\mathcal{P}}(v_i)|<m^3$, the set $P'$ contains at most $m^3$ of these paths. Thus there are at least $m^5-m^3\ge m^5/2$ unused paths, implying that
\begin{eqnarray}\label{eq: A size 2}
\nonumber |A| \geq  \sum_{i\in [m^5]} \frac{m^5}{2}\cdot \frac{d}{100} - |P'| \overset{\eqref{eq: P' size 2}}{\ge}\frac{dm^{10}}{200}-4m^9 \geq 10d m^9.
\end{eqnarray}
Also 
\begin{eqnarray*}
|B|=\left|\bigcup_{j\in J} \Ext(F_{u_j})\right|= |J| \cdot m^5 \cdot \frac{d}{100} \stackrel{\eqref{eq: J size 2}}{\geq}   \frac{dm^{15}}{200} \geq 10dm^{9}.
\end{eqnarray*}	

Then, by \eqref{eq-n-large}, \eqref{eq: size X} and \eqref{eq: P' size 2},
\begin{eqnarray*}
|C|&=&|P'|+|X|+|U|+\left|\bigcup_{k\in [m^{5}]}\Int(F_{v_k}) \right| 
\leq 4m^9 + dm^8 + m^{10} +\sum_{k \in [m^5]} m^5(m+2) \nonumber \\ &\leq& 2dm^8 \leq \frac{1}{4}\ep(10dm^9)\cdot 10dm^9.
\end{eqnarray*}	
Thus, we can apply Lemma~\ref{diameter} with $ A,B,C,10d m^{9}$ playing the roles of $X,X',W,x$ respectively to find a minimal path $Q$ of length at most $m$ in $H-C$ from $v\in A$ to $u\in B$. Let $i' \in [m^5]$ and $j' \in [m^{10}]$ be the indices such that $v\in \Ext(F_{v_{i'}})$ and $u\in \Ext(F_{u_{j'}})$. Then $P_{v_{i'},u_{j'}}:=Q\cup P(F_{v_{i'}},v)\cup P(F_{u_{j'}},u)$ is a $v_{i'},u_{j'}$-path of length at most
$$
m+ |P(F_{v_{i'}},v)|+|P(F_{u_{j'}},u)|\leq m + 2(m+3) \leq 4m.
$$
Note that by the definitions of $J$ and $B$ and the fact that $\mathcal{Q}$ satisfies (D1), we have $(v_{i'},u_{j'})\notin\mathcal{P}$. Let
$$
\mathcal{P}' := \mathcal{P}\cup \{ (v_{i'},u_{j'})\} \text{ and } \mathcal{Q}' := \mathcal{Q}\cup \{P_{v_{i'},u_{j'}}\}.
$$
We claim that $(\mathcal{P}',\mathcal{Q}')$ satisfies (D1)--(D3).
It is easy to see that (D1) holds; while (D2) and~(D3) follow from (D1), $Q\cap C=\emptyset$ and the fact that all units, in particular $F_{v_{i'}}$ and $F_{u_{j'}}$, are in $H-X$. This contradicts the maximality of $\mathcal{P}$.
Thus there exists a vertex $v_i \in V$ satisfying \eqref{eq: index connects to many2}. 

We can now construct the desired web (see Figure~\ref{figwebbuild}).
Let $U'\subseteq N_{\mathcal{P}}(v_i)$ be a subset of $U$ such that $|U'|= m^{3}$.
Let $P'':= \bigcup_{u_j\in U'} \Int(P_{v_{i},u_{j}})$. Then
$|P''|\leq |U'|\cdot 4m = 4m^4$. Recall that $P''$ does not contain any vertex in $U$ by (D2) and that $F_{u_j}$ is an $(m^5,d/100,m+2)$-unit for each $u_j\in U'$.
Thus $F_{u_j}- P''$ contains an $(m^3,d/100,m+2)$-unit $F'_{u_j}$ because $m^5 - 4m^4 \geq m^3$. Then $\bigcup_{u_j\in U'} \left( P_{v_{i},u_{j}} \cup F'_{u_j} \right)$
is an $(m^3,m^3,d/100,4m)$-web with core vertex $v_i$ in $H-X$. This finishes the proof. 	
\end{proof}


\subsection{Proof of Lemma~\ref{lem-sparse} when $d\ge \log^{100} n$}\label{sec-dense}

Using Lemma~\ref{lem-web}, we are now able to prove Lemma~\ref{lem-sparse} assuming in addition that $d\geq \log^{100} n$.
This is achieved via Theorem~\ref{lem-build-sub}, the general theorem about cycles in expander graphs stated in the introduction. We should note that the assumption $d \geq \log^{100}n$ is only used implicitly in the proof by applications of~(\ref{eq-n-large}).
Before presenting the proof, we give a brief sketch.
The desired cycle is found via an iterative procedure.
By Lemma~\ref{lem-web} we can find a collection of webs $W_1,\ldots,W_{200d}$ whose interiors are disjoint and let $Z$ be the set of core vertices of these webs and $U\in{Z\choose 100d}$. The aim is to find $u_1,\ldots,u_{k}$, $k \geq 98d$, in $U$ such that, for each $i \in [k]$, there is a short $u_i,u_{i+1}$-path $P_i$ which avoids $Z$ and the centres of any web not centred at $u_i$ or $u_{i+1}$, and $\Int(P_1),\ldots,\Int(P_{i})$ are pairwise disjoint.
If we can find such a path $P_i$ for all $i\in[k]$ (where $u_{k+1} := u_1$), then the concatenation of these paths is the desired cycle $C_U$.
To achieve this, at the $i^{\mathrm{th}}$ step, we ensure that most webs, including the current web centred at $u_i$, have few of their interior vertices in the paths $P_1,\ldots,P_{i-1}$ which we have already created. Then the exteriors of our webs are still large enough to exclude previously-built paths to find the next path $P_{i}$ between $u_i$ and some $u_{i+1}$ as required.

\begin{figure}
	\scalebox{1.6}{
\tikzstyle{every node}=[circle, draw, fill=black,
inner sep=0pt, minimum width=1pt]
\begin{tikzpicture}[decoration={markings,
	mark=between positions 0 and 1 step 2mm
	with { \draw node {}; }} ]

\begin{scope}[shift=(0:4),rotate=180]
\filldraw[draw=none,fill=gray!25] (0,0)--(1.5,0.5)--(1.5,-0.5)--cycle;
\filldraw[draw=none,fill=gray!50] (0,0)--(18.435:0.75)--(-18.435:0.75)--cycle;
\draw node[] at (0,0) (a0) {};
\draw \foreach \y in {0.5,0.2,-0.5}
{
	node at (1.5,\y) {}
};
\draw[dotted] (1.7,0)--(1.7,-0.3);
\draw[black] \foreach \y in {0.5,0.2,-0.5} \foreach \x in {0.1,0.05,...,-0.1}
{
	(1.5,\y)--(2,\y+\x) node {}
};
\draw node at (2,0.5+0.1) (a1) {};
\draw node at (1.5,0.5) (a3) {};
\draw node at (2,0.2+0.05) (a2) {};
\draw node at (1.5,0.2) (a4) {};
\end{scope}

\begin{scope}[shift=(45:4),rotate=225]
\filldraw[draw=none,fill=gray!25] (0,0)--(1.5,0.5)--(1.5,-0.5)--cycle;
\filldraw[draw=none,fill=gray!50] (0,0)--(18.435:0.75)--(-18.435:0.75)--cycle;
\draw node[] at (0,0) (b0) {};
\draw \foreach \y in {0.5,0.2,-0.5}
{
	node at (1.5,\y) {}
};
\draw[dotted] (1.7,0)--(1.7,-0.3);
\draw[black] \foreach \y in {0.5,0.2,-0.5} \foreach \x in {0.1,0.05,...,-0.1}
{
	(1.5,\y)--(2,\y+\x) node {}
};
\draw node at (2,0.5+0.1) (b1) {};
\draw node at (1.5,0.5) (b3) {};
\draw node at (2,0.2+0.05) (b2) {};
\draw node at (1.5,0.2) (b4) {};
\end{scope}

\begin{scope}[shift=(90:4),rotate=270]
\filldraw[draw=none,fill=gray!25] (0,0)--(1.5,0.5)--(1.5,-0.5)--cycle;
\filldraw[draw=none,fill=gray!50] (0,0)--(18.435:0.75)--(-18.435:0.75)--cycle;
\draw node[] at (0,0) (c0) {};
\draw \foreach \y in {0.5,0.2,-0.5}
{
	node at (1.5,\y) {}
};
\draw[dotted] (1.7,0)--(1.7,-0.3);
\draw[black] \foreach \y in {0.5,0.2,-0.5} \foreach \x in {0.1,0.05,...,-0.1}
{
	(1.5,\y)--(2,\y+\x) node {}
};
\draw node at (2,-0.5-0.1) (c1) {};
\draw node at (1.5,-0.5) (c3) {};
\draw node at (2,0.2+0.05) (c2) {};
\draw node at (1.5,0.2) (c4) {};
\end{scope}

\begin{scope}[shift=(135:4),rotate=315]
\filldraw[draw=none,fill=gray!25] (0,0)--(1.5,0.5)--(1.5,-0.5)--cycle;
\filldraw[draw=none,fill=gray!50] (0,0)--(18.435:0.75)--(-18.435:0.75)--cycle;
\draw node[] at (0,0) (d0) {};
\draw \foreach \y in {0.5,0.2,-0.5}
{
	node at (1.5,\y) {}
};
\draw[dotted] (1.7,0)--(1.7,-0.3);
\draw[black] \foreach \y in {0.5,0.2,-0.5} \foreach \x in {0.1,0.05,...,-0.1}
{
	(1.5,\y)--(2,\y+\x) node {}
};
\draw node at (2,0.5+0.1) (d1) {};
\draw node at (1.5,0.5) (d3) {};
\draw node at (2,0.2+0.05) (d2) {};
\draw node at (1.5,0.2) (d4) {};
\end{scope}

\begin{scope}[shift=(180:0.5),rotate=180]
\filldraw[draw=none,fill=gray!25] (0,0)--(1.5,0.5)--(1.5,-0.5)--cycle;
\filldraw[draw=none,fill=gray!50] (0,0)--(18.435:0.75)--(-18.435:0.75)--cycle;
\draw node[] at (0,0) (e0) {};
\draw \foreach \y in {0.5,0.2,-0.5}
{
	node at (1.5,\y) {}
};
\draw[dotted] (1.7,0)--(1.7,-0.3);
\draw[black] \foreach \y in {0.5,0.2,-0.5} \foreach \x in {0.1,0.05,...,-0.1}
{
	(1.5,\y)--(2,\y+\x) node {}
};
\draw node at (2,0.5+0.1) (e1) {};
\draw node at (1.5,0.5) (e3) {};
\draw node at (2,0.2+0.05) (e2) {};
\draw node at (1.5,0.2) (e4) {};
\draw[dashed,rounded corners] (-0.1,-0.7) rectangle (2.1,0.7);
\end{scope}

\begin{scope}[shift=(225:4),rotate=225+180]
\filldraw[draw=none,fill=gray!25] (0,0)--(1.5,0.5)--(1.5,-0.5)--cycle;
\filldraw[draw=none,fill=gray!50] (0,0)--(18.435:0.75)--(-18.435:0.75)--cycle;
\draw node[] at (0,0) (f0) {};
\draw \foreach \y in {0.5,0.2,-0.5}
{
	node at (1.5,\y) {}
};
\draw[dotted] (1.7,0)--(1.7,-0.3);
\draw[black] \foreach \y in {0.5,0.2,-0.5} \foreach \x in {0.1,0.05,...,-0.1}
{
	(1.5,\y)--(2,\y+\x) node {}
};
\draw node at (2,0.2+0.1) (f1) {};
\draw node at (1.5,0.2) (f3) {};
\draw node at (2,-0.5+0.05) (f2) {};
\draw node at (1.5,-0.5) (f4) {};
\draw node at (2,0.5+0.1) (f5) {};
\draw node at (1.5,0.5) (f6) {};
\end{scope}

\begin{scope}[shift=(270:4),rotate=90]
\filldraw[draw=none,fill=gray!25] (0,0)--(1.5,0.5)--(1.5,-0.5)--cycle;
\filldraw[draw=none,fill=gray!50] (0,0)--(18.435:0.75)--(-18.435:0.75)--cycle;
\draw node[] at (0,0) (g0) {};
\draw \foreach \y in {0.5,0.2,-0.5}
{
	node at (1.5,\y) {}
};
\draw[dotted] (1.7,0)--(1.7,-0.3);
\draw[black] \foreach \y in {0.5,0.2,-0.5} \foreach \x in {0.1,0.05,...,-0.1}
{
	(1.5,\y)--(2,\y+\x) node {}
};
\draw node at (2,0.5+0.1) (g1) {};
\draw node at (1.5,0.5) (g3) {};
\draw node at (2,0.2+0.05) (g2) {};
\draw node at (1.5,0.2) (g4) {};
\end{scope}

\begin{scope}[shift=(315:4),rotate=135]
\filldraw[draw=none,fill=gray!25] (0,0)--(1.5,0.5)--(1.5,-0.5)--cycle;
\filldraw[draw=none,fill=gray!50] (0,0)--(18.435:0.75)--(-18.435:0.75)--cycle;
\draw node[] at (0,0) (h0) {};
\draw \foreach \y in {0.5,0.2,-0.5}
{
	node at (1.5,\y) {}
};
\draw[dotted] (1.7,0)--(1.7,-0.3);
\draw[black] \foreach \y in {0.5,0.2,-0.5} \foreach \x in {0.1,0.05,...,-0.1}
{
	(1.5,\y)--(2,\y+\x) node {}
};
\draw node at (2,0.5+0.1) (h1) {};
\draw node at (1.5,0.5) (h3) {};
\draw node at (2,0.2+0.05) (h2) {};
\draw node at (1.5,0.2) (h4) {};
\end{scope}

\draw[red] plot [smooth] coordinates {(a2) (-10:1.4) (20:1.5) (b1)};
\path plot [smooth] coordinates {(a2) (-10:1.4) (20:1.5) (b1)} [postaction={draw,decorate,decoration={pre length=0.1cm, post length=0.1cm}}];

\draw[red] plot [smooth] coordinates {(b2) (60:1.4) (70:2.5) (75:3) (86:3) (80:3.5) (85:4.5) (95:4.5) (100:3.5) (105:3) (110:2.5) (c1)};
\path plot [smooth] coordinates {(b2) (60:1.4) (70:2.5) (75:3) (86:3) (80:3.5) (85:4.5) (95:4.5) (100:3.5) (105:3) (110:2.5) (c1)} [postaction={draw,decorate,decoration={pre length=0.1cm, post length=0.1cm}}];

\draw[red] plot [smooth] coordinates {(c2) (175:1.5) (d1)};
\path plot [smooth] coordinates {(c2) (175:1.5) (d1)} [postaction={draw,decorate,decoration={pre length=0.1cm, post length=0.1cm}}];

\draw[red] (a1) -- (a3) -- (a0) -- (a4) -- (a2);
\path plot coordinates {(a3) (a0) (a4)} [postaction={draw,decorate,decoration={pre length=0.1cm, post length=0.1cm}}];
\draw[red] (b1) -- (b3) -- (b0) -- (b4) -- (b2);
\path plot coordinates {(b3) (b0) (b4)} [postaction={draw,decorate,decoration={pre length=0.1cm, post length=0.1cm}}];
\draw[red] (c1) -- (c3) -- (c0) -- (c4) -- (c2);
\path plot coordinates {(c3) (c0) (c4)} [postaction={draw,decorate,decoration={pre length=0.1cm, post length=0.1cm}}];
\draw[red] (d1) -- (d3) -- (d0) -- (d4) -- (d2);
\path plot coordinates {(d3) (d0) (d4)} [postaction={draw,decorate,decoration={pre length=0.1cm, post length=0.1cm}}];

\draw[gray] (e1) -- (e3) -- (e0);
\path[gray] plot coordinates {(e3) (e0)} [postaction={draw,decorate,decoration={pre length=0.1cm, post length=0.1cm}}];
\draw[gray] (f5) -- (f6) -- (f0);
\path[gray] plot coordinates {(f6) (f0)} [postaction={draw,decorate,decoration={pre length=0.1cm, post length=0.1cm}}];

\draw[red] (f1) -- (f3) -- (f0) -- (f4) -- (f2);
\path plot coordinates {(f3) (f0) (f4)} [postaction={draw,decorate,decoration={pre length=0.1cm, post length=0.1cm}}];
\draw[red] (g1) -- (g3) -- (g0) -- (g4) -- (g2);
\path plot coordinates {(g3) (g0) (g4)} [postaction={draw,decorate,decoration={pre length=0.1cm, post length=0.1cm}}];
\draw[red] (h1) -- (h3) -- (h0) -- (h4) -- (h2);
\path plot coordinates {(h3) (h0) (h4)} [postaction={draw,decorate,decoration={pre length=0.1cm, post length=0.1cm}}];

\draw[red] plot [smooth] coordinates {(d2) (180:1.9) (f1)};
\path plot [smooth] coordinates {(d2) (180:1.9) (f1)} [postaction={draw,decorate,decoration={pre length=0.1cm, post length=0.1cm}}];

\draw[gray] plot [smooth] coordinates {(e1) (200:3) (f5)};
\path[gray] plot [smooth] coordinates {(e1) (200:3) (f5)} [postaction={draw,decorate,decoration={pre length=0.1cm, post length=0.1cm}}];

\draw[red] plot [smooth] coordinates {(f2) (185:1.5) (g1)};
\path plot [smooth] coordinates {(f2) (185:1.5) (g1)} [postaction={draw,decorate,decoration={pre length=0.1cm, post length=0.1cm}}];

\draw[red] plot [smooth] coordinates {(g2) (280:1.5) (h1)};
\path plot [smooth] coordinates {(g2) (280:1.5) (h1)} [postaction={draw,decorate,decoration={pre length=0.1cm, post length=0.1cm}}];

\draw[red] plot [smooth] coordinates {(h2) (300:1.5) (a1)};
\path plot [smooth] coordinates {(h2) (300:1.5) (a1)} [postaction={draw,decorate,decoration={pre length=0.1cm, post length=0.1cm}}];

\draw node[draw=none,fill=none,label=0:\small{$u_{i_{p-1}}$}] at (180:0.4) {};
\draw node[draw=none,fill=none,label=225:\small{$u_{i_p}$}] at (225:4) {};
\draw node[draw=none,fill=none,label=\small{$W_{i_p}$}] at (242:3.3) {};
\draw node[draw=none,fill=none,label=270:\small{$u_{i_{p+1}}$}] at (270:3.7) {};
\draw node[draw=none,fill=none,label=\small{$W_{i_{p+1}}$}] at (282:3.7) {};
\draw node[draw=none,fill=none,label=\footnotesize{$\mathrm{Cen}(W_{i_{p+1}})$}] at (282:4.6) {};

\draw node[draw=none,fill=none,label=135:\small{$u_{i_{99d}}$}] at (135:3.8) {};
\draw node[draw=none,fill=none,label=90:\footnotesize{$u_{i_{99d-1}}$}] at (89:3.6) {};
\draw node[draw=none,fill=none,label=\small{$P_p$}] at (245:1.7) {};
\draw node[draw=none,fill=none,label=\small{$P$}] at (158:3) {};
\draw node[draw=none,fill=none,label=\small{$P_{99d-1}$}] at (86:0.5) {};
\draw node[red,draw=none,fill=none,label=\textcolor{red}{\small{$C_U$}}] at (3:1.0) {};

\draw[rounded corners] (-3.9,-1) rectangle (-3.1,1);
\draw node[draw=none,fill=none,label=\footnotesize{$Z\setminus U$}] at (-3.5,-0.5) {};
\draw[->] (155:3) -- (162:1.85);
\end{tikzpicture}}
\caption{The proof of Theorem~\ref{lem-build-sub}: constructing a cycle $C_U$. Here $p \in [99d]$ is the smallest index such that $W_{i_p}$ is a good web (so the web $W_{i_{p-1}}$ enclosed in a dashed box is bad,~i.e.~its interior is over-used by the paths $P_j$). Paths $P_j$ can intersect the interior of a web but not its centre.}
\label{figjoin}
\end{figure}
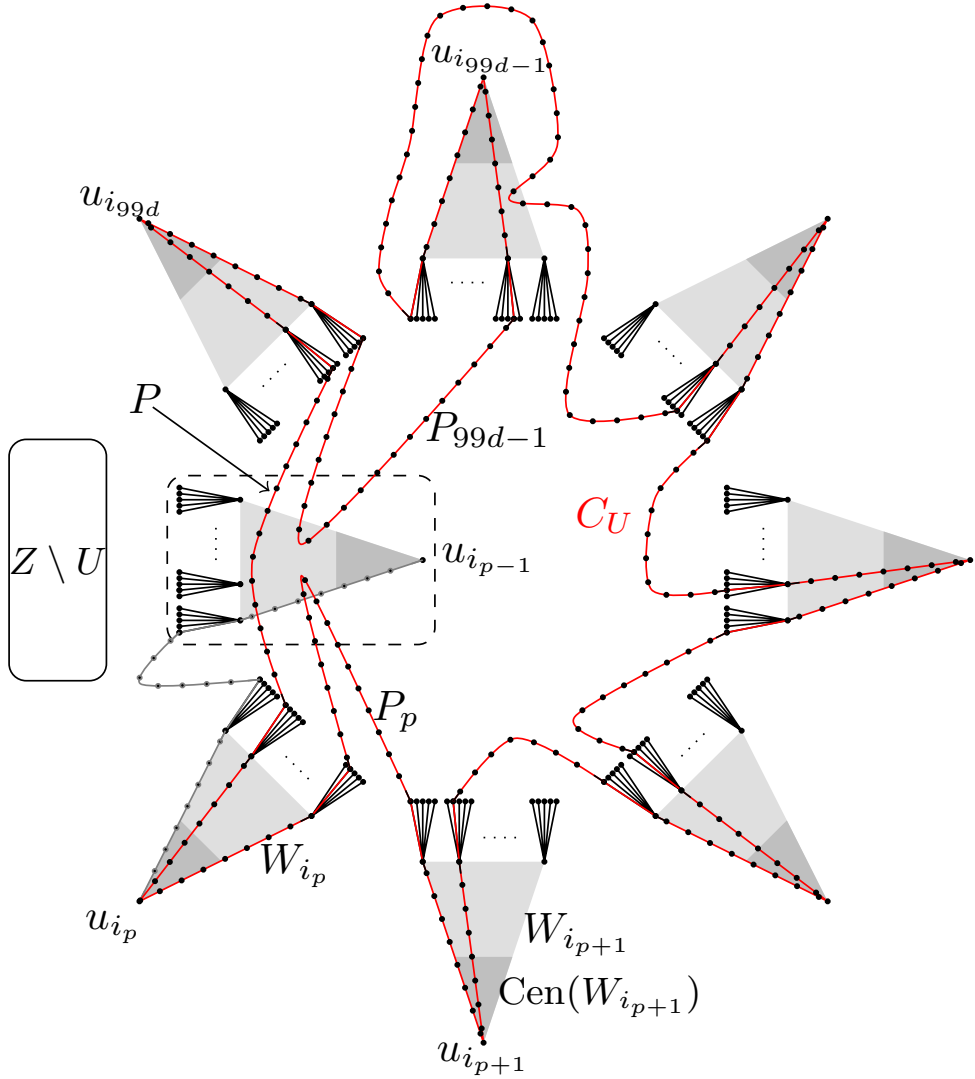

\medskip
\noindent
\emph{Proof of Theorem~\ref{lem-build-sub}.}
For given $\epsilon_1, L$ we choose $d_0, K_0$ so that 
$$0< \frac{1}{d_0}, \frac{1}{K_0} \ll \epsilon_1, \frac{1}{L} \leq 1.$$
As in the rest of the section, we let $m:= \frac{2}{\epsilon_1} \log^{3}(\frac{450n}{d})$, thus (\ref{eq-n-large}) holds.
Since $d\geq d_0, K\geq K_0$, Lemma~\ref{lem-web} implies that we can find in $H$ a collection $W_1,\ldots,W_{200d}$ of $(m^3, m^3, \frac{d}{100},4m)$-webs whose interiors $\Int(W_1),\ldots,\Int(W_{200d})$ are pairwise disjoint. Let $Z := \lbrace u_1,\ldots,u_{200d} \rbrace$ where $u_i$ is the core vertex of $W_i$ for all $i \in [200d]$. Fix an arbitrary $100d$-set $U$ in $Z$. Without loss of generality, assume that $U = \{u_1,\ldots,u_{100d}\}$. First, we show that there exists an index set $I=\{i_1,\ldots,i_{99d}\}\subseteq [100d]$ and a collection $\mathcal{Q}= \{ P_\ell : \ell\in [99d-1]\}$ of paths satisfying the following. For each $\ell\in [99d-1]$,
	\begin{itemize}
		\item[(B1)] $P_\ell$ is a $u_{i_\ell},u_{i_{\ell+1}}$-path of length at most $18m$;
		\item[(B2)] $\Int(P_\ell)$ is disjoint from $\bigcup_{k\in [200d]\setminus \{i_\ell,i_{\ell+1}\} }\Cen(W_{k})\cup Z$;
		\item[(B3)] $\Int(P_\ell)$ and $\Int(P_k)$ are disjoint for all $k \in [99d-1]\setminus \lbrace \ell \rbrace$;
		\item[(B4)] $|\Int(W_{i_{\ell+1}})\cap \bigcup_{k\in[\ell]} V(P_{{k}})| < 2 m^2$.
	\end{itemize}
	
To find such an $(I,\mathcal{Q})$, we will build a path between pairs in $U$ avoiding vertices used in previously-built paths and the centres of all other webs. During the process, we will skip a web if its interior is `over-used'.

Assume we have built $P_{1},\dots, P_{{s}}$ and determined $i_1,\ldots,i_{s+1}$ with $s<99d-1$ satisfying (B1)--(B4). Since index set $\{1\}$ with the empty collection of paths satisfies (B1)--(B4), such a collection $\{P_{1},\dots, P_{{s}}\}$ exists. Let $P':= \bigcup_{k\in[s]} \Int(P_{k})$.
For $i \in [200d]$, we say a web $W_i$ is \emph{bad} if $|\Int(W_i)\cap P'| \geq 2m^2$, and \emph{good} otherwise. Note that (B4) implies that $W_{i_{s+1}}$ is good. By (B1),
\begin{align}\label{eq: P' size 3}
|P'| \leq 18m\cdot s \leq 1800 dm.
\end{align}

\begin{claim}\label{cl: connection}
Let $W_{j_1}$ and $W_{j_2}$ be two good webs. Then there exists a path $P$ of length at most $18m$ in $H$ from $u_{j_1}$ to $u_{j_2}$ such that $\Int(P)$ is disjoint from $P'\cup Z \cup \bigcup_{k\in[200d]\setminus\{j_1,j_2\}}\Cen(W_{k})$.
\end{claim}
\begin{claimproof}
For $k\in [2]$, let $A_k:= \{w\in \Ext(W_{j_k}) : V(P(W_{j_k},w)) \cap P' =\emptyset \}$. Since both $W_{j_1}, W_{j_2}$ are good webs, $P'$ intersects at most $2m^2$ of the paths of $W_{j_k}$ in their interiors. Thus,
$$|A_k|\ge (m^3-2m^2)\cdot m^3\cdot \frac{d}{100}-|P'|\overset{\eqref{eq: P' size 3}}{\ge} \frac{dm^6}{200}-1800dm\ge \frac{dm^6}{300}.$$
Note that $|\Cen(W_k)|\leq 4m(m^3+1)$ for all $k\in [200d]$.
Let
$$
C:= P' \cup \bigcup_{k\in[200d]\setminus\{j_1,j_2\} }\Cen(W_{k})\cup Z.
$$
Thus \eqref{eq-n-large} and \eqref{eq: P' size 3} imply that
\begin{align*}
|C|&\leq |P'|+\sum_{k\in[200d]\setminus\{j_1,j_2\} }|\Cen(W_{k})|+ |Z|\leq 1800 dm + 200d\cdot 4m(m^3+1) + 200 d  \nonumber \\
&\leq 1000 dm^4\leq \frac{1}{4}\cdot\ep\left(\frac{dm^6}{300}\right) \cdot\frac{dm^6}{300}.
\end{align*}
Thus by applying Lemma~\ref{diameter} with $A_1, A_2,C,dm^6/300$ playing the roles of $X,X',W,x$ respectively, there exists a path $Q$ of length at most $m$ in $H$ between $w_1 \in A_1$ and $w_2\in A_2$ avoiding $C$.
Then $Q\cup P(W_{j_1},w_1)\cup P(W_{j_2},w_2)$ contains a $u_{j_1},u_{j_2}$-path $P$ and $$|P|\leq |Q|+|P(W_{j_1},w_1)|+|P(W_{j_2},w_2)|\leq m+ 2(4m+4m+1)\leq 18m.$$ By the choice of $A_1,A_2$ and the fact that the $\Cen(W_k)$ sets are all pairwise disjoint, $\Int(P)$ is disjoint from $P'\cup Z \cup \bigcup_{k\in[200d]\setminus\{j_1,j_2\}}\Cen(W_{k})$.
\end{claimproof}

\medskip
\noindent
Since the interiors of $W_1,\dots, W_{200d}$ are pairwise disjoint, \eqref{eq: P' size 3} implies that the number of webs whose interiors contain at least $m^2$ vertices of $P'$ is at most
 		\begin{eqnarray}\label{eq-bad-webs}
		\frac{1800dm}{m^2}< \frac{d}{2}.
		\end{eqnarray}
Since $s< 99d-1$, we can choose $i_{s+2} \in U \setminus \{i_1,\dots, i_{s+1}\}$ such that
\begin{align}\label{eq: s+1 choice good}
|\Int(W_{i_{s+2}}) \cap P'|\leq m^2.
\end{align}
Recall that $W_{i_{s+1}}$ is good. Thus by Claim~\ref{cl: connection}, there is a $u_{i_{s+1}},u_{i_{s+2}}$-path $P_{s+1}$ of length at most $18m$. Then it is easy to see that $\{i_1,\dots, i_{s+2}\}$ together with $P_1,\ldots,P_{s+1}$ satisfy (B1)--(B3) since $\Int(P_{s+1})$ is disjoint from $P'\cup \bigcup_{k\in[200d]\setminus \{i_\ell,i_{\ell+1}\} }\Cen(W_{k})\cup Z$. Moreover,
\begin{eqnarray*}
\left|\Int(W_{i_{s+2}})\cap \bigcup_{k\in[s+1]} V(P_{{k}})\right| = \left|\Int(W_{i_{s+2}})\cap P'\right|+ |\Int(W_{i_{s+2}})\cap  P_{{s+1}}|\stackrel{\eqref{eq: s+1 choice good}}{\leq} m^2 + 18 m < 2m^2,
\end{eqnarray*}
so (B4) also holds. Therefore, we can repeat this process until $s=99d-1$, upon which we obtain the desired $(I,\mathcal{Q})$ satisfying (B1)--(B4).

Observe that, as before, \eqref{eq-bad-webs} implies that less than $d/2$ indices $k \in [100d]\setminus I$ are such that $W_k$ is bad.
Let $p \in [99d]$ be the minimum index such that $W_{i_p}$ is a good web (see Figure~\ref{figjoin}).
Note that $W_{i_{99d}}$ is good by (B4).
Then $p \leq d/2$ and so $|\lbrace i_p,i_{p+1},\ldots,i_{99d} \rbrace| > 98d$.
By (B1), (B2) and (B3), the concatenation of $P_{p},P_{p+1},\ldots,P_{99d-1}$ is a $u_{i_p},u_{i_{99d}}$-path avoiding $Z\setminus U$.
By Claim~\ref{cl: connection}, there exists a $u_{i_p},u_{i_{99d}}$-path $P$ of length at most $18m$ such that $\Int(P)$ is disjoint from $\bigcup_{k\in[99d-1]}\Int(P_k)\cup Z \cup \bigcup_{k\in[200d]\setminus\{i_p,i_{99d}\}}\Cen(W_{k})$. 
Thus, the concatenation of $P_p,P_{p+1},\ldots,P_{99d-1},P$ form a cycle $C_U$, as in Figure~\ref{figjoin}.
Finally, by (B1), (B2) and Claim~\ref{cl: connection},
$$
V(C_U) \cap Z = \lbrace u_{i_p},u_{i_{p+1}},\ldots,u_{i_{99d}} \rbrace  \in \binom{U}{\geq 98d},
$$
completing the proof of the theorem.
\hfill$\square$

\medskip
The proof of Lemma~\ref{lem-sparse} in the case when $d \geq \log^{100}n$ now follows easily since there are not many distinct $U,U'$ such that $C_U \cap Z$ and $C_{U'} \cap Z$ are equal. That is, the cycles $C_U$ are `almost distinguishable' by their intersection with $Z$.

\begin{proof}[Proof of Theorem~\ref{lem-build-sub} $\Rightarrow$ Lemma~\ref{lem-sparse} when $d\ge\log^{100}n$.]

Apply Theorem~\ref{lem-build-sub} to obtain a set $Z \subseteq V(H)$ of size $200d$ such that, for every subset $U \subseteq Z$ of size $100d$, there exists a cycle $C_U$ with $V(C_U) \cap Z \in \binom{U}{\geq 98d}$.	
	
	Fix an arbitrary cycle $C$ in $H$ such that $V(C)\cap Z$ has size $98d\le |V(C)\cap Z|\le 100d$. There are at most ${|Z|-|V(C)\cap Z|\choose 100d-|V(C)\cap Z|}$ ways to choose a $100d$-set $U\subseteq Z$ containing $V(C)\cap Z$. 
	In other words, for a fixed cycle $C$ in $H$,
$$|\{U \subseteq Z: V(C_U) = V(C)\}| \leq \binom{|Z|-|V(C)\cap Z|}{100d-|V(C)\cap Z|} \leq \binom{102d}{2d}.$$	
Therefore the number of Hamiltonian subsets in $H$ is
	$$ c(H)\geq \frac{ \binom{|Z|}{100d}}{ \binom{102d}{2d}} \geq 2^{50d},$$
	as desired.
\end{proof}



\subsection{Proof of Lemma~\ref{lem-sparse} when $d\le \log^{100}n$}\label{sec-sparse}

In this section, we finish the proof of Lemma~\ref{lem-sparse} by proving it under the additional assumption that $d\leq \log^{100} n$. The proof will follow easily from the next result.

\begin{lemma}\label{lemma-sparse-sub}
Let $0< 1/d \ll \ep_1,1/L\leq 1$. Suppose that $H$ is an $n$-vertex $(\ep_1, d/30)$-expander with $d\le \log^{100}n$ and $d/10\le \de(H)\le \De(H)\le Ld$.
Then $H$ contains a set $Z$ of size $200d$ such that for every subset $U\subseteq Z$ of size $100d$, there exists a cycle $C_U$ with $V(C_U)\cap Z=U$.
\end{lemma}

First let us see how this implies Lemma~\ref{lem-sparse}.
Let $Z$ be the set guaranteed in Lemma~\ref{lemma-sparse-sub}. Then for distinct $100d$-sets $U, U'\subseteq Z$, their corresponding cycles $C_U$ and $C_{U'}$ can be distinguished in $Z$, implying that $c(H)\ge {200d\choose 100d}>2^{50d}$ as desired. 

The idea of the proof of Lemma~\ref{lemma-sparse-sub} is as follows.
Since $H$ has small maximum degree, we can choose $Z$ to be a set of $200d$ vertices which are very far apart in $H$. Let $U := \lbrace u_1,\ldots,u_{100d} \rbrace \subseteq Z$.
We will obtain the desired cycle $C_U$ via an iterative procedure.
For each $i \in [100d]$ we will obtain a path $P_i$ between $u_i$ and $u_{i+1}$ (where $u_{100d+1}:=u_1$) which avoids the rest of $Z$ and the interiors of previously-built paths.
The concatenation of these paths is the desired cycle $C_U$.
To enable the iteration to be completed, we ensure that there is a large set of vertices within a short distance of $u_i$, and similarly for $u_{i+1}$, which have not been used in previously-built paths.
Lemma~\ref{diameter} implies that there is a short path $P_i$ between these two large sets which avoids existing paths.
We then extend $P_i$ into a $u_i,u_{i+1}$-path, as required.


To prove Lemma~\ref{lemma-sparse-sub}, we need the following result from~\cite{L-M} (Lemma~5.5 when taking $s=5$ and $q=1$).
To state it we need some notation.
Given an integer $r \geq 0$, graph $G$ and $X \subseteq V(G)$, define $B_G^r(X)$ to be the \emph{ball of radius $r$ around $X$}; that is, the set of vertices at distance at most $r$ from $X$ in $G$.
For $r \geq 1$, define $\Gamma^{r}_G(X)$ by setting  
$$
\Gamma^r_G(X) := B^r_G(X)\setminus B^{r-1}_G(X).
$$
So $\Gamma^r_G(X)$ is the $r^{\mathrm{th}}$ sphere around $X$, i.e.~the set of vertices at distance exactly $r$ from $X$.
If $v \in V(G)$, we write $\Gamma^r_G(v)$ as shorthand for $\Gamma^r_G(\lbrace v \rbrace)$ and similarly abbreviate $B^r_G(\lbrace v \rbrace)$.

\begin{lemma}[\cite{L-M}]\label{corner} 
Let $0< 1/d \ll \ep_1\leq 1$ and $n\in \mathbb{N}$ be such that $d\leq \log^{100}n$, and set $r:=(\log\log n)^5$. Let $H$ be an $n$-vertex $(\ep_1,d/30)$-expander with $\de(H)\geq d/10$. 
Let $v \in V(H)$ and $u \in B^r_H(v)$, and let $P$ be a shortest $u,v$-path in $B^r_H(v)$.
 Then $|B^{r}_{H'}(v)|\geq d\log^7 n$, where $H' := H-(V(P)\setminus \lbrace v \rbrace)$.
\end{lemma}

\begin{figure}
\tikzstyle{every node}=[circle, draw, fill=black,
inner sep=0pt, minimum width=5pt]
\begin{tikzpicture}[decoration={markings,
	mark=between positions 0 and 1 step 5mm
	with { \draw node[minimum width=2pt] {}; }} ]

\clip (-8,-0.5) rectangle (8.2,8);
\begin{scope}[rotate=20]

\begin{scope}[shift=(0:6),rotate=0]
\filldraw[draw=none,fill=gray!25] (0,0) circle (2cm);
\filldraw[draw=none,fill=gray!50] (0,0) circle (1cm);
\draw node at (100:2) (a1) {};
\draw node at (210:2) (a2) {};
\draw node at (0,0) (a0) {};
\draw[<->] (-90:0.2) -- (-90:1);
\draw node[draw=none,fill=none,label=-90:$r$] at (-90:1) {};
\draw[<->] (45:0.2) -- (45:2);
\draw node[draw=none,fill=none,label=45:$k$] at (45:2) {};
\draw node[draw=none,fill=none,label=right:$Q_i^i$] at (100:1.5) {};
\end{scope}

\begin{scope}[shift=(72:6),rotate=72]
\filldraw[draw=none,fill=gray!25] (0,0) circle (2cm);
\filldraw[draw=none,fill=gray!50] (0,0) circle (1cm);
\draw node at (150:2) (b1) {};
\draw node at (230:2) (b2) {};
\draw node at (0,0) (b0) {};
\draw node[draw=none,fill=none,label=right:$Q_i^{i+1}$] at (190:1.2) {};
\draw node[draw=none,fill=none,label=left:$Q_{i+1}^{i+1}$] at (135:1.0) {};
\end{scope}

\begin{scope}[shift=(144:6),rotate=144]
\filldraw[draw=none,fill=gray!25] (0,0) circle (2cm);
\filldraw[draw=none,fill=gray!50] (0,0) circle (1cm);
\draw node at (150:2) (c1) {};
\draw node at (210:2) (c2) {};
\draw node at (0,0) (c0) {};
\end{scope}

\begin{scope}[shift=(216:6),rotate=216]
\filldraw[draw=none,fill=gray!25] (0,0) circle (2cm);
\filldraw[draw=none,fill=gray!50] (0,0) circle (1cm);
\draw node at (150:2) (d1) {};
\draw node at (210:2) (d2) {};
\draw node at (0,0) (d0) {};
\end{scope}

\begin{scope}[shift=(-72:6),rotate=-72]
\filldraw[draw=none,fill=gray!25] (0,0) circle (2cm);
\filldraw[draw=none,fill=gray!50] (0,0) circle (1cm);
\draw node at (150:2) (e1) {};
\draw node at (210:2) (e2) {};
\draw node at (0,0) (e0) {};
\end{scope}

\draw[red,thick] plot [smooth] coordinates {(a1) (30:4.2) (50:5.2) (b2) (b0) (b1) (70:3) (50:3.5) (20:3) (7:4.5) (0:4.5) (5:3) (40:1) (90:2.0) (c2) (c0) (c1) (168:3.2) (188:4.2) (d2) (d0) (d1) (240:3.2) (260:4.2) (e2) (e0) (e1) (-40:5) (-20:4.2) (a2) (a0) (a1)};
\path plot [smooth] coordinates {(a1) (30:4.2) (50:5.2) (b2) (b0) (b1) (70:3) (50:3.5) (20:3) (7:4.5) (0:4.5) (5:3) (40:1) (90:2.0) (c2) (c0) (c1) (168:3.2) (188:4.2) (d2) (d0) (d1) (240:3.2) (260:4.2) (e2) (e0) (e1) (-40:5) (-20:4.2) (a2) (a0) (a1)} [postaction={draw,decorate,decoration={pre length=0.1cm, post length=0.1cm}}];

\begin{scope}[shift=(0:6),rotate=0]
\draw node at (100:2) {};
\draw node at (210:2) {};
\draw node[label=0:$v_i$] at (0,0) {};
\draw node[label=above:$u_i$] at (100:2) {};
\end{scope}

\begin{scope}[shift=(72:6),rotate=72]
\draw node at (150:2) {};
\draw node at (230:2) {};
\draw node[label=72:$v_{i+1}$] at (0,0) {};
\draw node[label=below:$u_{i+1}$] at (230:2) {};
\end{scope}

\begin{scope}[shift=(144:6),rotate=144]
\draw node at (150:2) {};
\draw node at (210:2) {};
\draw node[label=144:$v_{i+2}$] at (0,0) {};
\end{scope}

\begin{scope}[shift=(216:6),rotate=216]
\draw node at (150:2) {};
\draw node at (210:2) {};
\draw node[label=216:$v_i$] at (0,0) {};
\end{scope}

\begin{scope}[shift=(-72:6),rotate=-72]
\draw node at (150:2) {};
\draw node at (210:2) {};
\draw node[label=-72:$v_i$] at (0,0) {};
\end{scope}

\draw node[draw=none,fill=none,label=35:$Q$] at (35:4.2) {};
\end{scope}

\end{tikzpicture}
\caption{The proof of Lemma~\ref{lemma-sparse-sub}: the construction of $P_i = Q_i^i \cup Q \cup Q_i^{i+1}$. The red line is a segment of $C_U$, the concatenation of $P_1,\ldots,P_{100d}$.}
\label{figsparse}
\end{figure}

\begin{proof}[Proof of Lemma~\ref{lemma-sparse-sub}] 
Since $1/d \ll 1$ and $d \leq \log^{100}n$, we have $1/n \ll 1$.
	Let
	$$ r:=(\log\log n)^5 \quad \text{ and }\quad k:=(\log n)^{7/8}.
	$$
	Note that
	\begin{equation}\label{eq: Delta size}
	\De(H)\le Ld\le d^2\leq \log^{200}n.
	\end{equation}
	We first show that there exists a set of $200d$ vertices $Z$ whose pairwise distances are all greater than $2k$.
In other words, for all distinct $z,z' \in Z$, we have that $B^k_H(z) \cap B^k_H(z') = \emptyset$.	
	Indeed, given such a set $Z$ with $|Z| < 200d$, by~(\ref{eq: Delta size}) and the fact that $1/n \ll 1$,
	\begin{eqnarray*}
		|B^{2k}(Z)|\leq |Z|\cdot \sum_{0 \leq i \leq 2k}\Delta(H)^i < 2|Z|\cdot \De(H)^{2k}< 400d\cdot (\log^{200}n)^{2k} <n.
	\end{eqnarray*}
	So we can add another vertex in $V(G)\setminus B^{2k}(Z)$ that is at distance greater than $2k$ from $Z$. Thus such a set $Z$ with $|Z| = 200d$ exists.

By applying Lemma~\ref{corner} with the empty path from $v$ to $v$ playing the role of  $P$, for every $z\in Z$ we have that
\begin{align}\label{eq: BrHvp big}
|B^r_{H}(z)| \geq d\log^7n.
\end{align}
	
	Fix an arbitrary $100d$-set $U\subseteq Z$, say $U=\{v_1,\ldots, v_{100d}\}$, and let $v_{100d+1}:= v_1$ and $v_0:=v_{100d}$.
Let $I  \subseteq [100d]$ be a maximal subset such that there exists a collection $\mathcal{Q}=\{P_j: j\in I \}$ satisfying the following properties.
	\begin{itemize}
		\item[(E1)] For each $j\in I $, $P_{j}$ is a $v_j,v_{j+1}$-path with length at most $2\log^4 n$;
		\item[(E2)] the paths in $\mathcal{Q}$ are pairwise internally vertex-disjoint;
		\item[(E3)] for every $j\in I $, the path $P_{j}$ is disjoint from $\bigcup_{p\in[100d]\setminus \{j,j+1\}}B^r_H(v_p)\cup (Z\setminus U)$;
		\item[(E4)] every $j\in I $ such that $j+1 \notin I $ satisfies $|B^r_{H-(V(P_{j})\setminus \{v_{j+1}\})}(v_{j+1})|\geq d\log^{7} n$;
		\item[(E5)] every $j\in I $ such that $j-1 \notin I $ satisfies $|B^r_{H-(V(P_{j})\setminus \{v_{j}\})}(v_{j})|\geq d\log^{7} n$.
	\end{itemize}

Note that such a maximal subset exists as by~\eqref{eq: BrHvp big}, we have that $(I ,\mathcal{Q})=(\{1\},\emptyset)$ satisfies (E1)--(E5). We show next that $I=[100d]$.
Indeed, suppose that there exists $i\in [100d]\setminus I $. 
Let 
$$P':= \bigcup_{j\in I }\Int(P_j),\quad W':=\bigcup_{p\in[100d] \setminus\{i,i+1\}} B^r_H(v_p), \quad\text{ and }\quad W:= P'\cup W'\cup (Z\setminus U).$$
We would like to estimate the sizes of $P', W'$ and $W$.
Firstly,
\begin{align}\label{eq: P' size P'}
|P'| \stackrel{({\rm E1})}{\leq} |I | \cdot 2\log^{4}n < 100d \cdot 2 \log^{4}n \leq d \log^{5} n.
\end{align}
By~(\ref{eq: Delta size}) and since $1/n \ll 1$,
\begin{align*}
|W'| \leq 100d\cdot  \sum_{0 \leq \ell \leq r} (\log^{200}n)^\ell \le 100\log^{100}n\cdot 2\log^{200r}n\leq \log^{300r}n = e^{300(\log \log n)^6} \leq e^{\log^{\frac{1}{5}}n}.
\end{align*}
This implies that 
\begin{eqnarray}\label{eq: W size W}
|W|\leq |P'|+|W'|+|Z\setminus U| \leq d\log^{5} n + e^{\log^{\frac{1}{5}}n} + 100d  < \frac{1}{4} \ep(e^{\log^{\frac{1}{4}}n})e^{\log^{\frac{1}{4}}n}.
\end{eqnarray}
Note that (E3) implies that
\begin{align}\label{eq-E4}
B^r_{H-P'}(v_i) &= \begin{cases} B^r_{H}(v_i) &\mbox{if } i-1\notin I \\
B^r_{H-(V(P_{i-1})\setminus \{v_{i}\})}(v_i) & \mbox{if } i-1\in I \end{cases}\quad\text{and}\\ 
\nonumber B^r_{H-P'}(v_{i+1}) &= \begin{cases} B^r_{H}(v_{i+1}) &\mbox{if } i+1\notin I \\
B^r_{H-(V(P_{i+1})\setminus \{v_{i+1}\})}(v_{i+1}) & \mbox{if } i+1\in I. \end{cases} 
\end{align}
Now we prove the following claim.
\begin{claim}
For $p\in \{i,i+1\}$, we have that
\begin{itemize}
\item[(i)] $|B^r_{H-P'}(v_p)|\geq d\log^7 n$;
\item[(ii)] $|B^{k}_{H-P'}(v_p)|\geq e^{\log^{1/4} n}$.
\end{itemize}
	\end{claim}
	\begin{claimproof}
Assertion (i) is a simple consequence of~(\ref{eq: BrHvp big}), (E4) and (E5).
Indeed, (E4) and (E5) imply that, if $i-1 \in I $, then $|B^r_{H-P'}(v_i)|\geq d\log^7 n$; and if $i+1 \in I $, then $|B^r_{H-P'}(v_{i+1})|\geq d\log^7 n$.
If $i-1 \notin I $, then $|B^r_{H-P'}(v_i)| = |B^r_H(v_i)| \geq d\log^7n$ by~(\ref{eq: BrHvp big}) and~(\ref{eq-E4}), and similarly if $i+1\notin I $, then $|B^r_{H-P'}(v_p)| \geq d\log^7n$.
This proves (i).

For (ii), let $X$ be a set with $d\log^7 n \leq |X| \leq e^{\log^{1/4} n}$. Then, since $\epsilon(x)$ is decreasing,
\begin{align}\label{eq-epsilon}
\eps(|X|) \geq \ep(e^{\log^{1/4} n}) =\frac{\ep_1}{\log^2{\frac{450e^{\log^{1/4} n}}{d}}}  > \frac{\eps_1}{\log^{1/2}{n}}.
\end{align}
Since $H$ is an $(\ep_1,d/30)$-expander and $|X|\geq d\log^{7}n$, if $X\cap P'=\emptyset$, then
\begin{eqnarray}\label{eq: N expansion}
|\Gamma_{H-P'}(X)| \geq |\Gamma_{H}(X)|- |P'| \geq \ep(|X|)|X| - |P'| 
\stackrel{\eqref{eq: P' size P'},\eqref{eq-epsilon}}{\geq} \frac{\eps_1}{2\log^{1/2}{n}}|X|.
\end{eqnarray}

Note that for $\ell \geq r$, (i) implies that $|B^\ell_{H-P'}(v_p)|\geq |B^r_{H-P'}(v_p)| \geq 	d\log^{7}n$.		
Thus, if $|B^\ell_{H-P'}(v_p)|<e^{\log^{1/4} n}$, then, as $B^\ell_{H-P'}(v_p)\cap P'=\emptyset$, applying~\eqref{eq: N expansion} with $X=B^\ell_{H-P'}(v_p)$ we have
$$|B^{\ell+1}_{H-P'}(v_p)| = |B^{\ell}_{H-P'}(v_p)|+ |\Gamma_{H-P'}(B^{\ell}_{H-P'}(v_p))| \stackrel{\eqref{eq: N expansion}}{\geq} \left(1+ \frac{\eps_1}{2\log^{1/2}n}\right)|B^{\ell}_{H-P'}(v_p)|.$$
Thus, if $|B^{\ell}_{H-P'}(v_p)| < e^{\log^{1/4}n}$ for all $\ell<k$, then  
\begin{align*}
\log|B^{k}_{H-P'}(v_p)| &\geq  \log\left(1+ \frac{\eps_1}{2\log^{1/2}n}\right)^{k-r} + \log|B^{r}_{H-P'}(v_p)| \geq \frac{\eps_1(k-r)}{4\log^{1/2}n}> \log^{1/4}n,
\end{align*}		
completing the proof of (ii).
	\end{claimproof}
	
	\medskip
Observe that, for $p\in \lbrace i,i+1\rbrace$, as all the balls $B^k_H(z)$ with $z\in Z$ are pairwise disjoint, the definitions of $W'$ and $W$ imply that
\begin{eqnarray}\label{eq-H-W}
|B^k_{H-W}(v_p)| = |B^k_{H-P'}(v_p)| \geq e^{\log^{1/4}n}.
\end{eqnarray}
Together with \eqref{eq: W size W}, this allows us to apply Lemma~\ref{diameter} with $B^{k}_{H-W}(v_i), B^{k}_{H-W}(v_{i+1}),W$ and $e^{\log^{1/4} n}$ playing the roles of $X,X',W$ and $x$ respectively to show the existence of a path of length $\log^{4}n$ between $B^{k}_{H-W}(v_i)$ and $B^{k}_{H-W}(v_{i+1})$ in $H-W$.
Let $Q$ be a shortest such path, $u_p:=B^{k}_{H-W}(v_p)\cap Q$ and $Q_i^p$ be a shortest path from $v_p$ to $u_p$ in $H-W$ for $p\in\{i,i+1\}$ (see Figure~\ref{figsparse}).
Now define
$$
P_i:= Q_i^i \cup Q\cup Q_i^{i+1}, \quad I ':=I \cup \{i\} \quad\text{ and }\quad\mathcal{Q}':=\mathcal{Q}\cup \{P_i\}.
$$

Note that $|P_i| \leq |Q| + 2k \leq 2 \log^{4}n$, thus $\mathcal{Q}'$ satisfies (E1). Since $P_i\cap W =\emptyset$, $\mathcal{Q}'$ satisfies (E2) and (E3). For (E4), set $P''=P'\cup(P_i\setminus\{v_{i+1}\})$. If $i+1\notin I '$, then $i+1\notin I$, so~\eqref{eq-E4} and~\eqref{eq-H-W} imply that $B_{H-W}^r(v_{i+1})=B^r_{H-P'}(v_{i+1})=B_{H}^r(v_{i+1})$. This, together with the fact that $Q_i^i\cup Q$ is disjoint from $B_{H-W}^r(v_{i+1})$, implies that $$B^r_{H-P''}(v_{i+1})=B^r_{H-(V(P_i)\setminus\{v_{i+1}\})}(v_{i+1})=B^r_{H-(V(Q_i^{i+1})\setminus\{v_{i+1}\})}(v_{i+1}).$$
Since $Q_i^{i+1}$ is a shortest path between $v_{i+1}$ and $u_{i+1}$ in $H-W$, it is also shortest in $H$ by the definition of $W$. Thus Lemma~\ref{corner} with $v_{i+1},Q_{i}^{i+1}$ playing the roles of $v$ and $P$ implies that (E4) for $j=i$.
For $j \in I'\setminus \lbrace i \rbrace$, (E4) holds since $\mathcal{Q}$ satisfies (E4) and $\mathcal{Q}'$ satisfies (E3).
Similarly, (E5) is also satisfied. This contradicts the maximality of $I $. Thus $I =[100d]$.

We claim that the concatenation of $P_1,\ldots,P_{100d}$ forms the desired cycle $C_U$.
That $C_U$ is a cycle follows from (E1) and (E2).
Furthermore, $C_U$ contains no vertices in $Z\setminus U$ by (E3), and $U \subseteq V(C_U)$ by (E1).
Thus $V(C_U)\cap Z = U$, as required.
This completes the proof of the lemma, and hence the proof of Lemma~\ref{lem-sparse}.
\end{proof}


\section{The proof of Theorem~\ref{main}}\label{end}

We will need the following definitions.

\begin{definition}[cut vertex, block, block graph, leaf block]
	A \emph{cut vertex} $x$ of a graph $G$ is such that $G-x$ has more components than $G$.
	A \emph{block} of $G$ is a maximal subgraph $H \subseteq G$ such that no vertex of $H$ is a cut vertex of $H$.
	The blocks of $G$ form a partition of $E(G)$, and any two blocks have at most one vertex in common. Note that a block is an induced subgraph.
	The \emph{block graph} $BL(G)$ is the graph whose vertex set is the set of blocks of $G$, where two blocks are joined by an edge in $BL(G)$ whenever they share a vertex in $G$.
	A \emph{leaf block} of $G$ is a block which contains exactly one cut vertex of $G$ (so it has degree one in $BL(G)$).
	Since $BL(G)$ is always a forest, there are at least two leaf blocks of $G$ unless $G$ is $2$-connected. Note also that every cut vertex has degree at least two in every block containing it, when that block is not a single edge.
\end{definition}

Now we start the main proof.

\medskip
\noindent
\emph{Proof of Theorem~\ref{main}.}
Let $0<\alpha \leq 1$. We choose $\ep_1,K, L$ so that
\begin{equation}\label{hierarchy}
0 < \frac{1}{d_0} \ll \frac{1}{K} \ll \frac{1}{L} \ll \eps_1,\alpha \leq 1\quad\text{where}\quad\eps_1 \leq \frac{1}{130}.
\end{equation}



Now, let $d \geq d_0$ be an integer.
Assume that $G$ is a counterexample to Theorem~\ref{main} such that $n := |G|$ is minimal. Then $d(G)\geq d$; $G\notin \{K_{d+1}, K_{d+1}\ast K_d\}$ and $c(G) < (2-\alpha)2^{d+1}$.
Since $K_{d+1}$ is the only graph with average degree at least $d$ on $d+1$ vertices, we have $n \geq d+2$.

\begin{claim}\label{cl: proper subgraph}
For any $V'\subsetneq V(G)$, we have $d(G[V']) \leq d$.
Also we have $\delta(G)\ge d/2$.
\end{claim}
\begin{claimproof}
Suppose that there is some $V' \subsetneq V(G)$ for which $d(G[V'])>d$. Then $G[V']\notin \{K_{d+1}, K_{d+1}\ast K_{d}\}$ since both $K_{d+1}$ and $K_{d+1}\ast K_{d}$ have average degree exactly $d$. Thus the minimality of $G$ implies that $G[V']$ is not a counterexample to Theorem~\ref{main}. Then $c(G)\ge c(G[V']) \geq (2-\alpha) 2^{d+1}$, a contradiction.

Assume now that there exists a vertex $x\in V(G)$ such that $d_{G}(x) < d/2$. Then
\begin{eqnarray*}
d(G-x) = \frac{2e(G)-2d_G(x)}{n-1} > \frac{dn - d}{n-1} =  d,
\end{eqnarray*}
a contradiction to the first part of the claim.
\end{claimproof}

\medskip
\noindent
Next, we show that $G$ contains a large $2$-connected subgraph which essentially inherits the average degree of $G$.
To see this, we need the following claim.

\begin{claim}\label{cl: leaf-block}
If $G$ is not $2$-connected, then for any leaf block $F\subseteq G$, the following hold.
\begin{itemize}
	\item[(i)]  $\de_2(F)\ge d/2$; $|F|\ge d$ and $d(F)\ge d-1$; 
	\item[(ii)] If in addition $d+1\le |F|\le 1.19d$, then $c(F)\geq (1-\alpha/2) 2^{d+1}.$ 
\end{itemize}
\end{claim}
\begin{claimproof}
	Fix a leaf block $F\subsetneq G$ and let $x$ be the corresponding cut vertex, i.e.~the only vertex in $F$ which has neighbours outside of $F$. Note that $\de_2(F)\geq \de(G)\ge d/2$ by Claim~\ref{cl: proper subgraph}.
	
	Let $G':= G - (V(F)\setminus\{x\})$. Let $n':=|G'|$.
	Then we have
	$$ n = |F|+n'-1, \quad \mbox{and}\quad  e(G) = e(F) + e(G').$$
	By Claim~\ref{cl: proper subgraph}, we have $d(F)\leq d$ and $d(G') \leq d$.
	Thus, 
	$$e(F) = e(G)-e(G') \ge dn/2 - dn'/2 = d(|F|-1)/2.$$ 
	Since every vertex of $F$ has at most $|F|-1$ neighbours in $F$, this implies that
	$$|F|-1\ge d(F) \geq \frac{d(|F| -1)}{|F|} = d - \frac{d}{|F|}.$$
	Rearranging this, we see that
	\begin{eqnarray}\label{eqF}
	|F|\geq d, \quad \mbox{and}\quad  d(F)\geq d-1,
	\end{eqnarray}
	proving (i).
	
	To prove (ii),
	suppose now that $d+1\le |F|\le 1.19d$. Since $x$ is a cut vertex, we have
	\begin{eqnarray}\label{eq-min-deg}
	\de(F)\geq\min\{\de(G),d_{F}(x)\} \quad\mbox{and} \quad \de(F-x)\ge \de(G)-1\ge (d-2)/2
	\end{eqnarray}
	by Claim~\ref{cl: proper subgraph}.
	If $d_{F}(x) \geq (d-1)/2$, then~\eqref{eq-min-deg} implies that $\de(F)\ge (d-1)/2$. Since $1/(d-1)\ll \alpha$ and $|F| \leq 1.19d \leq 1.2(d-1)$, we can apply Lemma~\ref{lem: small graph} to $F$ with $d-1$ playing the role of $d$ to see that $c(F)\geq  (1-\alpha/2)2^{|F|} \geq (1-\alpha/2)2^{d+1}$ as desired.
	
	So we may assume that $d_{F}(x) < (d-1)/2$. We claim that $|F|\geq d+2$. Indeed, if this is not so, then
	$$(d-1)|F|\stackrel{(\ref{eqF})}{\leq} 2e(F) \leq 2\binom{|F|-1}{2}+ 2 d_{F}(x) < (|F|-1)(|F|-2) + (d-1) \leq (d-1)|F|,$$
	a contradiction. Thus, $|F-x| \geq d+1$. Note that
	$$d(F-x)=\frac{2e(F)-2d_{F}(x)}{|F|-1}\stackrel{(\ref{eqF})}{>}\frac{(d-1)|F|-(d-1)}{|F|-1}=d-1.$$
	Then by~\eqref{eq-min-deg} and the fact that $1/(d-2)\ll \alpha$ and $|F|-1 \leq 1.19d-1 \leq 1.2(d-2)$, we can apply Lemma~\ref{lem: small graph} to $F-x$ with $d-2$ playing the role of $d$ to get $c(F)\geq c(F-x) \geq (1-\alpha/2)2^{d+1}$. 
\end{claimproof}

\medskip
\noindent
The next claim guarantees a large $2$-connected subgraph $G_1$ of $G$, which still has average degree at least roughly $d$.

\begin{claim}\label{cl: G1-2conn}
There exists $G_1\subseteq G$ such that
\begin{itemize}
	\item[(i)] $G_1$ is $2$-connected;
    \item[(ii)] $d-1\le d(G_1)\le 4d$;
	\item[(iii)] $\de_2(G_1)\ge d/2$;
	\item[(iv)] $|G_1|\ge K^2d$.
\end{itemize}
\end{claim}
\begin{claimproof}
Suppose first that $G$ is $2$-connected. Then we set $G_1:=G$. Then $d(G_1)\le 4d$, since otherwise Theorem~\ref{Tuza} implies $c(G)\ge 2^{2d}$, a contradiction. By Claim~\ref{cl: proper subgraph}, $\de(G_1)\ge d/2$. We are left to show that $|G_1|=n\ge K^2d$. Recall that $n\ge d+2$. If $n\le 1.2d$, then by Lemma~\ref{lem: small graph}, $c(G)\ge (1-\al/2)2^{n}\ge (2-\al)2^{d+1}$, a contradiction. Thus we can assume that $1.2d\le n< K^2d$. Applying Lemma~\ref{lem: 1.2 to K} with $G,d$ and $K^2$ playing the roles of $G,d$ and $K$ respectively, we have that $c(G)\ge 2^{(1+1/200)d}\ge (2-\al)2^{d+1}$, a contradiction.

We may therefore assume that $G$ is not $2$-connected. Then $G$ contains at least two leaf blocks $G_1,G_2$, where $|G_1|\geq |G_2|$. Then by Claims~\ref{cl: proper subgraph} and~\ref{cl: leaf-block}(i), for $i\in[2]$ we have that 
\begin{align}\label{eq: G1G2 property}
\de_2(G_i)\ge d/2; \quad |G_i|\ge d \quad\text{ and }\quad d-1\le d(G_i)\le d.
\end{align} 
We may assume that 
$$d\le |G_2|\le |G_1|< K^2d,$$ otherwise $G_1$ is the desired subgraph, as above. We distinguish the following three cases.

\medskip
\noindent \textbf{Case 1:} $|G_1|\ge 1.19d$. 

\noindent
Then applying Lemma~\ref{lem: 1.2 to K} to $G_1$ with $d-1$ and $K^2$ playing the roles of $d$ and $K$ respectively, we get $c(G_1)\ge 2^{(1+1/200)(d-1)}\ge (2-\al)2^{d+1}$, a contradiction. 

\medskip
\noindent \textbf{Case 2:} $d+1\leq |G_2|\le|G_1|\leq 1.19d$. 

\noindent
Then Claim~\ref{cl: leaf-block}(ii) implies that for $i\in [2]$, $c(G_i) \geq (1-\alpha/2) 2^{d+1}$. Thus $c(G)\geq c(G_1)+ c(G_2) \geq (2-\alpha) 2^{d+1}$, a contradiction.

\medskip
\noindent \textbf{Case 3:} $|G_2|=d$.

\noindent
Note that since $d(G_2)\ge d-1$ from \eqref{eq: G1G2 property}, $G_2$ is isomorphic to $K_{d}$. Let $x_2$ be the cut vertex of $G$ in $G_2$. Consider the graph $G^*:= G  - (V(G_2)\setminus\{x_2\})$. 
Then
$$d(G^*)=\frac{2e(G^*)}{|G^*|}=\frac{2e(G)-2e(G_2)}{n-d+1}\ge\frac{dn-d(d-1)}{n-d+1}=d.$$ 
Since $c(G^*)\le c(G)<(2-\al)2^{d+1}$, the minimality of $G$ implies that $G^*$ is isomorphic to either $K_{d+1}$ or $K_{d+1}\ast K_{d}$. If $G^*\cong K_{d+1}$, then $G = K_{d+1}\ast K_{d}$, a contradiction. If $G^*\cong K_{d+1}\ast K_{d}$, then 
$$c(G)= c(G_2) + c(G^*) = 2c(K_d) + c(K_{d+1})\geq (2-\alpha) 2^{d+1},$$
a contradiction.	
\end{claimproof}

\medskip
\noindent
Let $G_1\subseteq G$ be the subgraph guaranteed by Claim~\ref{cl: G1-2conn}. 
The following claim states that, if $H$ is any subgraph of $G_1$ with large average degree, then the average degree remains large after removing any small subset of vertices $U$.
The claim will then be used to find expander subgraphs of $G_1$ which are almost regular.

\begin{claim}\label{cl: vertex deletion fine}
Suppose that $U$ is a subset of $V(G_1)$ such that $|U| \leq 10|G_1|/L$. Then $d(G_1 - U) \geq 0.9(d-1)$.
\end{claim}
\begin{claimproof}
We may assume that $|U| = 10|G_1|/L$. 
(Indeed, if not, we can take $10|G_1|/L-|U|$ vertices in $V(G_1)\setminus U$ of largest degree and add them to $U$ to obtain $U'$ such that $d(G_1-U) \geq d(G_1-U')$.)
Let $B:= V(G_1)\setminus U$ and suppose to the contrary that $d(G_1[B]) < 0.9 (d-1)$. Claim~\ref{cl: proper subgraph} implies that $d(G_1[U])\leq d$. 
Let $H := G_1[B,U]$.
Then, since $|U|= 10|G_1|/L$ and $1/L \ll 1$,
\begin{align}\label{eq: eH lower bound}
2e(H) &= 2(e(G_1)- e(G_1[U]) - e(G_1[B])) \nonumber \\ &\geq (d-1)|G_1| - d|U| - 0.9 (d-1)(|G_1|-|U|)\nonumber\\ & \geq \frac{(d-1)|G_1|}{10} -\frac{10d|G_1|}{L} \geq \frac{d|G_1|}{20}.
\end{align}
Thus $d(H) \geq d/20$. Let $t: =|G_1|/|U|= L/10$. 

If $H$ contains a subgraph $H'$ with $d(H')\geq 3d$, then Theorem~\ref{Tuza} implies that \begin{align}
\nonumber c(G)&\geq c(H') \geq 2^{3d/2},
\end{align}
a contradiction.
Thus we may thus assume that every subgraph of $H$ has average degree at most $3d$. Let 
$$B':= \left\{ v\in B:  d_{H}(v) \geq \frac{d}{100}\right\}.$$ 
Then by the definition of $B'$, we have
\begin{align*}
e(G_1[U,B']) = e(H) - e(G_1[U,B\setminus B'])\stackrel{\eqref{eq: eH lower bound}}{\geq} \frac{d|G_1|}{40} - \frac{d( |G_1| - |B'|)}{100} \ge \frac{d|G_1|}{100} = \frac{dt|U|}{100}.
\end{align*}
Since every subgraph of $H$ has average degree at most $3d$, we have
$$\frac{ dt|U|}{50(|U|+|B'|)} \leq \frac{2 e(G_1[U,B'])}{|U|+|B'|}\leq d(H[U,B']) \leq 3d.$$
Thus by the fact that $t= L/10$ and $1/L \ll 1$,
\begin{align}\label{eq: B'}
|B'| \geq \frac{(t-150)|U|}{150} \geq \frac{t|U|}{200}.
\end{align}

Let $B''\subseteq B'$ be a subset of $B'$ with 
$$|B''| = t^{1/2}|U|=\frac{|G_1|}{t^{1/2}}\ge \frac{K^2d}{(L/10)^{1/2}}> Kd,$$ 
where we used Claim~\ref{cl: G1-2conn}(iv) and that $1/K\ll 1/L$.
Now, for each such $B''$, we define a cycle $C_{B''}$.
Let $H':= H[B'',U]$.
Then again, since $d_H(v) \geq d/100$ for all $v \in B''$,
$$d(H') = \frac{2e(H')}{|H'|} \geq \frac{ d|B''|/100 }{|B''|+|U|} \geq \frac{d|B''|/100}{2|B''|}\geq \frac{d}{200}.$$
Thus by Theorem~\ref{lem: cycle av deg}, $H'$ contains a cycle of length at least $d/200$.
Let $C_{B''}$ be any such cycle, then we have $(V(C_{B''})\setminus U) \subseteq B''$  and $|V(C_{B''})\cap B'|\geq d/400$.

For a cycle $C$ in $H$, $V(C)=V(C_{B''})$ holds for at most 
$$\binom{|B'|-|V(C)\cap B'|}{ |B''| - |V(C)\cap B'| } \leq \binom{|B'|-d/400}{t^{1/2}|U| - d/400}$$ distinct subsets $B'' \subseteq B'$ of size $t^{1/2}|U|$. Thus, we get
\begin{align*}
c(G) &\geq c(H)\geq \sum_{|B''|=t^{1/2}|U|}  \binom{|B'|-d/400}{t^{1/2}|U| - d/400}^{-1} = \binom{|B'|}{t^{1/2}|U|}\binom{|B'|-d/400}{t^{1/2}|U|- d/400}^{-1} 
 \\ &= \frac{|B'|(|B'|-1)\dots(|B'|-d/400+1)}{t^{1/2}|U|(t^{1/2}|U|-1)\dots(t^{1/2}|U|-d/400+1) } \stackrel{\eqref{eq: B'}}{\geq} \left(\frac{t^{1/2}}{400}\right)^{\frac{d}{400}} \geq \left(\frac{L^{1/2}}{4000}\right)^{\frac{d}{400}} \geq 2^{2d},
\end{align*}
a contradiction. Note that we get the final inequality since $1/L\ll 1$. 
Thus $d(G_1[B]) \geq 0.9(d-1)$, as required.
\end{claimproof}

\medskip
\noindent
Define
$$U:= \{v \in G_1: d_{G_1}(v) \geq Ld\}.$$
Then 
\begin{align}\label{eq: U size large degree}
|U| \leq 5|G_1|/L
,\end{align} otherwise 
$$d(G_1) \geq \frac{ Ld  |U|}{|G_1|} \geq 5d,$$
contradicting Claim~\ref{cl: G1-2conn}(ii). Clearly, $\Delta(G_1 - U) \leq L d$, and by Claim~\ref{cl: vertex deletion fine},
\begin{equation}\label{d1}
d_1 := d(G_1 - U) \ge 0.9(d-1) \geq 0.89d.
\end{equation}

Observe that $\eps_1 \leq \frac{1}{130}$ and
$$
0.89\left(1-\frac{13\eps_1}{\log 3}\right) \stackrel{(\ref{hierarchy})}{>} 0.8 \quad\text{ and }\quad \frac{\eps_1 d_1}{6\log^2(5/(\frac{1}{30}))} \stackrel{(\ref{d1})}{\geq} \frac{0.89\eps_1d_0}{6\log^2(150)} \stackrel{(\ref{hierarchy})}{\geq} 2.
$$
Thus we can apply Lemma~\ref{lem-expander} to $G_1 - U$ with $13,\eps_1,1/30$ and $d_1$ playing the roles of $C,\eps_1,c'$ and $d$ respectively to find a $2$-connected $(\ep_1,d_1/30)$-expander $H_1\subseteq G_1 - U$ such that
\begin{equation*}
d(H_1)\geq 0.8d;\quad \delta(H_1)\geq 0.4d; \quad \text{ and }\quad \Delta(H_1)\leq Ld.
\end{equation*}
Since $1/d, 1/K \ll \epsilon_1, 1/L\leq 1$ from \eqref{hierarchy}, if $|H_1| \geq Kd$, we can apply Lemma~\ref{lem-sparse} with $d,\eps_1$ and $L$ playing the roles of $d,\eps_1$ and $L$ respectively to see that
\begin{eqnarray}\label{eq-small-expander}
c(G)\geq c(H_1) \geq 2^{50d},
\end{eqnarray} 
a contradiction. Thus we have $|H_1|\leq Kd$ and consequently \eqref{hierarchy} and Claim~\ref{cl: G1-2conn}(iv) imply that
$$|U\cup V(H_1)|\stackrel{ \eqref{eq: U size large degree}}{\leq} \frac{5|G_1|}{L} + Kd \leq (\frac{5}{L}+ \frac{1}{K})|G_1| \leq \frac{10|G_1|}{L}.$$ 
Thus Claim~\ref{cl: vertex deletion fine} applied with $U \cup V(H_1)$ playing the role of $U$ implies that $d_2 := d(G_1 -U- H_1 )\geq 0.9(d-1)\geq 0.89d$.
Apply Lemma~\ref{lem-expander} to $G_1-U-H_1$ with $13,\eps_1,1/30$ and $d_2$ playing the roles of $C,\eps_1,c'$ and $d$ respectively to see, as before, that $G_1 - U - H_1$ contains a $2$-connected $(\ep_1,d_2/30)$-expander $H_2$ such that 
\begin{equation*}
d(H_2)\geq 0.8d;\quad \delta(H_2) \geq 0.4d; \quad \text{ and }\quad \Delta(H_2)\leq Ld.
\end{equation*}
In a similar way as we obtained \eqref{eq-small-expander}, we may assume $|H_2|\leq Kd$.

We have shown that $G_1$ contains two vertex-disjoint expanders $H_1,H_2$.
Recall that $G_1$ is $2$-connected by Claim~\ref{cl: G1-2conn}(i). Thus we can choose two minimal vertex-disjoint paths $P_1$ and $P_2$ from $V(H_1)$ to $V(H_2)$ in $G_1$. Assume that $P_1$ is from $x_1 \in V(H_1)$ to $x_2 \in V(H_2)$ and $P_2$ is from $y_1 \in V(H_1)$ to $y_2\in V(H_2)$. 

Then for any path $P$ in $H_1$ from $x_1$ to $y_1$ and any path $P'$ in $H_2$ from $x_2$ to $y_2$, we obtain a cycle $C(P,P')$ with edge-set $E(P)\cup E(P')\cup E(P_1)\cup E(P_2)$. 
That is, $$V(P) \cup V(P') \cup V(P_1) \cup V(P_2)$$ forms a Hamiltonian subset.
Thus $c(G) \geq p_{x_1y_1}(H_1)\cdot p_{x_2y_2}(H_2)$. Applying Lemmas~\ref{lem: small graph} and~\ref{lem: 1.2 to K} to $H_i$, for $i\in [2]$, and with $0.8d$ playing the role of $d$, we get
$$c(G)\geq  p_{x_1y_1}(H_1)\cdot p_{x_2y_2}(H_2) \geq 2^{0.89\cdot 0.8d} \cdot 2^{0.89\cdot 0.8d} \geq 2^{1.4d},$$
a contradiction. Thus the counterexample does not exist for any $d \geq d_0$ and we have shown that, in this range, $d(G) \geq (2-\alpha)2^{d+1}$ as long as $d(G)\geq d$ and $G\notin \{K_{d+1}, K_{d+1}\ast K_{d}\}$. 
This completes the proof of Theorem~\ref{main}.
\hfill$\square$

\section{Concluding remarks}\label{conclusion}

In this paper, we proved Conjecture~\ref{komlos} for all large $d$.
The most obvious remaining open question is to extend this to all $d$.
As we rely on the regularity lemma, we cannot hope to apply our techniques here.

%

Our proof of Koml\'os's conjecture can also be adapted to prove Theorem~\ref{bip}, the bipartite version of Conjecture~\ref{komlos} asked by Tuza~\cite{tuza3}. 
We only sketch the proof here since it is very similar to the proof of Theorem~\ref{main}.
Indeed, proceed exactly as in the proof of Theorem~\ref{main} in Section~\ref{end}, with Lemmas~\ref{lem: small graph} and~\ref{lem: 1.2 to K} replaced by Lemma~\ref{lem: small graph}$^\prime$ and~\ref{lem: 1.2 to K}$^\prime$ stated below.
(In fact much of this, though true, is redundant here since we are now assuming that $\delta(G) \geq d$.) 

Let $a \leq b$ be positive integers. Every cycle in $K_{a,b}$ contains exactly $i$ vertices from each class for some $2 \leq i \leq a$. Thus
$$
c(K_{a,b}) = \sum_{i=2}^{a} \binom{a}{i}\binom{b}{i} = \sum_{i=0}^{a} \binom{a}{a-i}\binom{b}{i} - \binom{a}{a}\cdot 1 - \binom{a}{a-1}\binom{b}{1} = \binom{a+b}{a} - (ab+1)
$$
and in particular
$$
c(K_{d,d}) = \binom{2d}{d} - (d^2+1) \sim \frac{2^{2d}}{\sqrt{\pi d}}.
$$
We can use Lemma~\ref{lem-sparse} in its present form since it guarantees $2^{50d}$ Hamiltonian subsets, which is still sufficient for Theorem~\ref{bip}.

\medskip
\noindent
\textbf{Lemma~\ref{lem: small graph}$^\prime$.}
\emph{
Let $\alpha > 0$.
Then there exists $d_0 > 0$ such that the following holds for all $d \geq d_0$.
Let $n \in \mathbb{N}$ be such that $2d\leq n\leq 2.2 d$. 
Let $G$ be an $n$-vertex bipartite graph with $\delta(G)\geq d$.
Then
\begin{itemize}
\item[(i)] $c(G)\geq (1-\alpha)\binom{n}{d} \geq (1-\alpha)c(K_{d,n-d})$;
\item[(ii)] if $x,y \in V(G)$ are distinct, then $p_{xy}(G)\geq (1-\alpha)\binom{n-2}{d-1}$.
\end{itemize} 
}

\medskip
\noindent
Observe that $c(K_{d,n-d})$ is at least almost twice as large as $c(K_{d,d})$ when $n > 2d$, and $p_{x'y'}(K_{d,n-d}) \geq \binom{n-2}{d-1}$ for any distinct vertices $x',y' \in V(K_{d,n-d})$.

\medskip
\noindent
\textbf{Lemma~\ref{lem: 1.2 to K}$^\prime$.}
\emph{
For all $K>0$, there exists $d_0 = d_0(K)$ such that the following holds for all $d \geq d_0$.
Let $G$ be a $2$-connected bipartite graph on $n$ vertices where $2.19 d \leq n \leq Kd$ and $\delta_2(G)\geq d$.
Let $x,y$ be two distinct vertices of $G$. Then
\begin{equation}
c(G)> 2^{(2+ \frac{1}{200})d}\ \text{ and }\ p_{xy}(G) > 2^{1.89d}.
\end{equation}
}

\noindent
The proof of Lemma~\ref{lem: small graph}$^\prime$ is very similar to its original counterpart.
Any graph satisfying the hypotheses is an almost balanced almost complete bipartite graph and thus a similar probabilistic argument shows that almost every balanced subset is Hamiltonian.
(A version of Dirac's theorem states that any balanced bipartite graph with parts of order $k$ and minimum degree at least $(k+1)/2$ is Hamiltonian.)
To prove Lemma~\ref{lem: 1.2 to K}$^\prime$, one must simply replace Lemma~\ref{lem: sun} with the following statement.

\medskip
\noindent
\textbf{Lemma~\ref{lem: sun}$^\prime$.}
\emph{
Let $n \geq 2.18d$ and let $G$ be an $n$-vertex bipartite graph with $\delta(G) \geq d$.
Then $G$ contains at least one of the following:
\begin{itemize}
\item[(i)] two vertex-disjoint cycles $C_1,C_2$ with $|C_1| + |C_2| \geq 3.8d$;
\item[(ii)] a path $P$ with $|P| \geq (2+1/100)d$;
\item[(iii)] an $(a,b)$-sun with $a \geq 2d$ and $b \geq d/20$.
\end{itemize}
}

\medskip
\noindent
Again, the proof is very similar.


\medskip
\noindent
In a graph $G$, say that $U \subseteq V(G)$ is a \emph{weak Hamiltonian subset} of $G$ if $G[U]$ contains a spanning cycle; or $U = \lbrace x,y \rbrace$ and $xy \in E(G)$; or $|U| \in \lbrace 0,1 \rbrace$.
Observe that every subset of the vertices of a complete graph is weak Hamiltonian.
Theorem~\ref{main} immediately implies that
there exists $d_0 > 0$ such that for all integers $d \geq d_0$, every graph $G$ with average degree at least $d$ contains at least as many weak Hamiltonian subsets as $K_{d+1}$ (that is, $2^{d+1}$).
This proves another conjecture of Tuza~\cite{tuza3}, stated in the quest for `nicer formulas' than~(\ref{complete}).

\section{Acknowledgements}

We are indebted to the referee for her/his careful reading and valuable comments.

\medskip

{\footnotesize \obeylines \parindent=0pt
	
	\begin{tabular}{lllll}
        Jaehoon Kim	&\ & 	Hong Liu, Maryam Sharifzadeh and Katherine Staden         \\
		School of Mathematics &\ & Mathematics Institute		  		 	 \\
		University of Birmingham &\ & University of Warwick 	  			 	 \\
		Birmingham &\ & Coventry                             			 \\
        B15 2TT      &\ & 		CV4 7AL				  				\\
		UK &\ & UK						      \\
	\end{tabular}
}

\begin{flushleft}
	{\it{E-mail addresses}:
		 \tt{j.kim.3@bham.ac.uk, $\lbrace$h.liu.9,~m.sharifzadeh,~k.l.staden$\rbrace$@warwick.ac.uk.}}
\end{flushleft}

\appendix 

\section{}

Here we give the details of the omitted proofs of Proposition~\ref{homomorphism} and Lemma~\ref{blowuppath}.
\begin{proof}[Proof of Proposition~\ref{homomorphism}]
	It is convenient to prove a general claim from which (i)--(iii) will follow.
	Let $G$ be a graph with $V(G) = \lbrace z_1,\ldots,z_\ell \rbrace$.
	Given $k \in \mathbb{N}$, let $(z_1,\dots, z_\ell)^k$ denote the sequence 
	$$
	(z_1,\dots, z_\ell,z_1,\ldots,z_\ell,\ldots,z_1,\dots, z_\ell)
	$$
	of length $k\ell$.
	Consider the following claim.
	
	\begin{claim}\label{subclaim}
		Let $N \in \mathbb{N}$ and let $R := (y_1,\ldots,y_\ell)$ be a circuit in $G$ with $V(G) = \lbrace y_1,\ldots,y_\ell \rbrace$. Let $x,x' \in V(G)$ be two not necessarily distinct vertices.
		Then there is a walk $W$ in $G$ from $x$ to $x'$ and $N \cdot{\rm deg}(y,R) \leq {\rm deg}(y,W) \leq  (N+1)\cdot{\rm deg}(y,R)$ for all $y\in V(G)$.
	\end{claim}
	
	\begin{claimproof}
		Assume, without loss of generality, that $x$ appears no later than $x'$ in $R$.
		Then we can choose $1 \leq s \leq t \leq \ell$ such that $y_s = x$ and $y_t = x'$.
		Define the concatenation
		$$
		W := \left\{\begin{array}{ll}
		(y_t, \ldots, y_{\ell})(y_1,\ldots,y_{\ell})^{N}(y_1,\ldots,y_s) & \text{if } s<t \\
		(y_t, \ldots, y_{\ell})(y_1,\ldots,y_{\ell})^{N-1}(y_1,\ldots,y_s) & \text{if } s=t,
		\end{array}\right.
		$$
		where $(y_1,\ldots,y_s)$ and $(y_t, \ldots, y_{\ell})$ are subwalks of $R$. It is clear by the definition of $R$ that $W$ is a walk in $G$.
		Then clearly we have
		$$N \cdot {\rm deg}(y,R) \leq {\rm deg}(y,W) \leq  (N+1)\cdot{\rm deg}(y,R),$$
		as required.
	\end{claimproof}
	
	\medskip
	\noindent
	To prove (i), write $C = x_1\ldots x_a$ and define a circuit
	$$
	R := (x_1,\ldots,x_a).
	$$
	Then ${\rm deg}(x,R)=1$ for all $x \in V(C)$ and we are done by applying Claim~\ref{subclaim} with $2n$ playing the role of $N$.

	To prove (ii),
	write $P = x_1\ldots x_a$ and let 
	$$
	R := (x_1,\dots, x_a, x_{a-1},\dots, x_{2})
	$$
	be a circuit of length $2a-1$.
	Clearly ${\rm deg}(x,R)=1$ for $x \in \lbrace x_1,x_a \rbrace$ and ${\rm deg}(x,R)=2$ for $x \in V(P) \setminus \lbrace x_1,x_a \rbrace$.
	We are done by applying Claim~\ref{subclaim} with $n$ playing the role of $N$.

	For (iii), write $V(S) := \lbrace x_1,\ldots,x_a \rbrace \cup \lbrace y_{i_1},\ldots,y_{i_b} \rbrace$ where $\Cor(S) = \lbrace x_{i_k},y_{i_k} : k \in [b] \rbrace$.
	Define $R$ to be the circuit in $S$ obtained by concatenating the two cycles, one with $x_{i_j}$'s and the other with $y_{i_j}$'s for $j\in [b]$: that is,
	$$
	R := (x_{i_1},x_{i_1+1},\dots, x_a, x_1,\dots, x_{i_1-1}, y_{i_1}, x_{i_1+1},\dots, x_{i_b-1}, y_{i_{b}},x_{i_{b}+1},\dots, x_a,x_1,\ldots,x_{i_1-1}).
	$$
	It is easy to see that, for $x \in \Cor(S)$ we have ${\rm deg}(x,R)=1$ and for $x \in V(S)\setminus \Cor(S)$ we have ${\rm deg}(x,R)=2$.
	We are immediately done by applying Claim~\ref{subclaim} with $n$ playing the role of $N$.
\end{proof}

\medskip
We now turn to the proof of Lemma~\ref{blowuppath}. Given a graph $R$ with $V(R)=[r]$ and an integer $\ell \leq r$, we say that a graph $R'$ is the \emph{$(\ell,2)$-blow-up} of $R$ if $V(R') = [r+\ell]$ and 
$$
E(R') = E(R) \cup \bigcup_{i \in [\ell]}\left(\left\lbrace \lbrace i+r,x\rbrace: x \in N_R(i) \right\rbrace \cup \left\lbrace \lbrace i+r,j+r\rbrace : j \in N_{R}(i)\cap [\ell] \right\rbrace\right).
$$
That is, $R'$ is obtained from $R$ by duplicating vertices $[\ell] \subseteq V(R)$.

\begin{lemma}\label{superregularpartition}
	Suppose $0 < 1/n \ll \epsilon \ll \gamma, 1/\Delta_R \leq 1$ and $1/n \ll 1/r$ and $0\leq \ell \leq r$ with $r,n,\ell \in \mathbb{N}$.
	Let $R$ be a graph with $V(R) = [r]$ and $\Delta(R)\leq \Delta_R$.
	Let $G$ be an $(\epsilon,\gamma)$-super-regular graph with respect to a vertex partition $(R,V_1,\dots, V_r)$ such that
	$$
	2n\leq |V_i|\leq 2n+2 \text{ for } i \in [\ell] \quad\text{ and }\quad n\leq |V_i| \leq n+2 \text{ for } \ell+1 \leq i \leq r.
	$$
	Let $R'$ be the $(\ell,2)$-blow-up of $R$.
	For $i \in [r]\setminus [\ell]$, let $V'_i:= V_i$.
	Then for all $i\in [\ell]$, there exists a partition $V_i = V'_{i}\cup V'_{r+i}$ such that $n\leq |V'_{i}|, |V'_{r+i}| \leq n+2$ and $G$ is $(4\epsilon,\gamma/2)$-super-regular with respect to vertex partition $(R',V'_1,\dots,V'_{r+\ell})$.
\end{lemma}

The proof of this lemma is a standard application of slicing lemma and Chernoff bounds. We provide only a sketch here: For each $i\in [\ell]$, we take a random partition $V'_{i}\cup V'_{r+i}$ of $V_i$ such that $|V'_{i}| = n$. Then slicing lemma ensures that $G[V'_i,V'_j]$ is $(4\epsilon,\gamma/2)$-regular for each $ij\in E(R')$.
For a vertex $v\in V_i$ and $j\in N_{R'}(i)$, Chernoff bounds give us that 
$$\mathbb{P}[|d_{G}(v)\cap V'_j| \geq \gamma n/2]\geq 1 - e^{-\epsilon n^2}.$$
Thus, union bounds over all vertices $v$ and all $i,j\in [r+\ell]$ show that the conclusion holds with high probability.

\begin{proof}[Proof of Lemma~\ref{blowuppath}]
	Let 
	$$N_1  := \{i\in [r]: n\leq n_i\leq n+2 \} \enspace \text{ and } \enspace N_2 := \{i\in [r]: 2n\leq n_i\leq 2n+2\}.$$ 
	So $[r] = N_1 \cup N_2$ is a partition.
	Without loss of generality, we may assume that $N_2  = \{1,\dots, \ell \}$ for some $0\leq \ell \leq r$. 
	Let $R'$ be the $(\ell,2)$-blow-up of $R$.
	Apply Lemma~\ref{superregularpartition} to find a new partition $V'_1,\dots, V'_{r+\ell}$ of $G$ such that $G$ is $(4\epsilon,\gamma/2)$-super-regular with respect to vertex partition $(R',V'_1,\dots, V'_{r+\ell})$ such that $|V'_i| \in \{n,n+1,n+2\}$ for each $i\in [r+\ell]$.
	Let $W' := (w_1',\ldots,w'_m)$ be a walk in $R'$ obtained as follows. For all $t \in [m]$, let
	$w_t' := w_t$ if $w_t \in N_1$,
	and for each $w_t'$ such that $w_t \in N_2$, we choose $w_t' \in \lbrace w_t,w_t+r \rbrace$ so that $\mathrm{deg}(i,W') = |V'_i|$ and $\mathrm{deg}(i+r,W')=|V'_{i+r}|$ for all $i \in [\ell]$. This is possible by (iv) and fact that $n_i = |V'_i|+|V'_{i+r}|$. Thus $\mathrm{deg}(i,W')=|V_i'|$ for all $i \in [r+\ell]$.
	
	Let $Q = (q_1,\dots, q_m)$ be a path of length $m$. For each $i\in [r+\ell]$, let $X_i:= \{ q_j: w'_j = i\}$.
	Since $W'$ is a walk in $R'$ with $\mathrm{deg}(i,W')=|V_i'|=|X_i|$, $Q$ admits a vertex partition $(R',X_1,\dots, X_{r+\ell})$. Let $G' := G-\lbrace x,y \rbrace$.
	Then $G'$ is $(5\eps,\gamma/3)$-super-regular with respect to vertex partition $(R',V''_1,\dots, V''_{r+\ell})$, where $V''_i := V'_i \setminus \lbrace x,y \rbrace$ for all $i \in [r+\ell]$.
	Let $S_{q_2} := N_{G'}(x,V''_{w'_2})$ and $S_{q_{m-1}} := N_{G'}(y,V''_{w'_{m-1}})$.
	Now $w'_1 w'_2 \in E(R')$ by definition.
	Since $G'[V''_{w'_1},V''_{w'_2}]$ is $(5\eps,\gamma/3)$-super-regular, we have that $|S_{q_2}| \geq \gamma|V''_{w'_2}|/3$.
	Similarly $|S_{q_{m-1}}| \geq \gamma|V''_{w'_{m-1}}|/3$.
	
	Let $Q^*$ denote the truncated path $(q_2,\dots, q_{m-1})$. Theorem~\ref{blowup} implies that there is an embedding of $Q^*$ into $G'$ such that $q_2$ is mapped to a vertex in $S_{q_2}$ and $q_{m-1}$ is mapped to a vertex in $S_{q_{m-1}}$ and $V(Q^*)=V(G')$.
	By the choice of $S_{q_2}$ and $S_{q_{m-1}}$, this embedding can be extended to $x$ and $y$ to obtain a path $P$ which spans $V(G)$ and has endpoints $x$ and $y$.
\end{proof}


\begin{thebibliography}{99}

\bibitem{ahrens}
W. Ahrens,
\"{U}ber das Gleichungssystem einer Kirchhoffschen galvanischen Stromverzweigung,
\emph{Math. Ann.} {\bf 49} (1897), 311--324. 

\bibitem{agt}
A.~Arman, D.~S.~Gunderson and S.~Tsaturian,
Triangle-free graphs with the maximum number of cycles,
\emph{Discrete Math.} {\bf 339} (2) (2016), 699--711.

\bibitem{at}
R.~E.~L.~Aldred and C.~Thomassen,
On the number of cycles in $3$-connected cubic graphs,
\emph{J. Combin. Theory B} {\bf 71} (1997), 79--84.



\bibitem{BLSh}
J.~Balogh, H.~Liu and M.~Sharifzadeh,
\newblock Subdivisions of a large clique in $C_6$-free graphs.
\newblock {\em  J. Combin. Theory Ser. B}, 112, (2015), 18--35.

\bibitem{bce}
C.~A.~Barefoot, L.~Clarke and R.~Entringer,
Cubic graphs with the minimum number of cycles,
\emph{Congr. Numer.} {\bf 53} (1986), 49--62.

\bibitem{Csaba} B.~Csaba, 
\newblock On the Bollob\'as-Eldridge conjecture for bipartite graphs,
\newblock {\em Combin. Probab. Comput.} {\bf 16} (2007), 661--691.

\bibitem{nash-w}
B.~Csaba, D.~K\"uhn, A.~Lo, D.~Osthus, A.~Treglown,
\newblock Proof of the 1-factorization and Hamilton decomposition conjectures,
\newblock {\em  Mem. Amer. Math. Soc.}, to appear.

\bibitem{ck1}
B.~Cuckler and J.~Kahn,
Hamiltonian cycles in Dirac graphs,
\emph{Combinatorica} {\bf 29} (2009), 299--326.

\bibitem{dirac}
G.~A.~Dirac,
Some theorems on abstract graphs,
\emph{Proc. London Math. Soc.} {\bf 2} (1952), 69--81.

\bibitem{erdgal} P.~Erd\H{o}s and T.~Gallai, On maximal paths and circuits of graphs, \emph{Acta Math. Acad. Sci. Hunar.} {\bf 10} (1959), 337--356.

\bibitem{JLR} S. Janson, T. \L uczak and A. Ruci\'nski, \emph{Random graphs},
Wiley-Interscience, 2000.

\bibitem{karp}
R.~Karp,
\newblock Reducibility among combinatorial problems,
\newblock {\em Complexity of computer computations}, Plenum, New York, (1972), 85--103.


\bibitem{KKOT16}
J.~Kim, D.~K{\"u}hn, D.~Osthus, M.~Tyomkyn, 
\newblock \emph{A blow-up lemma for
  approximate decompositions}, 
  \newblock {\em submitted}, arXiv:1604.07282 (2016).


\bibitem{krivsam}
M.~Krivelevich and W.~Samotij,
Optimal packings of Hamilton cycles in sparse random graphs,
\emph{SIAM J. Disc. Math.}
{\bf 26} (3) (2012), 964--982.

\bibitem{knor} M.~Knor, On the number of cycles in $k$-connected graphs,
\emph{Acta Math. Univ. Comen.}
{\bf 63} (2), (1994), 315--321.

\bibitem{kko} F.~Knox, D.~K\"uhn and D.~Osthus,
Edge-disjoint Hamilton cycles in random graphs,
\emph{Random Struct. Alg.}
{\bf 46} (2015), 397--445.

\bibitem{KSSblowup} J.~Koml\'{o}s, G.N.~S\'{a}rk\"ozy and E.~Szemer\'{e}di, Blow-up lemma, {\em Combinatorica} {\bf 17} (1997), 109--123.

\bibitem{ksssurvey}
J.~Koml\'os, A.~Shokoufandeh, M.~Simonovits and E.~Szemer\'edi,
The regularity lemma and its applications in graph theory,
\emph{Theoretical aspects of computer science} (Tehran, 2000), 84--112, Lecture Notes in Comput. Sci., 2292, Springer, Berlin, 2002. 

\bibitem{kssurvey}
J.~Koml\'os and M.~Simonovits,
Szemer\'edi's regularity lemma and its applications in graph theory,
\emph{Combinatorics: Paul Erd\H{o}s is eighty, Vol. 2} (Keszthely, 1993), 295--352, Bolyai Soc. Math. Stud., 2, J\'anos Bolyai Math. Soc., Budapest, 1996.

\bibitem{K-Sz-1}
J.~Koml{\'o}s, E.~Szemer{\'e}di,
\newblock Topological cliques in graphs.
\newblock {\em Combin. Probab. Comput.}, 3, (1994), 247--256.

\bibitem{K-Sz-2}
J.~Koml{\'o}s, E.~Szemer{\'e}di,
\newblock Topological cliques in graphs {II}.
\newblock {\em Combin. Probab. Comput.}, 5, (1996), 79--90.

\bibitem{kelly}
D.~K\"uhn, D.~Osthus,
\newblock Hamilton decompositions of regular expanders: a proof of Kelly's conjecture for large tournaments,
\newblock {\em  Adv. Math.}, 237, (2013), 62--146.

\bibitem{apps}
D.~K\"uhn, D.~Osthus,
Hamilton decompositions of regular expanders: applications,
\emph{J. Comb. Theory B}
{\bf 104} (1) (2014), 1--27.

\bibitem{Ham-survey}
D.~K\"uhn, D.~Osthus,
\newblock Hamilton cycles in graphs and hypergraphs: an extremal perspective,
\newblock {\em  Proceedings of the International Congress of Mathematicians}, (2014), Seoul, Korea, Vol 4, 381--406.

\bibitem{maxplanar} D.~K\"uhn, D.~Osthus and A.~Taraz,
Large planar subgraphs in dense graphs,
\emph{J. Combin. Theory B} {\bf 95} (2005), 263--281.

\bibitem{tuzapaper} H.~Liu, M.~Sharifzadeh and K.~Staden, Local conditions for exponentially many subdivisions, \emph{Combin. Probab. Comput.}, to appear.

\bibitem{L-M}
H.~Liu, R.H.~Montgomery,
\newblock A proof of Mader's conjecture on large clique subdivisions in $C_4$-free graphs.
\newblock {\em J. London Math. Soc.}, to appear.

\bibitem{Richard}
R.H.~Montgomery,
\newblock Logarithmically small minors and topological minors.
\newblock {\em  J. Lond. Math. Soc.}, {\bf 91} (2), (2015), 71--88.

\bibitem{tashscott}
N.~Morrison and A.~Scott,
Maximising the number of induced cycles in a graph,
submitted, arXiv:1603.02960 (2016).

\bibitem{will} W.~Perkins, personal communication.

\bibitem{posathm} L.~P\'osa, A theorem concerning Hamiltonian lines, 
\emph{Magyar Tud. Akad. Mat. Fiz. Oszt. Kozl.} {\bf 7} (1962), 225--226.

\bibitem{szem} E.~Szemer\'edi,
Regular partitions of graphs,
\emph{Probl\'emes Combinatoires et Th\'eorie des Graphes Colloques Internationaux CNRS} {\bf 260} (1978), 399--401.

\bibitem{tuza2} Zs.~Tuza,
Exponentially many distinguishable cycles in graphs,
\emph{J. Combin. Infor. Sys. Sci.} {\bf 15} (1--4), (1990), 281--285.

\bibitem{tuza1} Zs.~Tuza,
Unsolved combinatorial problems, Part I,
\emph{BRICS Lecture Series, LS-01-1}, (2001).

\bibitem{tuza3} Zs.~Tuza, Problems on cycles and colorings, \emph{Disc. Math.} {\bf 313} (2013), 2007--2013.

\bibitem{volkmann} L. Volkmann,
Estimations for the Number of Cycles in a Graph,
\emph{Per. Math. Hungar.} {\bf 33}, (1996), 153--161. 

\end{thebibliography}
\end{document}